\documentclass{article}

\usepackage{format}

\title{Rate-Optimal Regret for the Safe Learning-based Control of the Constrained Linear Quadratic Regulator\thanks{The authors are with the Department of Electrical and Computer Engineering, University of California Santa Barbara. This work was supported by NSF grant \#2330154. (email: shutchinson@ucsb.edu)}}
\author{Spencer Hutchinson \and Nanfei Jiang \and Mahnoosh Alizadeh}

\date{}
\begin{document}

\maketitle

\begin{abstract}
    We study the problem of adaptive control of the stochastic linear quadratic regulator (LQR) with constraints that must be satisfied at every time step.
    Prior work on the multidimensional problem has shown $\Octil(T^{2/3})$ regret and satisfaction of \emph{robust constraints}, leaving open the question of whether $\Octil(\sqrt{T})$ regret can be attained in the constrained LQR setting.
    We contribute to this problem by showing $\Octil(\sqrt{T})$ regret and satisfaction of \emph{chance constraints}.
    This type of constraints allow us to handle unbounded noise and also enable analytical techniques not directly applicable to robust constraints.
    Our proposed algorithm for this problem uses an SDP to select an optimistic policy, and then ``scales back'' this policy until it is verifiably-safe.
    Our theoretical analysis establishes regret and constraint guarantees via a key lemma that bounds the system covariance in terms of the chosen policy.
    This covariance-based analysis is in contrast with the cost-to-go based analysis that is typically used in adaptive LQR.
\end{abstract}


\section{Introduction}

We study stochastic linear-quadratic (LQ) control, in which the state vector $x_t$ evolves according to a linear system with input $u_t$ and stochastic disturbance $w_t$,
$$
x_{t+1} = A x_t + B u_t + w_t,
$$ and the performance is measured with a quadratic cost function $x_t^\top Q x_t + u_t^\top R u_t$.
In particular, we consider the \emph{adaptive control} setting, in which the system matrices $A, B$ are unknown and have to be learned during the control of the system, via observation of the state $x_t$. 
This fundamental problem has a long history in the control literature \citep{aastrom1973self,lai1986extended,lu2025almost} and has recently emerged as an important benchmark for reinforcement learning \citep{abbasi2011regret,dean2018regret}, with recent works giving algorithms and accompanying analysis that establish \emph{regret bounds} on their performance.
Notably, these regret bounds differ from the classic asymptotic results, in that they guarantee a certain level of suboptimality in the algorithm's cost over a \emph{finite} number of time steps. 
This line of research has culminated in efficient algorithms with matching upper and lower regret bounds of $\Octil(\sqrt{T})$ over $T$ time steps, hence establishing the optimality of this regret rate \citep{abbasi2011regret,dean2018regret, mania2019certainty,simchowitz2020naive,abeille2020efficient,ziemann2024regret}.

However, this standard LQ adaptive control setting does not readily allow for enforcing constraints on the state and input at every time step, which is often desirable in real-world systems.
For example, modern robotic systems are often operating near humans, and therefore must ensure that unsafe modes of operation are avoided \citep{brunke2022safe}.
Motivated by such real-world applications, several recent works have developed algorithms for LQ adaptive control that ensure that the state and input satisfy constraints at every time step \citep{dean2019safely,li2021safe}.
In particular, these works consider linear \emph{robust constraints} in which the disturbances $w_t$ belong to a bounded set $\Wc$, and then enforces state and input constraints for any realization of the disturbance sequence, i.e. $\alpha^\top x_t \leq \beta, \ \forall (w_\tau)_{\tau=1}^{t-1} \in \Wc$.
These existing works tackled this problem through the algorithmic frameworks of system-level synthesis \citep{dean2019safely} and disturbance action policy \citep{li2021safe}, ultimately establishing regret guarantees as strong as $\Octil(T^{2/3})$.\footnote{This regret bound is for the multidimensional setting. Stronger regret guarantees have been shown for the simpler $1$-dimensional setting. See Section \ref{sec:rel:const_adap}.}
Thus, the existing literature has left open the question of whether $\Octil(\sqrt{T})$ regret is possible in the constrained setting as is the case for the standard unconstrained setting.

We contribute to this gap in the literature by making a sharp departure from existing work in both problem formulation and algorithmic approach.
In particular, we show $\Octil(\sqrt{T})$ regret and constraint satisfaction for adaptive LQ control with (unbounded) Gaussian disturbances and linear \emph{chance-constraints}, which enforce constraint satisfaction with a minimum probability level, i.e. $\Pb(\alpha^\top x_t \leq \beta) \geq 1 - \delta$.
This choice of constraint enables meaningful safety guarantees even when the disturbance distribution is unbounded, in which case robust constraints cannot be satisfied. It is also analytically convenient because, as we show in the paper, these chance constraints can be expressed in terms of the system covariance.
Thus, constrained policies can be efficiently optimized with a covariance SDP, enabling us to build on the optimistic SDP-based approach for unconstrained adaptive LQ control from \cite{cohen2019learning}.
However, since \cite{cohen2019learning} is concerned with the \emph{aggregate} cost rather than \emph{per time-step} constraints, this approach alone does not provide guarantees on the system covariance at every time step as is needed for constraint satisfaction.
Thus, we develop a novel analysis framework that bounds the system covariance in terms of the chosen policy at every time step.
This \emph{covariance-based} analysis is in contrast to the \emph{cost-to-go-based} analysis typically used in adaptive LQ control \citep{abbasi2011regret,cohen2019learning,simchowitz2020naive}.
Furthermore, a purely-optimistic approach cannot natively enforce constraints, so we incorporate the ``scaled-back optimism'' technique from the stochastic bandit literature \citep{hutchinson2024directional}.
This approach combines optimism in cost with pessimistic constraint enforcement, and therefore allows for both efficiency and constraint satisfaction.

\subsection{Related Work}

Next, we discuss related work in the areas of \emph{constrained adaptive control}, \emph{covariance control} and \emph{constrained reinforcement learning}.

\paragraph{Constrained Adaptive Control}

\label{sec:rel:const_adap}

The problem of constrained adaptive control has traditionally been studied through the framework of model predictive control (MPC), which has shown to be highly successful in practice \citep{genceli1996new,aswani2013provably,mesbah2018stochastic,hewing2020learning}. 
However, the MPC literature has typically not provided regret guarantees, and instead has focused on asymptotic performance guarantees along with constraint satisfaction \citep{langson2004robust,hewing2020recursively,stamouli2022adaptive}.
In a different line of literature, several works have established finite horizon guarantees for the constrained adaptive control problem \citep{dean2019safely,li2021safe}.
In particular, \cite{dean2019safely} considered the problem of learning the controller that minimizes the expected LQ cost and satisfies robust linear constraints.
In this setting, they use the system-level synthesis framework to provide suboptimality bounds and constraint satisfaction guarantees for the controller that is learned after a finite number of time steps, while also ensuring constraint satisfaction during the learning process.
In the same setting of expected LQ cost and robust linear constraints, \cite{li2021safe} builds on the safe disturbance action policy framework from \cite{li2021online}, to give an algorithm with $\Octil(T^{2/3})$ regret with respect to the best linear controller that robustly satisfies the constraints.
Recently, \cite{schiffer2025foundations} studied the one-dimensional version of this setting (both the input and state are one-dimension) and gave $\Octil(\sqrt{T})$ regret.
Although this matches the regret lower bound, it is unclear whether their approach can be extended to the multidimensional setting.
Overall, our work differs from these prior works in that we consider chance-constraints instead of robust constraints, and that we guarantee $\Octil(\sqrt{T})$ regret in the multidimensional setting.
Furthermore, we use a covariance SDP approach as opposed to system-level synthesis \citep{dean2019safely} or disturbance-action policy \citep{li2021safe}.

\paragraph{Covariance Control}

Our proposed approach builds on the literature on covariance control, which aims to design a controller that drives system goes to a desirable covariance \citep{collins2003theory,grigoriadis1997minimum}.
Most relevantly, several works in this literature have incorporated chance-constraints in to the problem, and shown how the control sequence can be (approximately) optimized with an SDP, e.g. \citep{okamoto2018optimal,rapakoulias2023discrete}.
This idea serves as inspiration for our proposed approach.
However, different from these works, our algorithm optimizes the steady-state covariance and therefore the chance-constraint can be \emph{exactly} expressed as a linear constraint on the covariance.

\paragraph{Constrained Reinforcement Learning}

Much of the literature on constrained reinforcement learning has been focused on constrained MDPs, which enforce constraints on average over the entire horizon \citep{altman2021constrained}.
However, there is a growing body of literature that enforces constraints at every time step, e.g. \citep{berkenkamp2017safe,fisac2018general,cheng2019end,wachi2020safe,liu2021learning,yao2023constraint}.
Nonetheless, most existing regret guarantees for this setting apply to either finite state-space or the episodic setting, and are thus distinct from the infinite state-space and single trajectory setting that we consider.

\subsection{Notation and Definitions}

We use the notation $\Fc_t$ to denote the $\sigma$-algebra induced by the randomness in the first $t$ time steps. 
For matrices $X_1,...,X_m$, the notation $\diag(X_1,...,X_m)$ refers to the matrix with $X_1,...,X_m$ on the block diagonal.
We use $\Phi(x)$ to be the CDF of the standard normal, i.e. $\Phi(x) = \Pb_{X \sim \Nc(0,1)}(X \leq x)$.
Furthermore, we use $\Phi^{-1}(q)$ to be the inverse of $\Phi(x)$ (i.e. the quantile function).
Lastly, we use \emph{$(\kappa,\gamma)$-strongly stable} \citep{cohen2018online} to refer to linear policies $K$ such that there exists decomposition $A + BK = Q L Q^{-1}$ with $\| L \| \leq 1 - \gamma$ and $\| K \|, \| Q \|, \| Q^{-1} \| \leq \kappa$.

\section{Problem Setup}

We consider the adaptive control of the stochastic discrete-time linear system,
\begin{equation}
    \label{eq:dynamics}
x_{t+1} = A_\star x_{t} + B_\star u_t + w_t, \qquad w_t \sim \Nc(\bzero, W)
\end{equation}    
where $x_t \in \Rb^n$ is the state, $u_t \in \Rb^m$ is the input, and $w_t$ are i.i.d. Gaussian random disturbances with covariance matrix $W$.
In each time step $t$, a learner observes the current state $x_t$ and then chooses a control input $u_t$.
The system matrices $\Theta_\star = [A_\star\ B_\star]$ are unknown to the learner.

The goal of the learner in this setting is to minimize the cumulative cost over $T$ time steps, while satisfying constraints on the state and input at every time step.
More precisely, the cost is measured with a quadratic function of the state and input $\ell(x_t, u_t)$ where,
\begin{equation}
    \label{eq:cost_func}
    \ell(x,u) := x^\top Q x + u^\top R u,
\end{equation}    
for some $Q \succeq 0, R \succeq 0$.
At the same time, there are linear chance-constraints on the state and input that must be satisfied for all $t \in [T]$,
\begin{equation}
    \label{eq:chance_const}
    \Pb \left( \alpha_j^\top z_t \leq \beta \right) \geq 1 - \delta, \quad \forall j \in [J]
\end{equation}
where $z_t^\top = [x_t^\top \ u_t^\top]$ and $\delta \in (0,0.5)$.
We consider chance constraints on the input rather than deterministic constraints because the algorithm and regret benchmark both use linear controllers (which is standard in adaptive control, e.g. \cite{abbasi2011regret,li2021safe}), and therefore deterministic constraints on the input \emph{cannot} be satisfied given the unbounded disturbance distribution.


The performance of the learner is measured via \emph{regret},
\begin{equation}
    \label{eq:reg}
    \reg_T = \sum_{t=1}^{T} \ell(x_t, u_t) - J^\star,
\end{equation}
where $J^\star$ is the minimum expected cost attainable by stable linear controllers that satisfy the chance-constraints with full knowledge of the system, i.e.,
\begin{subequations}
\label{eq:opt_cost}
\begin{align}
    J^\star = \min_{K} & \ \Eb \sum_{t=1}^{T} \ell(x_t, u_t)\\
    \mathrm{s.t.}\ & \ x_{t+1} = A_\star x_{t} + B_\star u_t + w_t, \qquad w_t \sim \Nc(\bzero, W)\\
    & \ u_t = K x_t\\
    & \ K \ \mathrm{is}\ (\kappa,\gamma)\ \text{strongly-stable} \label{eq:opt_cost:ss}\\
    & \ \Pb ( \alpha_j^\top z_t \leq \beta ) \geq 1 - \delta, \quad \forall j \in [J] \label{eq:opt_cost:cc}
\end{align}
\end{subequations}
We emphasize that the constraints imposed on the benchmark policy in \eqref{eq:opt_cost:cc} are the \emph{exact same} as those imposed on the learner in \eqref{eq:chance_const}.

We conclude this section by introducing additional notation that will ease presentation.
In particular, we take the combined cost matrix $\mathbf{Q} := \diag(Q,R)$ and the outer product of the constraint vector $\balpha_j := \alpha_j \alpha_j^\top$.
Additionally, we use terms for the size of the disturbance $\bar{w} := \max(1,\tr(W))$, the constraint vector $D := \max(1,\| \alpha_1\|,...,\|\alpha_J\|)$, the cost matrix $R_Q := \max(1,\tr(\mathbf{Q}))$ and the initial state $R_1 := \max(1,\| x_1 \|)$.

\subsection{Assumptions}

Next, we state the assumptions that we use.
The first assumption (Assumption~\ref{ass:null_policy}) specifies that the zero policy $u_t = \bzero$ is stable and ensures that the chance-constraint is strictly satisfied.

\begin{assumption}[Zero Policy is Stable and Strictly-feasible]
    \label{ass:null_policy}
    The policy $u_t = \bzero$ is $(\kappa,\gamma)$-strongly stable and ensures that $\Pb(\alpha_j^\top z_t \leq \beta - \epsilon) \geq 1 - \delta$ for all $j \in [J], t \in [T]$, where $\epsilon \in (0,1]$.
\end{assumption}

If instead of Assumption \ref{ass:null_policy}, there exists a linear controller $K \neq \bzero$ that stabilizes the system and strictly satisfies the constraint, then the input can be augmented with this controller, i.e. $u_t = K x_t + \bar{u}_t$.
Thus, Assumption \ref{ass:null_policy} is satisfied for the augmented system with dynamics matrix $A + B K$ and input $\bar{u}_t$.
Note that the existence of a known stabilizing linear controller is a standard assumption in the learning-based adaptive control literature \citep{cohen2019learning,simchowitz2020naive}, and the additional requirement that such a controller is strictly-feasible is used in prior safe learning-based control work \citep{li2021safe}.

Next, Assumption \ref{ass:noise_cov} specifies that the covariance matrix of the disturbance is positive definite and Assumption \ref{ass:dyn} specifies that the system matrix $\Theta_\star = [A_\star\ B_\star]$ is bounded in norm by a known scalar $S$.
Assumption \ref{ass:noise_cov} and \ref{ass:dyn} are standard in the learning-based adaptive control literature \citep{cohen2019learning,plevrakis2020geometric,simchowitz2020naive}.

\begin{assumption}[Disturbance Covariance is Positive Definite]
    \label{ass:noise_cov}
    It holds that $W \succeq \sigma I$ where $\sigma \in (0,1]$.
\end{assumption}

\begin{assumption}[Bound on Dynamics Matrices]
    \label{ass:dyn}
    It holds that $\max(1,\| \Theta_\star \|) \leq S$.
\end{assumption}


\section{Algorithm}

In this section, we propose Algorithm~\ref{alg:main_alg} to address the problem at hand.
This algorithm specifies control policies by choosing a \emph{target covariance} matrix $\Sigma_t$ in each time step and then choosing a control input $u_t$ that will (approximately) induce this covariance in the system.
This covariance-based design approach is described in detail in Section \ref{sec:cov_des}.
At the beginning of the algorithm, the subroutine $\mathsf{Initialize}$ (Algorithm \ref{alg:init}) is called (line \ref{lne:init}), which drives the system with iid noise to determine an initial estimate of the system matrices and therefore identify a good starting covariance matrix $\bar{\Sigma}_0$.
A description of $\mathsf{Initialize}$ is given in Appendix \ref{sec:init}.
Then, for the remainder of the time steps, the algorithm operates over a number of \emph{phases} enumerated by $k$, which each have a \emph{phase covariance} $\bar{\Sigma}_k$.
This phase covariance is chosen using the subroutine $\mathsf{PhaseUpdate}_k$ (Algorithm \ref{alg:phase}) in line \ref{lne:phase_upd}, which operates by first calculating an ``optimistic covariance'' matrix with an SDP, and then shifting this matrix towards the covariance of the zero policy to ensure constraint satisfaction.
A detailed description of $\mathsf{PhaseUpdate}_k$ is given Section~\ref{sec:phase_upd2}.
The condition for a new phase to start is that the determinant of the estimation Gram matrix doubles (line \ref{lne:tstep_start}), which is commonly used in the learning-based control literature, e.g. \citep{abbasi2011regret}.
This ensures that a new phase starts only once a substantial amount of information has been gained about the system, resulting in a quantity of $\Octil(1)$ phases.
Then, in each time step within a given phase, the target covariance $\Sigma_t$ is moved towards the phase covariance $\bar{\Sigma}_k$ (line \ref{lne:mixing}).
This approach of slowly mixing each phase covariance in to the target covariance ensures that the policy changes slowly, which, as we will see in the analysis, helps to ensure that the target covariance effectively approximates the true covariance.

\begin{algorithm}[t]
    \caption{ }
    \label{alg:main_alg}
\begin{algorithmic}[1]
    \Require delay $\rho$, constraint threshold $\xi$,  covariance bound $\nu$, confidence radii $\eta, \mu$, variation $\zeta$, regularization $\lambda$, exploratory noise radius $c$, exploration duration $\tau_0$ 
    \State $\bar{\Sigma}_0, \tau_1 \leftarrow \mathsf{Initialize}$ \Comment{call Algorithm \ref{alg:init}} \label{lne:init}
    \State $\Sigma_{\tau_1} = \bar{\Sigma}_0, k = 0, \bar{V}_k = \sum_{s=1}^{\tau_1-\rho} z_s z_s^\top + \lambda I$
    \For{$t = \tau_1,...,T$}
        \State $V_{t-\rho} = \sum_{s=1}^{t-\rho} z_s z_s^\top + \lambda I$
        \If{$\det(V_{t-\rho}) \geq 2 \det(\bar{V}_k)$ \textbf{or} $t = \tau_1$} \label{lne:tstep_start}
            \State $k = k + 1, \bar{V}_k = V_{t - \rho}, \tau_k = t$. \Comment{initialize phase}
            \State $\bar{\Sigma}_k \leftarrow$ $\mathsf{Phase Update}_k$ \Comment{call Algorithm \ref{alg:phase}} \label{lne:phase_upd}
        \EndIf\
        \State $ \Sigma_t = \Sigma_{t-1} + (\bar{\Sigma}_k - \Sigma_{t-1}) \zeta$ \Comment{mix $\bar{\Sigma}_k$ in to $\Sigma_t$} \label{lne:mixing}
        \State $K_t = \Sigma_{t,ux} \Sigma_{t,xx}^{-1}, \ \ U_t = \Sigma_{t,uu} - K_t \Sigma_{t,xx} K_t^\top, \ \ \Sigma_t = \begin{bmatrix} \Sigma_{t,xx} & \Sigma_{t,xu}\\ \Sigma_{t,ux} & \Sigma_{t,uu} \end{bmatrix}$ \Comment{extract policy} \label{lne:extract}
        \State $u_t = K_t x_t + v_t, \ v_t \sim \Nc(\bzero, U_t)$ \Comment{implement policy} \label{lne:tstep_end}
    \EndFor
\end{algorithmic}
\end{algorithm}

\subsection{Covariance-based Policy Design}
\label{sec:cov_des}

In this section, we provide further details on the technique of specifying a control policy using the desired system covariance, which (to the best of our knowledge) was introduced to the learning-based control literature by \cite{cohen2018online,cohen2019learning}.
In particular, a control policy is specified using a \emph{target covariance} $\Sigma_t$, which represents the desired covariance of the state and input $\cov(z_t)$.
Once the algorithm has chosen a target covariance $\Sigma_t$, it chooses the input distribution according to line \ref{lne:extract} and \ref{lne:tstep_end} in Algorithm \ref{alg:main_alg}, which can be understood as (approximately) inducing a system covariance of $\Sigma_t$.
To gain intuition for how this control policy might induce the desired covariance in the system, we can consider the idealized fictional setting in which $\Sigma_t$ is deterministic and $\cov(x_t) = \Sigma_{t,xx}$,
\begin{equation}
\label{eq:naive1}
\textstyle
    \cov(z_t) = \begin{bmatrix} I \\ K_t \end{bmatrix} \cov(x_t) \begin{bmatrix} I \\ K_t \end{bmatrix}^\top + \begin{bmatrix} 0 & 0\\ 0 & U_t \end{bmatrix} = \begin{bmatrix} I \\ K_t \end{bmatrix} \Sigma_{t,xx} \begin{bmatrix} I \\ K_t \end{bmatrix}^\top + \begin{bmatrix} 0 & 0\\ 0 & U_t \end{bmatrix} = \Sigma_{t},  
\end{equation}
and therefore $\Sigma_t$ would be the true covariance of $z_t$.
Although this simple analysis provides intuition, the required assumptions ($\cov(x_t) = \Sigma_{t,xx}$ and deterministic $\Sigma_t$) are not generally satisfied and therefore our algorithm and analysis will require a more involved approach.

Notably, our approach for ensuring that the target covariance approximates the true covariance substantially differs from prior work \cite{cohen2018online,cohen2019learning}.
This is because these prior works only consider the unconstrained setting and therefore only need to ensure effective approximation of the system covariance in the \emph{aggregate} to ensure low cost, whereas the constraints in our setting apply \emph{per time-step} and therefore necessitate that the covariance is well-approximated in every time step.
In particular, our algorithm uses delayed information for estimating the system matrices (discussed in the next section) and only slowly-varies the target covariance (line \ref{lne:mixing}), which allows us to show that the target covariance approximates the system covariance at \emph{every time step}.
This also results in a substantially different analysis approach, as discussed in Section \ref{sec:ana}.

\subsection{Description of $\mathsf{PhaseUpdate}_k$}
\label{sec:phase_upd2}

\begin{algorithm}[t]
    \caption{$\mathsf{Phase Update}_k$}
    \label{alg:phase}
\begin{algorithmic}[1]
            \State estimate dynamics with delayed information: \label{lne:est}
            \begin{equation}
                \label{eq:sys_est}
                \hat{\Theta}_k = \argmin_{\Theta} \bigg( \sum_{s=1}^{\tau_k-\rho} \| \Theta z_s - x_{s+1} \|^2 + \lambda \| \Theta  - \hat{\Theta}_0 \|^2 \bigg)
            \end{equation}
            \State constrained optimistic SDP: \label{lne:sdp}
            \begin{subequations}
            \label{eq:opt_sdp}
            \begin{align}
                \Sigma_k^o = \argmin_{\Sigma =
        \left[\begin{smallmatrix}
        \Sigma_{xx} & \Sigma_{xu}\\
        \Sigma_{ux} & \Sigma_{uu}
        \end{smallmatrix}\right] \succeq 0} & \ \langle \bQ, \Sigma \rangle \label{eq:opt_sdp:cost}\\
                \text{s.t.} & \ \Sigma_{xx} \succeq \hat{\Theta}_k \Sigma \hat{\Theta}_k^\top + W - \eta \langle \bar{V}_k^{-1}, \Sigma \rangle I \label{eq:opt_sdp:stat}\\
                & \ \langle \balpha_j, \Sigma \rangle \leq \xi, \ \forall j \in [J] \label{eq:opt_sdp:const}\\
                & \ \tr(\Sigma) \leq \nu \label{eq:opt_sdp:bound}
            \end{align}            \end{subequations}
            \State estimate covariance of zero policy: \label{lne:cov_null}
            \begin{equation}
                \label{eq:safe_update}
                \Sigma_{k}^{\safe} = \begin{bmatrix}
                    \Sigma_{k,xx}^{\safe} & 0\\
                    0 & 0
                \end{bmatrix}
                \quad \mathrm{s.t.}
                \quad
                \Sigma_{k,xx}^{\safe} = \hat{\Theta}_k \Sigma_{k}^{\safe} \hat{\Theta}_k^\top + W
            \end{equation}
            \State construct pessimistic set: \label{lne:pess}
            $$\Ec_k^p = \{ \Sigma \succeq 0 : \langle \balpha_j, \Sigma \rangle + \mu \langle \bar{V}_{k-1}^{-1}, \Sigma \rangle \leq \xi, \forall j \in [J] \}$$
            \State scale optimistic covariance in to pessimistic set:\label{lne:scale}
            \begin{align*}
                \phi_k & = \max\{ \phi \in [0,1] : \phi \Sigma_k^o + (1 - \phi) \Sigma_{k}^{\safe} \in \Ec_k^p \},\\
                \bar{\Sigma}_k & = \phi_k \Sigma_k^o + (1 - \phi_k) \Sigma_{k}^{\safe}
            \end{align*}
            \State \Return $\bar{\Sigma}_k$
\end{algorithmic}
\end{algorithm}

In this section, we describe the subroutine $\mathsf{PhaseUpdate}_k$ given in Algorithm \ref{alg:phase}, which is used to choose the phase covariance $\bar{\Sigma}_k$ at the beginning of each phase.
In the following, we describe each of the steps in $\mathsf{PhaseUpdate}_k$.

\paragraph{Dynamics Estimation with Delayed Information (line \ref{lne:est})}

The dynamics $\Theta_\star = [A_\star\ B_\star]$ are estimated with least-squares estimation, using the information from time steps $[1,\tau_k - \rho]$ where $\tau_k$ is the time step at the beginning of the current phase.
The reason that the information is delayed by $\rho$ time steps before being used is because it ensures that the policy at a given time step $(K_t, U_t)$ is $\Fc_{t-\rho}$-measurable.
Intuitively, this ensures that the policy at time step $t$ is ``deterministic'' conditioned on $\Fc_{t-\rho}$, and therefore that the conditional covariance of the state $\cov(x_t|\Fc_{t-\rho})$ can be linearly related to $\cov(z_t|\Fc_{t-\rho})$ analogous to the idealized setting in \eqref{eq:naive1}, i.e. 
$$
\textstyle
\cov(z_t|\Fc_{t-\rho}) = \begin{bmatrix} I \\ K_t \end{bmatrix} \cov(x_t|\Fc_{t-\rho}) \begin{bmatrix} I \\ K_t \end{bmatrix}^\top + \begin{bmatrix} 0 & 0\\ 0 & U_t \end{bmatrix}
$$
This will then allow us to bound the difference between the conditional covariance $\cov(z_t|\Fc_{t-\rho})$ and the target covariance $\Sigma_t$, which is critical for our analysis.


\paragraph{Constrained Optimistic SDP (line \ref{lne:sdp})}


Then, using the estimate of the system matrices, an SDP \eqref{eq:opt_sdp} is used to choose the \emph{optimistic covariance} $\Sigma_k^o$.
The key novelty of this SDP formulation compared to prior work \cite{cohen2019learning} is that it uses a linear constraint to enforce the chance-constraint, which is critical for ensuring constraint satisfaction.
In particular, the SDP constraint \eqref{eq:opt_sdp:const} can be understood as representing a chance-constraint on the steady-state covariance, as it holds that,
\begin{equation}
    \label{eq:sdp_cc}
    \Pb(\alpha_j^\top z_t \leq \beta) \geq 1 - \delta \quad \iff \quad \ip{\balpha_j, \cov(z_t)} \leq \beta^2 / \Phi^{-1}(1 - \delta)^2,
\end{equation}
when the system is at steady-state and when $z_t$ is Gaussian (where $\balpha_j := \alpha_j \alpha_j^\top$).
The other pieces of the SDP follow from \cite{cohen2019learning}, which, for completeness, we discuss as follows.
Indeed, the SDP cost \eqref{eq:opt_sdp:cost} can be understood as representing the expected cost at steady-state, as it holds in this case that, $\Eb[\ell(x_t,u_t)] = \ip{\mathbf{Q}, \cov(z_t)}$ where $\mathbf{Q} := \diag(Q,R)$.
Also, the trace constraint \eqref{eq:opt_sdp:bound} ensures stability.
The linear constraint \eqref{eq:opt_sdp:stat} imposes an approximate steady-state condition. 
Indeed, from the system \eqref{eq:dynamics}, a steady-state covariance $\Sigma$ must satisfy the fixed-point condition $\Sigma_{xx} = \Theta_\star \Sigma \Theta_\star^\top + W$.
However, since $\Theta_\star$ is unknown, \eqref{eq:opt_sdp:stat} enforces a relaxed version of this condition via known bounds on the estimation error,
\begin{equation}
    \label{eq:dyn_error}
    \| \hat{\Theta}_k \Sigma \hat{\Theta}_k - \Theta_\star \Sigma \Theta_\star \| \leq \eta \langle \bar{V}_k^{-1}, \Sigma \rangle,
\end{equation}
which holds with high probability for an appropriately chosen $\eta$.
Thus, any steady-state covariance matrix for the true system will satisfy \eqref{eq:opt_sdp:stat} (with high probability).
As a result, the SDP is ``optimistic'' in that it uses the estimation error bounds \eqref{eq:dyn_error} to lower bound the optimal cost under the true system, and therefore we expect the optimized matrix $\Sigma_k^o$ to result in low regret given the well-known learning paradigm of ``optimism in the face of uncertainty''.
However, the optimistic covariance $\Sigma_k^o$ has no guarantees of constraint satisfaction.
We address this issue next.

\paragraph{Safe Scaling (lines \ref{lne:cov_null},\ref{lne:pess}, \ref{lne:scale})}

The remainder of Algorithm \ref{alg:phase} is dedicated to transforming $\Sigma_k^o$ to ensure constraint satisfaction.
It does so by using the ``scaled-back optimism'' technique proposed in the stochastic bandit literature by \cite{hutchinson2024directional} and further refined by \cite{gangrade2025constrained}.
This technique involves choosing the largest scaling on the optimistic decision that guarantees constraint satisfaction.
We adapt this approach to our setting by first estimating the covariance of the zero policy with $\Sigma_k^\safe$ (line \ref{lne:cov_null}) and then finding the convex combination of $\Sigma_k^o$ and $\Sigma_k^\safe$ such that the resulting covariance $\phi_k \Sigma_k^o + (1 - \phi_k) \Sigma_{k}^{\safe}$ is in the ``pessimistic set'' $\Ec_k^p$.
This set $\Ec_k^p$ (defined in line \ref{lne:pess}) is chosen to ensure constraint satisfaction given the error between $\Sigma_t$ and the true covariance.
Therefore the resulting phase covariance matrix $\bar{\Sigma}_k$ is both efficient (due to the optimistic SDP) and safe (since it is in the pessimistic set).


\section{Analysis}

\label{sec:ana}

In this section, we give the formal theoretical guarantees of the proposed algorithm.
In particular, Theorem \ref{thm:main} shows that Algorithm \ref{alg:main_alg} ensures that the chance-constraints are satisfied for all time steps, and that the regret is $\Octil(\sqrt{T})$ with high probability.

\begin{theorem}
    \label{thm:main}
    Let Assumptions \ref{ass:null_policy}, \ref{ass:noise_cov} and \ref{ass:dyn} hold. 
    Then, if $T \geq T_{\min}$, there exists an appropriate choice of algorithm parameters\footnote{The specific choice of algorithm parameters is given in the full version of the theorem in Appendix \ref{sec:main_proof}.} for Algorithm \ref{alg:main_alg} such that,
    \begin{align*}
        & \Pb(\alpha_j^\top z_t \leq \beta) \geq 1 - \delta, \quad \forall j \in [J], t \in [T]\\
        & \Pb \left(\reg_T \leq C_{\mathrm{reg}} \sqrt{T} \right) \geq 1 - \delta/T,
    \end{align*}
    where $T_{\min}, C_{\mathrm{reg}} = \poly \big(d,\log(T), R_Q, R_1, S, \bar{w}, D, \beta, \kappa, \gamma^{-1}, \sigma^{-1}, \epsilon^{-1}, \delta^{-1}, \Phi^{-1}(1-\delta)^{-1} \big)$.
\end{theorem}

To facilitate discussion, we note the choice of algorithm parameters used in proving Theorem \ref{thm:main} (where $\asymp$ hides polynomial dependence on the problem parameters and $\log(T)$):
    \begin{equation}
        \label{eq:alg_params}
        \textstyle
        \lambda \asymp \sqrt{T},\eta \asymp \sqrt{T},\mu \asymp \sqrt{T},\zeta \asymp \frac{1}{\sqrt{T}}, \tau_0 \asymp \sqrt{T}, \nu \asymp 1,\rho \asymp 1,c\asymp 1,\xi - \frac{\beta^2}{\Phi^{-1}(1 - \delta)^2} \asymp \frac{1}{\sqrt{T}},
\end{equation}
In the remainder of this section, we give a high-level discussion of the proof of Theorem \ref{thm:main}, and defer the complete proof to Appendix \ref{sec:main_proof}.
The central piece of this proof is characterizing the error between the target covariance $\Sigma_t$ and the actual (conditional) covariance $\cov(z_t | \Fc_{t-\rho})$, which is presented in Section \ref{sec:cov_approx}.
This analysis is then used for establishing the regret bound and constraint guarantees, as discussed in Section \ref{sec:reg_ana} and Section \ref{sec:const_ana}, respectively.

We emphasize that our analysis approach differs substantially from the prior optimistic approaches for learning to control \citep{abbasi2011regret,cohen2019learning}.
For one, these works analyze the regret in terms of the cost-to-go matrix $P$, whereas we analyze the regret and constraint in terms of the covariance matrix $\Sigma$.
We use this approach because the constraint is defined in terms of the covariance $\Sigma$ and therefore requires analysis in the $\Sigma$-domain.
In comparison to \cite{cohen2019learning}, an advantage of our approach is that we do not need to relate the $\Sigma$-SDP to the dual $P$-SDP, and therefore can use weaker assumptions.
For example, \cite{cohen2019learning} requires that both $Q$ and $R$ are positive definite, whereas we require that neither are.
Furthermore, \cite{abbasi2011regret} and \cite{cohen2019learning} are only interested in minimizing \emph{total} cost (and not constraint satisfaction), and therefore only need to bound the performance of the algorithm in the aggregate.
In contrast, the constraints in our setting apply every time-step, and therefore we indeed need guarantees that apply at every time step.
In our setting, this translates to bounds on the difference between the actual and target covariance at every time step, which we show in the following.

In this section, we denote the conditional mean as $m_{t|t-\rho} = \Eb[z_t | \Fc_{t-\rho}]$ and the conditional covariance as $S_{t|t-\rho} := \cov(z_t | \Fc_{t-\rho})$.
We also use $k_t$ to refer to the phase that time step $t$ belongs.
The lemmas in this section are proved in Appendix \ref{sec:ana_proofs}.

\subsection{Covariance Approximation}
\label{sec:cov_approx}

In this section, we bound the difference between the actual (conditional) covariance $S_{t|t-\rho}$ and the target covariance $\Sigma_t$.
This bound characterizes the algorithm's ability to steer the system covariance to a desirable value, and therefore is critical to characterizing the performance of the algorithm.
As a first step towards establishing this bound, we give a technical lemma (Lemma \ref{lem:cov_approx1}) that provides sufficient conditions for the error between the actual and target covariance to be bounded. 

\begin{lemma}
    \label{lem:cov_approx1}
    Suppose that the sequence of target covariances $(\Sigma_t)_{t \in [T]}$ satisfy the following:
    \begin{enumerate}
        \item $\Sigma_s$ are $\Fc_{t-\rho}$-measurable for all $s \in [t-\rho,t]$. \label{it:cov:meas}
        \item $\| \Sigma_{s+1} - \Sigma_s \| \leq \bar{\zeta}$ for all $s \in [t - \rho, t]$, \label{it:cov:vary}
        \item $\Sigma_{s,xx} \succeq \Theta_\star \Sigma_s \Theta_\star^\top + W - \bar{\eta}_s I$ for all $s \in [t-\rho,t]$, \label{it:cov:1}
        \item $0 \preceq \Sigma_t \preceq \nu$  for all $t$,
        \item $\| x_{t-\rho} \| \leq R_x$
    \end{enumerate}
    Furthermore, if $\bar{\zeta} \leq \frac{\bar{\sigma}^2}{2 \nu}$, then it holds that,
    \begin{align}
        \| m_{t|t-\rho} \| & \leq (1 + \tilde{\kappa}) \tilde{\kappa} \exp(- \tilde{\gamma} \rho / 2) R_x \label{eq:mean_app_error}\\
        S_{t|t - \rho} - \Sigma_t & \preceq (1 + \tilde{\kappa})^2 \tilde{\kappa}^2 \big(  \nu \exp( - \tilde{\gamma} \rho) +  \rho \bar{\zeta} +  {\textstyle \sum_{s=t-\rho}^{t}} \bar{\eta}_s \big)I, \label{eq:cov_app_error}
    \end{align}
    where $(\tilde{\kappa}, \tilde{\gamma}) = (\frac{\nu}{\bar{\sigma}}, \frac{\bar{\sigma}}{2 \nu})$ and $\bar{\sigma} = \sigma - \max_{s \in [t-\rho,r]}\bar{\eta}_s$.
\end{lemma}

The conditions in Lemma \ref{lem:cov_approx1} can be interpreted as follows: condition 1 is that the past $\rho$ target covariances are $\Fc_{t-\rho}$-measurable, condition 2 is a bound the variation of the target covariances, condition 3 is a relaxed steady-state condition, condition 4 is that target covariances are psd and bounded, and condition 5 is a bound on the state at time step $t - \rho$.
Our algorithm is designed to ensure that all of these conditions are satisfied with high probability such that the bounds $\| m_{t|t-\rho} \|$ and $S_{t|t-\rho} - \Sigma_t$  in \eqref{eq:mean_app_error} and \eqref{eq:cov_app_error} are small.

The following discussion explains how the algorithm ensures that the conditions of Lemma \ref{lem:cov_approx1} are satisfied.
First, the proposed algorithm uses delayed state information when estimating the system matrices \eqref{eq:sys_est} with delay $\rho \asymp \log(T)$, which ensures that $\Sigma_s$ are $\Fc_{t-\rho}$-measurable for all $s \in [t-\rho,t]$ (i.e. condition 1) and that the term $\exp(-\tilde{\gamma}\rho)$ in \eqref{eq:cov_app_error} is $1/T$.
Furthermore, the algorithm slowly varies the target covariance with rate $\zeta \asymp 1/\sqrt{T}$, which ensures that $\| \Sigma_t - \Sigma_{t+1}\| \lesssim 1/\sqrt{T}$ (i.e. condition 2) and therefore that the term $\bar{\zeta}$ in \eqref{eq:cov_app_error} is $1/\sqrt{T}$.
Ensuring that the relaxed steady-state condition holds (condition 3) is much more difficult because the steady-state condition in the optimistic SDP in \eqref{eq:opt_sdp:stat} has a relaxation of $\eta \langle \bar{V}_k^{-1}, \Sigma \rangle $, which changes from phase to phase.
Therefore, it is challenging to ascertain whether the target covariance $\Sigma_t$ satisfies a steady-state condition, given that it only slowly tracks the phase covariance. 
We remedy this by showing that the phases can be made sufficiently long by choosing $\lambda$ appropriately, and therefore that the target covariance $\Sigma_t$ will get close to $\bar{\Sigma}_k$ by the end of the phase.
Therefore, we can show that the target covariance in the current phase is close to a mixture of the last phase covariance $\bar{\Sigma}_{k-1}$ and current phase covariance $\bar{\Sigma}_k$.
This allows for us to show a relaxed steady-state condition on $\Sigma_t$ where the relaxation $\bar{\eta}$ is approximately $\eta \langle \bar{V}_{k-1}^{-1}, \Sigma \rangle$, i.e. from the previous phase $k - 1$.
Thus, we show that condition 3 is satisfied with the $\bar{\eta}$ in \eqref{eq:cov_app_error} as $\eta \ip{\bar{V}_{k-1}^{-1}, \Sigma_t }$ with a shrinking error term.
Next, to show that $0 \preceq \Sigma_t \preceq \nu$ (condition 4) holds, we note that the target covariance is a convex combination of matrices that satisfy $0 \preceq \Sigma \preceq \nu$ since this condition is satisfied by the past optimistic covariances and zero policy covariances.
Thus, from convexity, it holds that $0 \preceq \Sigma_t \preceq \nu$.
Finally, to show the high probability boundedness of the states (condition 5), we use a fairly standard analysis approach similar to \cite{abbasi2011regret} and \cite{cohen2019learning}.

Ultimately, the aformentioned techniques result in a bound on the covariance approximation error of the form in Lemma \ref{lem:approx2}.
Note that we later show (in Lemma \ref{lem:ellip_pot}) that $\sum_{t=1}^{T} \ip{\bar{V}_{k-1}^{-1}, \Sigma_t} = \Octil(1)$, and therefore Lemma \ref{lem:approx2} implies that the aggregate covariance approximation error is $\Octil(\sqrt{T})$.

\begin{lemma}
    \label{lem:approx2}
    Assume the same as Theorem \ref{thm:main}.
    Then, it holds with high probability that,
    \begin{equation}
    \label{eq:cov_error}
    \textstyle
        \| m_{t|t-\rho} \| \leq \Octil \left( \frac{1}{T} \right), \quad S_{t|t-\rho} - \Sigma_t \preceq \Octil(1) \left( \frac{1}{\sqrt{T}} + \sqrt{T} \ip{\bar{V}_{k_t-1}^{-1}, \Sigma_t} \right) I .
    \end{equation}
\end{lemma}

\subsection{Regret Analysis}
\label{sec:reg_ana}

In this section, we discuss the regret analysis used for Theorem \ref{thm:main}.
First, we introduce the $\tilde{\Sigma}$  as the ``benchmark policy'' that the algorithm aims to learn:
\begin{subequations}
    \label{eq:bench}
    \begin{align}
    \tilde{\Sigma} = \argmin_{\Sigma \succeq 0} & \ \langle \bQ, \Sigma \rangle\\
    \text{s.t.} & \ \Sigma_{xx} = \Theta_\star \Sigma \Theta_\star^\top + W \label{eq:bench_ss}\\
    & \ \langle \balpha_j, \Sigma \rangle \leq \xi, \ \forall j \in [J] \label{eq:bench_const} \\
    & \ \tr(\Sigma) \leq \nu 
\end{align}
\end{subequations}
Note that this SDP only differs from the optimistic SDP in \eqref{eq:opt_sdp} in that it uses the exact steady-state condition with the exact system \eqref{eq:bench_ss} instead of the relaxed one with the estimated system in \eqref{eq:opt_sdp:stat}.
Then, we decompose the regret:
\begin{equation}
\label{eq:reg_decomp}
    \reg_T = \underbrace{\sum_{t=\tau_1}^{T} \left( \ell(x_t, u_t) - \ip{\mathbf{Q}, \Sigma_t} \right)}_{\tone} + \underbrace{\sum_{t=\tau_1}^{T} \ip{\mathbf{Q}, \Sigma_t - \tilde{\Sigma}} }_{\ttwo} + \underbrace{T \ip{\mathbf{Q}, \tilde{\Sigma}} - J_\star}_{\tthree} + \underbrace{\sum_{t=1}^{\tau_1-1}(\ell(x_t, u_t) - \ip{\mathbf{Q}, \tilde{\Sigma}} )}_{\tfour}
\end{equation}
Intuitively, Term I measures how effectively the algorithm steers the actual cost to the target cost, Term II is the regret in learning the benchmark covariance, Term III is the error in using the benchmark covariance to approximate the optimal cost, and Term IV is the regret of the $\mathsf{Initialize}$ subroutine.
We discuss Term I, II and III in the following sections. 
The analysis of Term IV is straightforward and therefore discussion of it is omitted (the rigorous analysis of Term IV is given in the full proof in Appendix \ref{sec:main_proof}).

\subsubsection{Term I}

In order to handle $\tone$, we further decompose it as follows.
\begin{align*}
    & \tone = \\
    & \underbrace{\sum_{t=\tau_1}^{T} \left( \ell(x_t, u_t) - \Eb[\ell(x_t,u_t) |\Fc_{t-\rho}] \right)}_{\mathrm{Term\ I.A}} + \underbrace{\sum_{t=\tau_1}^{T} \left( \Eb[\ell(x_t,u_t) |\Fc_{t-\rho}] - \ip{\mathbf{Q}, S_{t|t-\rho}} \right)}_{\mathrm{Term\ I.B}} + \underbrace{\sum_{t=\tau_1}^{T} \ip{\mathbf{Q}, S_{t|t-\rho} - \Sigma_t}}_{\mathrm{Term\ I.C}}
\end{align*}
Notice that Term I.A is simply the difference between the realized cost and the conditional expectation of the cost, and therefore can be shown to be $\Octil(\sqrt{T})$ with high probability via an Azuma-type inequality.
Next, Term~I.B can be written as $\sum_{t=\tau_1}^{T} m_{t|t-\rho}^\top \mathbf{Q} m_{t|t-\rho}$, which is $\Octil(1)$ given \eqref{eq:mean_app_error} in Lemma \ref{lem:approx2}.
Finally, since Term I.C is a linear function of the covariance approximation error, we can bound it with \eqref{eq:cov_app_error} in Lemma \ref{lem:approx2}.
However, the term $\ip{\bar{V}_{k_t-1}^{-1}, \Sigma_t}$ shows up in \eqref{eq:cov_app_error}, and we have not yet established a bound on it.
This is remedied in the following lemma.

\begin{lemma}
    \label{lem:ellip_pot}
    Assume the same as Theorem \ref{thm:main}.
    Then, it holds that $\sum_{t=1}^{T} \ip{\bar{V}_{k-1}^{-1}, \Sigma_t} = \Octil(1)$ with high probability.
\end{lemma}

Thus, it follows from Lemma \ref{lem:approx2} and Lemma \ref{lem:ellip_pot} that Term I.C is $\Octil(\sqrt{T})$.
Putting together the bounds for each of the terms shows that Term I is $\Octil(\sqrt{T})$ as desired.

\subsubsection{Term II}

Term II is decomposed as follows.
$$
    \ttwo = \underbrace{\sum_{t=\tau_1}^{T} \ip{\mathbf{Q}, \Sigma_t - \bar{\Sigma}_{k_t}}}_{\mathrm{Term\ II.A}} + \underbrace{\sum_{t=\tau_1}^{T} \ip{\mathbf{Q}, \bar{\Sigma}_{k_t} - \Sigma_{k_t}^o}}_{\mathrm{Term\ II.B}} + \underbrace{\sum_{t=\tau_1}^{T} \ip{\mathbf{Q}, \Sigma_{k_t}^o - \tilde{\Sigma}}}_{\mathrm{Term\ II.C}}
$$
Term II.A represents how closely $\Sigma_t$ tracks the phase covariance $\bar{\Sigma}_k$.
Since the tracking rate is $\zeta \asymp 1/\sqrt{T}$, and there are only $N= \Octil(1)$ phases, it can be shown that $\sum_{t=\tau_1}^{T} \| \Sigma_t  - \bar{\Sigma}_{k_t} \| \lesssim N/\zeta = \Octil(\sqrt{T})$ and therefore Term II.A is $\Octil(\sqrt{T})$.
Term II.B represents the difference between the phase covariance $\bar{\Sigma}_k$ and the optimistic covariance $\Sigma_k^o$, and therefore is proporitional to the safe scaling term $1 - \phi_k$.
We bound $1 - \phi_k$ using an analysis similar to the one used by \cite{hutchinson2024directional,gangrade2025constrained} to analyze a related scaling-based constraint satisfaction approach in the safe stochastic bandit literature.
This yields a bound of the form $1 - \phi_k \lesssim \sqrt{T} \ip{\bar{V}_{k-1}^{-1}, \Sigma_t}$, which combined with Lemma \ref{lem:ellip_pot}, shows that Term II.B is $\Octil(\sqrt{T})$.
Finally, since the optimistic covariance $\Sigma_{k_t}^o$ is indeed ``optimistic,'' it holds that Term II.C is non-positive with high probability.
Thus, putting the bounds on all of the terms together, it holds that $\ttwo$ is $\Octil(\sqrt{T})$.

\subsubsection{Term III}

For the analysis of Term III, we introduce a covariance matrix that serve as an intermediate between the optimal cost and the benchmark covariance,
\begin{subequations}
    \label{eq:tilde}
    \begin{align}
    \Sigma' = \argmin_{\Sigma =
    \left[\begin{smallmatrix}
    \Sigma_{xx} & \Sigma_{xu}\\
    \Sigma_{ux} & \Sigma_{uu}
    \end{smallmatrix}\right]} & \ \langle \bQ, \Sigma \rangle\\
    \text{s.t.} & \ \Sigma_{xx} = \Theta_\star \Sigma \Theta_\star^\top + W \\
    & \ \langle \balpha_j, \Sigma \rangle \leq \left( \frac{\beta - C_1 \exp(-\gamma T)}{\Phi^{-1}(1 - \delta)} \right)^2 + C_2 \exp(-\gamma T) , \ \forall j \in [J] \label{eq:tilde_ss}\\
    & \ \tr(\Sigma) \leq \nu  \\
    & \ \Sigma \succeq 0 
\end{align}
\end{subequations}
Notice that this SDP only differs from the one for the benchmark covariance \eqref{eq:bench} in the constraint \eqref{eq:tilde_ss}, where the terms $C_1,C_2 > 0$ are polynomial in the problem constants.
Note that \eqref{eq:tilde} is designed to ensure that the steady-state covariance of the optimal policy is a feasible matrix.
To this end, the terms $C_1 \exp(-\gamma T)$ and $C_2 \exp(-\gamma T)$ account for the fact that optimal policy only needs to ensure constraint satisfaction on the transient system (i.e. not at steady-state). 

Then, using $\Sigma'$, we decompose Term III as,
$$
    \tthree = \underbrace{T \ip{\mathbf{Q}, \tilde{\Sigma} - \Sigma'}}_{\mathrm{Term\ III.A}} + \underbrace{T \ip{\mathbf{Q}, \Sigma'} - J_\star}_{\mathrm{Term\ III.B}}.
$$
We handle Term III.A by bounding the distance between the feasible sets of the SDPs for $\tilde{\Sigma}$ and $\Sigma'$.
As expected, this can be bound in terms of the difference between the righthand sides of the constraints \eqref{eq:bench_const} and \eqref{eq:tilde_ss}, which has a $\exp(-\gamma T)$ from \eqref{eq:tilde_ss} and a $1/\sqrt{T}$ term from $\xi$ as in \eqref{eq:alg_params}.
Thus, Term III.A is bounded by $\sqrt{T}$.
As for Term III.B, we use the fact that the steady-state covariance of the optimal policy is feasible for \eqref{eq:tilde}, and therefore $\ip{\mathbf{Q}, \Sigma'}$ is a lower bound on the cost of the steady-state covariance of the optimal policy.
Thus, Term III.B captures the transient part of the cost, and is therefore no larger than a constant.

\subsection{Constraint Analysis}
\label{sec:const_ana}

In this section, we give an overview of the analysis showing that the chance-constraints are satisfied.
In particular, this analysis relies on the following technical lemma, which provides conditions on the conditional mean and covariance of the state that guarantee satisfaction of the chance-constraints.

\begin{lemma}
\label{lem:const_bound}
    Fix $\omega \in [0,\delta)$.
    If $z_t|\Fc_{t-\rho} \sim \Nc(m_{t|t-\rho}, S_{t|t-\rho})$, and,
    \begin{equation}
        \label{eq:cov_const}
        \textstyle
        \Pb\Big(\ip{\balpha_j, S_{t|t-\rho}} \leq \big( \frac{\beta - \alpha_j^\top m_{t|t-\rho}}{\Phi^{-1}(1 - \delta + \omega)} \big)^2 , \alpha_j^\top m_{t|t-\rho} \leq \beta \Big) \geq 1 - \omega
    \end{equation}
    then it holds that $\Pb( \alpha_j^\top z_t \leq \beta) \geq 1 - \delta$.
\end{lemma}

We use Lemma \ref{lem:const_bound} to guarantee constraint satisfaction by showing that $z_t|\Fc_{t-\rho}$ is indeed Gaussian and that \eqref{eq:cov_const} is satisfied.
First, note that the use of delayed information for system estimation \eqref{eq:sys_est} ensures that the policies $K_t, U_t$ in the past $\rho$ time steps are $\Fc_{t-\rho}$-measurable and therefore that $z_t|\Fc_{t-\rho}$ is Gaussian.
Then, to show that \eqref{eq:cov_const} is satisfied, we use Lemma \ref{lem:approx2} to get that $\| m_{t|t-\rho} \| \lesssim 1/T$ and therefore, \eqref{eq:cov_const} can be approximated as,
\begin{equation}
    \label{eq:cond_cov_cons}
    \textstyle
    \ip{\balpha_j, S_{t|t-\rho}} \leq \left( \frac{\beta - C_3 \frac{1}{T}}{\Phi^{-1}(1 - \delta + \omega)} \right)^2,
\end{equation}
where $C_3$ depends on problem constants.
Our algorithm constrains $\Sigma_t$ to ensure that $S_{t|t-\rho}$ robustly satisfies \eqref{eq:cond_cov_cons} given the error between $S_{t|t-\rho}$ and $\Sigma_t$ in Lemma \ref{lem:approx2}.
The algorithm imposes this constraint by ensuring that the phase covariance $\bar{\Sigma}_k$ (and, by extension, the target covariance $\Sigma_t$) is in the pessimistic set $\Ec_{k}^p$ via the safe scaling step in line \ref{lne:scale} of Algorithm \ref{alg:phase}.
Thus, by appropriately choosing the algorithm parameters $\mu$ and $\xi$, the pessimistic set $\Ec_{k}^p$ can be designed such that \eqref{eq:cond_cov_cons} is satisfied given that $\Sigma_t \in \Ec_k^p$. 
Doing so requires \eqref{eq:cond_cov_cons} to be related to the steady-state covariance constraint $\ip{\balpha_j, \Sigma} \leq \beta^2 / \Phi^{-1}(1-\delta)^2$, which we do by taking the probability level $\omega \asymp 1/T$ and then deriving a local Lipschitz bound on the reciprocal quantile $1/\Phi^{-1}( \cdot)$.

\section{Conclusion}

In this work, we have shown $\Octil(\sqrt{T})$ regret and constraint satisfaction at every time step for the LQ adaptive control problem with chance-constraints.
Some interesting directions for future work are to consider more general disturbance distributions (e.g. using Cantelli's inequality), and to investigate whether certainty equivalence yields $\Octil(\sqrt{T})$ regret in the constrained setting (as has already been shown in the unconstrained setting).



\bibliographystyle{plainnat}
\bibliography{references}

\appendix

\newpage

\section{Description of $\mathsf{Initialize}$ Subroutine}

\label{sec:init}

The $\mathsf{Initialize}$ subroutine (Algorithm \ref{alg:init}) consists of $\tau_0$ time steps in which random noise is played to explore the system, and then an additional $\rho$ time steps of just playing the zero policy.
Then, it computes the least-squares estimator of the system $\hat{\Theta}_0$, and uses it to compute an estimate of the covariance of the zero policy $\bar{\Sigma}_0$.
It then returns $\bar{\Sigma}_0$ and the total number of time steps taken $\tau_1$.

\begin{algorithm}[H]
    \caption{$\mathsf{Initialize}$}
    \label{alg:init}
\begin{algorithmic}[1]
    \For{$t = 1,..., \tau_0-1$}
        \State $u_t \sim \Uc(c \Sb)$ \Comment{exploration policy} \label{lne:iid_noise}
    \EndFor
    \For{$t = \tau_0,...,\tau_0+\rho-1$}
        \State $u_t = \bzero$ \Comment{zero policy} \label{lne:settle}
    \EndFor
    \State $\hat{\Theta}_0 = \argmin_{\Theta} \sum_{s=1}^{\tau_0-1} \| \Theta z_s - x_{s+1} \|^2$ \Comment{initial estimate of dynamics}
    \State $\bar{\Sigma}_{0} = \begin{bmatrix}
            \bar{\Sigma}_{0,xx} & 0\\
            0 & 0
        \end{bmatrix}
        \ \mathrm{s.t.}
        \
        \bar{\Sigma}_{0,xx} = \hat{\Theta}_0 \bar{\Sigma}_{0} \hat{\Theta}_0^\top + W$ \Comment{estimate cov of zero policy}
    \State $\tau_1 = \tau_0 + \rho$
    \State \Return $\bar{\Sigma}_0$, $\tau_1$
\end{algorithmic}
\end{algorithm}

\section{Proof of Theorem \ref{thm:main}}
\label{sec:main_proof}

In this section, we give the complete proof of Theorem \ref{thm:main}.
First, we give the full version of Theorem~\ref{thm:main}.

\begin{theorem}[Full Version of Theorem \ref{thm:main}]
\label{thm:main_apx}
    Let Assumptions \ref{ass:null_policy}, \ref{ass:noise_cov} and \ref{ass:dyn} hold.
    Suppose that the algorithm parameters are chosen as follows:
    \begin{align*}
        \lambda & = \max\left( 2 C_4 r (1 + S R_z \sqrt{T}), 2 R_z^2(H + \rho), 4 \kappa^2 \gamma^{-2} r, 4 C_4^2 r^2 S^2 \right)\\
        \nu & = (1 + \kappa)^2 \gamma^{-1} \kappa^2 \bar{w}\\
        \eta & = r (1 + S \sqrt{\lambda + T R_z^2})\\
        \mu & = 2(1 + \tilde{\kappa})^2 \tilde{\kappa}^2 D^2 \eta\\
        \rho & = 2 \tilde{\gamma}^{-1} \log \left( \frac{\max(\beta,\epsilon/4) T}{(1+\tilde{\kappa}) \tilde{\kappa} R_z D} \right)\\
        \zeta & = \frac{\sigma^2}{16 \nu^2 \sqrt{T}}\\
        \xi & = \left( \frac{\beta - D (1 + \tilde{\kappa}) \tilde{\kappa} \left(1 - \tilde{\gamma}/2\right)^{\rho} R_z}{ \Phi^{-1}(1 - \delta + \omega)} \right)^2 - D^2 C_1\\
        \tau_0 & = 1 + 300^2 \bar{w}^2 d^2 \max(\sigma^{-1}, m/c^2) (n + d \log(306 \bar{\Gamma} \max(\sigma^{-1}, m/c^2) \omega^{-1})) \lambda \\
        c & = \frac{\gamma \epsilon}{D \kappa^{2} S^{2}} 
    \end{align*}
    The preceding definitions use the following quantities.
    \begin{align}
        h & = \frac{8 \nu (1 + S^2 + 2 (n+m))}{\sigma} \\
        H & = \frac{1}{\zeta} \log(h T)\\
        C_1 & = (1 + \tilde{\kappa})^2 \tilde{\kappa}^2  D^2 \Big(  \nu \exp( - \tilde{\gamma} \rho) +  2 \nu \zeta \rho +  2 \rho \nu T^{-1} h^{-1} (1 + S^2 + 2 \eta (n + m) \lambda^{-1}) + 2 \nu \zeta \eta (n + m) \lambda^{-1} \rho(\rho + 1) \Big) \label{eq:c1}\\
        R_z & = \max\left(\frac{2 \nu}{\sigma}+\frac{4 \nu^2}{\sigma^2},\kappa + \kappa^2\right)\left( R_1 + 4 \max(\frac{4 \nu}{\sigma}, \gamma^{-1}) \left( \left(\sqrt{\bar{w}} + S \sqrt{\nu} \right) \sqrt{2 n \log(12 n T /\omega)} + c \right) \right) \label{eq:rz}\\
        r & = 4 \bar{w} d^2 \log\left( 3 \left(d + T R_z^2 \right)/\omega \right) + 2\\
        \omega & = \frac{\delta}{6 T} \\
        C_4 & = \max \left( \frac{8 \nu}{\sigma}, 1, 8 \Phi^{-1} (1 - \delta) \nu (2(1 + \tilde{\kappa})^2 \tilde{\kappa}^2 D^2 + \kappa^2 \gamma^{-1} D^2) \right)\\
        (\tilde{\kappa}, \tilde{\gamma})  & = (\frac{2 \nu}{\sigma}, \frac{\sigma}{4 \nu}) \label{eq:tild_ss} \\
        \bar{\Gamma} & = \left( \kappa^2 R_1 + \kappa^2 \gamma^{-1} (\sqrt{2 n \bar{w} \log(2 n T/\omega)} + S c) + c \right)^2
    \end{align}
    Then, if $T \geq T_{\min}$, it holds that,
    \begin{gather*}
        \Pb(\alpha_j^\top z_t \leq \beta) \geq 1 - \delta, \quad \forall j \in [J], t \in [T]\\
        \Pb \left(\reg_T \leq C_{\mathrm{reg}} \sqrt{T} \right) \geq 1 - \delta/T,
    \end{gather*}
    where $T_{\min}, C_{\mathrm{reg}} = \poly \big(d,\log(T), R_Q, R_1, S, \bar{w}, D, \beta, \kappa, \gamma^{-1}, \sigma^{-1}, \epsilon^{-1}, \delta^{-1}, \Phi^{-1}(1-\delta)^{-1} \big)$.
\end{theorem}

In order to prove Theorem \ref{thm:main_apx}, we first give additional notation in Section \ref{sec:not}, some useful lemmas in Section \ref{sec:useful_lemmas}, bounds on the main regret terms in Section \ref{sec:reg_decomp_apx} and the constraint satisfaction guarantee in Section \ref{sec:const_sat}.
Finally, we complete the proof in Section \ref{sec:comp_proof}.

\subsection{Notation}
\label{sec:not}

In this section, we introduce the main notation.
First, we use $k_t$ to refer to the phase index in which time step $t$ belongs, and $N$ to be the total number of phases.
We also use $\Sigma_t = \bar{\Sigma}_0$ for all $t \in [\tau_0, \tau_0 + \rho-1]$.
Also, we define the conditional mean and covariance (for $t > \rho$) as $m_{t|\tau} := \Eb\left[ z_t | \Fc_{\tau} \right]$ and $S_{t|\tau} := \Eb\left[ (z_t - m_{t|\tau}) (z_t - m_{t|\tau})^\top  | \Fc_{\tau} \right]$.
In the following subsections, we introduce specific types of additional notation.



\subsubsection{Baseline Policies}

We define several sets of baseline policies in this section.
First, we define the set of $\Sigma$ which satisfies the stationarity condition and with trace bounded by $\nu$:
$$
    \bar{\Ec}_\nu := \{ \Sigma \succeq 0 : \Sigma\xx = \Theta_\star \Sigma \Theta_\star^\top + W,\ \tr(\Sigma) \leq \nu \}
$$
With this, we then define the subset of $\bar{\Ec}_\nu$ that satisfies the linear constraint in \eqref{eq:opt_sdp}, and the resulting $\Sigma$:
\begin{align*}
    \tilde{\Ec} & := \left\{ \Sigma \in \bar{\Ec}_\nu : \ip{\balpha_j, \Sigma} \leq \xi  \right\}\\
    \tilde{\Sigma} & := \argmin_{\Sigma \in \tilde{\Ec}}\ \ip{\mathbf{Q}, \Sigma}
\end{align*}
Also, let $\Sigma_{\star}^{\safe}$ be the true covariance of the state under zero input, i.e. 
$$
    \Sigma_{\star}^{\safe} = \begin{bmatrix} \Sigma_{\star,xx}^{\safe} & 0 \\ 0 & 0 \end{bmatrix} \quad \mathrm{s.t.} \quad \Sigma_{\star,xx}^{\safe} = \Theta_\star \Sigma_{\star}^{\safe} \Theta_\star^\top + W
$$
Lastly, we let $\Pi_\star$ be the set of linear controllers that are feasible for the optimal controller, i.e.
\begin{equation}
\label{eq:pistar}
    \Pi_\star := \big\{  (\kappa,\gamma)\ \text{strongly-stable}\ K : \Pb ( \alpha_j^\top z_t^K \leq \beta ) \geq 1 - \delta, \ \forall j \in [J] \big\}
\end{equation}

\subsubsection{Named Terms}

In addition to the terms defined in Theorem \ref{thm:main_apx}, we also introduce the following terms:
\begin{align}
    & \begin{aligned}
    C_2 & := 48 \max\left( \frac{\nu}{\sigma}, 1 \right) d \log(1 + \lambda^{-1} T) + 12 \max\left( \frac{\nu}{\sigma}, 1 \right) \lambda^{-1} R_z \rho \log(2 T / \omega )
    \end{aligned} \label{eq:c2}\\
    & \begin{aligned}
        C_3 & := \frac{\left( 2 \beta + \epsilon \right) D \left( (1 + \tilde{\kappa}) \tilde{\kappa} \left(1 - \tilde{\gamma}/2\right)^{\rho} R_z + D (1 + \kappa) \kappa \left(1 - \gamma\right)^{T} R_1 \right)}{ \Phi^{-1}(1 - \delta)^2}\\
        & \qquad + \frac{2 \Phi^{-1}(1-\delta/2) 4 \beta^2 \omega}{\phi(\Phi^{-1}(1-\delta/2)) \Phi^{-1}(1-\delta)^4} + D^2 (1 + \kappa)^2 \kappa^2 \bar{w} \gamma^{-1} (1 - \gamma)^{2 T}
    \end{aligned} \label{eq:c3}\\
    & \begin{aligned}
        \epsilon_1 & := \frac{\epsilon^2}{4 \Phi^{-1}(1 - \delta)^2}  - D^2 (1 + \kappa)^2 \kappa^2 \gamma^{-1} (1 - \gamma)^{2 T} - \frac{2 \Phi^{-1}(1-\delta/2)  \beta^2 \omega}{\phi(\Phi^{-1}(1-\delta/2)) \Phi^{-1}(1-\delta)^4} - C_1\\
        & \quad - \lambda^{-1} \nu (\mu + \kappa^2 \gamma^{-1} \eta D^2)
    \end{aligned} \label{eq:e1}
\end{align}

\subsubsection{High-probability Events}

We also use $F_k$ to refer to the event that the approximation bound holds for the first $k$ phases and the state is bounded up to the start of phase $k$:
\begin{align*}
    F_k = \Big\{ & \| \Theta_\star \Sigma \Theta_\star - \hat{\Theta}_i \Sigma \hat{\Theta}_i \| \leq \eta \ip{\bar{V}_i^{-1}, \Sigma} \ \forall \Sigma \succeq 0,\ \forall i \in \{0,1,...,k\}, \\
    & \| \Theta_\star - \hat{\Theta}_i \|  \leq \frac{\gamma}{2 \kappa}  \ \forall i \in \{0,1,...,k\}, \\
    & \| z_t \| \leq R_z, \ \forall t \in [\tau_0,\tau_{k}-1]\Big\}
\end{align*}
We also use $E$ to refer to the event that the state is always bounded, and the approximation bound always holds:
\begin{equation}
\label{eq:event_E}
    \begin{split}
        E := \Big\{ & \| \Theta_\star \Sigma \Theta_\star - \hat{\Theta}_k \Sigma \hat{\Theta}_k \| \leq \eta \ip{\bar{V}_k^{-1}, \Sigma} \ \forall \Sigma \succeq 0,\ \forall k \in \{0,...,N\}, \\
        & \| \Theta_\star - \hat{\Theta}_k \|  \leq \frac{\gamma}{2 \kappa}  \ \forall k \in \{0,...,N\}, \\
        & \| z_t \| \leq R_z, \ \forall t \in [T]\Big\}.
    \end{split}
\end{equation}
Note that $F_1 \subseteq \cdots \subseteq F_N \subseteq E$.
We also use $G_1$ for the event that the bound on the Gram weighted target covariance holds:
\begin{equation}
        G_1 := \left\{ \sum_{t=1}^{T} \ip{\bar{V}_{k_t-1}^{-1}, \Sigma_t} \leq C_2  \right\} \label{eq:g1}
\end{equation}
Next, we use $G_2$ for the event that the bound on Term I.A holds:
\begin{equation}
\label{eq:g2}
    G_2 := \left\{ \sum_{t=\tau_1}^{T} \left( \ell(x_t, u_t) - \ip{\mathbf{Q}, S_{t|t-\rho}} \right) \leq 2 R_Q R_z^2 \sqrt{\rho 2  T \log(\rho/\omega)} + R_Q (1 + \tilde{\kappa})^2 \tilde{\kappa}^2 R_z^2 \exp(-\tilde{\gamma} \rho/2) T\right\}
\end{equation}
Lastly, we use $G_3$ to be the event that the estimation error bound on the initial system estimate holds:
\begin{equation}
\label{eq:g3}
    G_3 := \left\{ \| \hat{\Theta}_0 - \Theta_\star \|_F \leq 300 \bar{w} d \sqrt{\frac{n + d \log(306 \bar{\Gamma} \min(\sigma, c^2/m)^{-1} \omega^{-1})}{(\tau_0 - 1)\min(\sigma, c^2/m)}} \right\}
\end{equation}

\subsection{Useful Lemmas}
\label{sec:useful_lemmas}
In this section, we give lemmas that will be useful for the analysis.
First, Lemma \ref{lem:slow_var} gives some properties of the slowly-varying policy method, which is proven in Section \ref{sec:slow_var_proof}.

\begin{lemma}[Properties of Slowly-varying Policy]
    \label{lem:slow_var}
    Assume that $\zeta \in (0,1]$. For $t \geq \tau_0$, it holds that,
    \begin{itemize}
        \item[1.] $\Sigma_t \in \conv \left\{ \Sigbar_{0},..., \Sigbar_{k_t} \right\}$.
    \end{itemize}
    If it additionally holds that $\nu \geq 2 \kappa^2 \tr(W) \gamma^{-1}$and event $E$ holds, then for all $t \geq \tau_0$:
    \begin{enumerate}
        \item[2.] $\sum_{t=1}^{T} \| \Sigbar_{k_t} - \Sigma_t \| \leq \frac{2 \nu N}{\zeta}$
        \item[3.] $\| \Sigma_{t+1} - \Sigma_{t} \| \leq 2 \nu \zeta$.
        \item[4.] $\| \bar{\Sigma}_{k_t} - \Sigma_t \| \leq 2 \nu (1 - \zeta)^{t - \tau_k}$.
    \end{enumerate}
\end{lemma}

Next, Lemma \ref{lem:bounded_state} shows that the state is bounded with high probability, which is proven in Section~\ref{sec:bounded_state_proof}.

\begin{lemma}[Bounded State and Estimation Bound]
    \label{lem:bounded_state}
    In addition to the assumptions of Lemma \ref{lem:safe_margin}, assume that:
    \begin{enumerate}
        \item $\lambda \geq \max \left(\frac{H R_z^2}{\log(2)},\frac{8 \eta \nu}{\sigma}, \eta, 4 \kappa^2 \gamma^{-2} \left( 4 \bar{w} d^2 \log\left( 3 \left(d + \lambda^{-1} T R_z^2 \right)/\omega \right) + 2 \right) \right)$
        \item $\Pb \left( \| \hat{\Theta}_0 - \Theta_\star \|_F \leq \lambda^{-1/2} \right) \geq 1 - \omega/3$
        \item $\eta \geq (4 \bar{w} d^2 \log\left( 3 \left(d + \lambda^{-1} T R_z^2 \right)/\omega \right) + 2) (1  + S \sqrt{\lambda + T R_z^2 })$
    \end{enumerate}
    Then, it holds that $\Pb(E) \geq 1 - \omega$ where $E$ is defined in \eqref{eq:event_E}.
\end{lemma}

Then, we give a lemma that provides an ordering on the Gram matrices in the current phase and between phases. The proof is in Section \ref{sec:phase_upd}.

\begin{lemma}[Properties of Phase Update]
    \label{lem:phase_upd}
    Suppose that $\| z_t \| \leq R_z$ for all $t$, and that $\lambda \geq \max\left( \frac{R_z^2 \rho}{\log(2)}, R_z^2  \right)$.
    Then, the following hold:
    \begin{enumerate}
        \item $V_{t - \rho} \preceq 2 \bar{V}_{k_t}$ for all $t$,
        \item $\bar{V}_k \preceq 3 \bar{V}_{k-1}$ for all $k$,
        \item $V_t \preceq 2 V_{t - \rho}$ for all $t$,
        \item $N \leq 1 + 2 d \log \left( 1  + \frac{R_z^2 T}{d \lambda} \right)$,
        \item $\tau_{k+1} - \tau_k \geq \frac{\log(2)}{R_z^2} \lambda$ for all $k$.
    \end{enumerate}
\end{lemma}

Next, we give a lemma that shows how conditional chance-constraints can be expressed in terms of the conditional mean and conditional covariance. The proof is in Section \ref{sec:quant_form}.

\begin{lemma}[Reformulation of Chance Constraints]
    \label{lem:quant_form}
    If $z_t | \Fc_\tau \sim \Nc(m_{t|\tau}, S_{t|\tau})$, then it holds that,
    $$
        \Pb( \alpha_j^\top z_t \leq \beta | \Fc_\tau) \geq 1 - \delta \quad \iff \quad \alpha_j^\top m_{t|\tau} + \sqrt{\ip{\balpha_j, S_{t | \tau}}} \Phi^{-1}(1 - \delta) \leq \beta,
    $$
    where $\Phi(\cdot)$ is the standard normal CDF.
\end{lemma}

Next, we give a lemma that provides a bound on the difference between the conditional covariance $\Sigma_{t|t-\rho}$ and the chosen covariance matrix $\Sigma_t$. The proof is in Section \ref{sec:main_cov}.

\begin{lemma}[State Approximation]
    \label{lem:main_cov}
    In addition to the assumptions of Theorem \ref{thm:main_apx}, also assume that:
    \begin{enumerate}
        \item $H = \zeta^{-1} \log(h T)$ \label{it:state:6}
        \item $\lambda \geq \max\left(\frac{R_z^2 (H + \rho)}{\log(2)}, \frac{8 \eta \nu}{\sigma}, \eta \right)$ \label{it:state:7}
        \item $h \geq \frac{8 \nu (1 + S^2 + 2 (n+m))}{\sigma}$ \label{it:state:8}
        \item $\zeta \leq \frac{\sigma^2}{16 \nu^2}$ \label{it:state:9}
        \item $\nu \geq \frac{2 \kappa^2 \tr(W)}{\gamma}$ \label{it:state:10}
    \end{enumerate}
    Then, under event $E$, it holds for all $t \in [\tau_1, T]$ that,
    \begin{align*}
        S_{t|t-\rho} & \preceq \Sigma_t + \left (C_1 + (1 + \tilde{\kappa})^2 \tilde{\kappa}^2 2 \rho \eta \ip{\bar{V}_{k-1}^{-1}, \Sigma_t}\right ) I \\
        \| m_{t|t - \rho} \| & \leq (1 + \tilde{\kappa}) \tilde{\kappa} \exp(-\tilde{\gamma} \rho /2) R_z,
    \end{align*}
    where $(\tilde{\kappa}, \tilde{\gamma})$ and $C_1$ are defined in \eqref{eq:tild_ss} and  \eqref{eq:c1} respectively.
\end{lemma}

Then, Lemma \ref{lem:safe_margin} show the conditions that are required for the initial safe policy to strictly satisfy the constraint in the optimistic SDP in the algorithm.
The proof is given in Section \ref{sec:safe_margin_proof}.

\begin{lemma}
    \label{lem:safe_margin}
    In addition to the assumptions of Theorem \ref{thm:main_apx}, also assume that:
    \begin{enumerate}
        \item $\nu \geq 2 \bar{w} \gamma^{-1} \kappa^2$
        \item $\rho \geq 2 \tilde{\gamma}^{-1} \log \left( \frac{\max(\beta,\epsilon/4)}{(1+\tilde{\kappa}) \tilde{\kappa} R_z D} \right)$
        \item $T \geq \gamma^{-1} \log \left( \frac{\max(\beta,\epsilon/4)}{(1+\kappa) \kappa R_1 D} \right)$
        \item $\omega \leq \delta/2$
        \item $\epsilon_1 > 0$ where $\epsilon_1$ is defined in \eqref{eq:e1},
    \end{enumerate}
    Then, under event $F_k$, it holds that,
    \begin{gather*}
        \ip{\balpha_j, \Sigma_k^{\safe}} + \mu \ip{\bar{V}_{k-1}^{-1}, \Sigma_{k}^{\safe}} \leq \xi - \epsilon_1\\
        \ip{\balpha_j, \Sigma_\star^{\safe}} \leq \xi - \epsilon_1
    \end{gather*}
\end{lemma}

Then, we give a Lemma that bounds the cumulative of $\ip{\bar{V}_{k_t-1}^{-1}, \Sigma_t}$ via the event $G_1$, which is related to the estimation error of the parameters. The proof is in Section \ref{sec:ell_pot}.

\begin{lemma}[Cumulative Estimation Error]
    \label{lem:ell_pot}
    Suppose that the assumptions of Lemma \ref{lem:main_cov} hold.
    Then, if $\Pb(E) \geq 1 - \omega$, it holds that $\Pb(E \cap G_1) \geq 1 - 2 \omega$ where $G_1$ is defined in \eqref{eq:g1}.
\end{lemma}

Then, we give an error bound on the initial dynamics estimation.
The proof is in Section \ref{sec:exploration_proof}.

\begin{lemma}[Initialization Phase]
    \label{lem:exploration}
    In addition to the assumptions of Theorem \ref{thm:main_apx}, also assume that:
    \begin{enumerate}
        \item $\tau_0 \geq 1 + 134 d \log\left( \frac{34 R_1^2}{\omega \min(\sigma, c^2/m)} \right)$
        \item $c = \gamma \epsilon D^{-1} \kappa^{-2} S^{-1}$
    \end{enumerate}
    Then, it holds that $\Pb(\alpha_j^\top z_t \leq \beta) \geq 1 - \delta$ for all $t \in [\tau_1-1]$, and $\Pb\left( G_3 \right) \geq 1 - \omega/3$, where event $G_3$ is defined in \eqref{eq:g3}.
\end{lemma}

Lastly, we give a bound showing that the realized cost concentrates around the conditional expectation of the cost, i.e. a bound on Term I.A.
The proof is in Section \ref{sec:term1a}.

\begin{lemma}[Concentration of Cost]
    \label{lem:term1a}
    Suppose that the assumptions of Lemma \ref{lem:main_cov} hold.
    If $\Pb(E) \geq 1 - \omega$, then it holds that $\Pb(E \cap G_2) \geq 1 - 2 \omega$, where $G_2$ is defined in \eqref{eq:g2}.
\end{lemma}

\subsection{Regret Analysis}
\label{sec:reg_decomp_apx}

Recall the regret decomposition from Section \ref{sec:reg_ana}:
\begin{align*}
    \reg_T & = \sum_{t=1}^{T} \ell(x_t, u_t) - J_\star\\
    & = \underbrace{\sum_{t=\tau_1}^{T} \left( \ell(x_t, u_t) - \ip{\mathbf{Q}, \Sigma_t} \right)}_{\tone} + \underbrace{\sum_{t=\tau_1}^{T} \ip{\mathbf{Q}, \Sigma_t - \tilde{\Sigma}} }_{\ttwo} + \underbrace{T \ip{\mathbf{Q}, \tilde{\Sigma}} - J_\star}_{\tthree} + \underbrace{\sum_{t=1}^{\tau_1-1}(\ell(x_t, u_t) - \ip{\mathbf{Q}, \tilde{\Sigma}} )}_{\tfour}
\end{align*}

Then, we give the bounds on each of the terms.
First, the bound on Term I is given in Lemma~\ref{lem:tone_first}, which is proven in Section \ref{sec:tone_proof}.

\begin{lemma}[Term I]
    \label{lem:tone_first}
    Suppose that the assumptions of Lemma \ref{lem:main_cov} hold.
    Then, under the intersection of events $E$, $G_1$ and $G_2$,
    \begin{align*}
        \mathrm{Term\ I} & \leq 2 R_Q R_z^2 \sqrt{\rho 2  T \log(\rho/\omega)} + R_Q (1 + \tilde{\kappa})^2 \tilde{\kappa}^2 R_z^2 \exp(-\tilde{\gamma} \rho/2) T + R_Q T C_1\\
        & \quad + 2 (1 + \tilde{\kappa})^2 \tilde{\kappa}^2 R_Q \eta \rho C_2
    \end{align*}
\end{lemma}

Next, we give a bound on Term II, which is proven in Section \ref{sec:ttwo_proof}.

\begin{lemma}[Term II]
    \label{lem:ttwo_first}
    Suppose that the conditions of Lemma \ref{lem:safe_margin} hold.
    Then, under event $E$, it holds that,
    $$
        \mathrm{Term\ II} \leq \zeta^{-1} R_Q 2 \nu N + 2 \nu R_Q \mu C_2 \epsilon_1^{-1} + 4 (n + m) N \nu^2 R_Q \mu \epsilon_1^{-1} \lambda^{-1} \zeta^{-1} 
    $$
\end{lemma}

Then, we give a bound on Term III, which is proven in Section \ref{sec:tthree_proof}.

\begin{lemma}[Term III]
    \label{lem:opt_approx_first}
    Assume that the following:
    \begin{enumerate}
        \item $\nu \geq (1 + \kappa)^2 \gamma^{-1} \kappa^2$
        \item $\rho \geq 2 \tilde{\gamma}^{-1} \log \left( \frac{\max(\beta,\epsilon/4)}{(1+\tilde{\kappa}) \tilde{\kappa} R_z D} \right)$
        \item $T \geq \gamma^{-1} \log \left( \frac{\max(\beta,\epsilon/4)}{(1+\kappa) \kappa R_1 D} \right)$
        \item $\omega \leq \delta/2$
        \item $\xi = \left( \frac{\beta - D (1 + \tilde{\kappa}) \tilde{\kappa} \left(1 - \tilde{\gamma}/2\right)^{\rho} R_z}{ \Phi^{-1}(1 - \delta + \omega)} \right)^2 - D^2 C_1$
        \item $\ip{\balpha_j, \Sigma_\star^{\safe}} \leq \xi - \epsilon_1$
    \end{enumerate}
    Then, it follows that,
    \begin{align*}
        \tthree \leq R_Q \frac{2 \nu}{\epsilon_1} (D^2 C_1 + C_3) T + R_Q (1 + \kappa)^2 \kappa^2 \gamma^{-2}\bar{w}.
    \end{align*}
\end{lemma}

The bound on Term IV is straightforward, so we leave the analysis to Section \ref{sec:comp_proof}.

\subsection{Constraint Satisfaction}
\label{sec:const_sat}

Next, we give the constraint satisfaction guarantees in the following lemma, which is proven in Section \ref{sec:const_sat_proof}.

\begin{lemma}[Constraint Satisfaction Guarantee]
    \label{lem:const_sat_first}
    In addition to the assumptions in Lemma \ref{lem:main_cov}, Lemma \ref{lem:safe_margin} and Lemma \ref{lem:exploration}, suppose that the following hold:
    \begin{enumerate}
        \item $\xi = \left( \frac{\beta - D (1 + \tilde{\kappa}) \tilde{\kappa} \left(1 - \tilde{\gamma}/2\right)^{\rho} R_z}{ \Phi^{-1}(1 - \delta + \omega)} \right)^2 - D^2 C_1$
        \item $\mu = 2 (1 + \tilde{\kappa})^2 \tilde{\kappa}^2  D^2 \eta$
        \item $\rho \geq 2 \tilde{\gamma}^{-1} \log \left( \frac{\beta}{(1+\tilde{\kappa}) \tilde{\kappa} R_z D} \right)$
        \item $\omega < \delta$
        \item $\Pb(E) \geq 1 - \omega$
    \end{enumerate}
    Then, it holds that $\Pb( \alpha_j^\top z_t \leq \beta) \geq 1 - \delta$ for all $t \in [T]$ and $j \in [J]$.
\end{lemma}

\subsection{Completing the Proof}
\label{sec:comp_proof}

Finally, combining the preceding results, we  give the proof of Theorem \ref{thm:main_apx}.

\begin{proof}[Proof of Theorem \ref{thm:main_apx}]
    To simplify arguments, we use $x \lesssim y$ and $x \asymp y$ to denote $x \leq C y$ and $x = C y$ respectively, where $C$ is polynomial in the problem parameters and $\log(T)$.
    We show the claims by applying the bounds on each of the regret terms, and showing the constraint satisfaction guarantees.

    \emph{Properties of Parameter Choices:}
    First, we discuss some properties of the choice of algorithm parameters in Theorem \ref{thm:main_apx}.
    The key property that we show is that the chosen $\lambda$ and $\eta$ satisfy the following conditions,
    \begin{equation}
        \label{eq:lam_thet}
        \begin{split}
        \lambda & \geq \max \left( \frac{8 \eta \nu}{\sigma}, \eta \right),\\
        \lambda & \geq 8 \Phi^{-1} (1 - \delta) \nu (\mu + \kappa^2 \gamma^{-1} \eta D^2),\\
        \eta & \geq r (1 + S\sqrt{\lambda + T R_z^2})        
    \end{split}
    \end{equation}
    First, it is immediate that the choice of $\lambda$ in Theorem \ref{thm:main_apx} satisfies,
    \begin{align*}
        & \lambda \geq 4 C_4^2 r^2 S^2\\ & \ \implies \ C_4 r S \sqrt{\lambda} \leq \lambda /2 \\ & \ \implies \  C_4 r (1 + S \sqrt{\lambda} + R_z \sqrt{T}) \leq \lambda/2 + C_4 r (1 + S R_z \sqrt{T})
    \end{align*}
    Since the choice of $\lambda$ also satisfies $\lambda \geq 2 C_4 r (1 + S R_z \sqrt{T})$, it follows that,
    \begin{align*}
        \lambda & = \lambda/2 + \lambda/2\\
        & \geq \lambda/2 + C_4 r (1 + S R_z \sqrt{T})\\
        & \geq C_4 r (1 + S \sqrt{\lambda} + R_z \sqrt{T})\\
        & \geq C_4 r (1 + S \sqrt{\lambda + T R_z^2})\\
        & \geq \max \left( \frac{8 \nu}{\sigma}, 1, 8 \Phi^{-1} (1 - \delta) \nu (2(1 + \tilde{\kappa})^2 \tilde{\kappa}^2 D^2 + \kappa^2 \gamma^{-1} D^2) \right) r (1 + S\sqrt{\lambda + T R_z^2}),
    \end{align*}
    which satisfies \eqref{eq:lam_thet}. The last line uses the definition of $C_4$.

    \emph{Bounds on Named Terms:} Next, we show bounds on the key terms $C_1, C_2, C_3, \epsilon_1$.
    First, it holds that,
    \begin{equation}
    \label{eq:c1_bound}
    \begin{split}
        C_1 & = (1 + \tilde{\kappa})^2 \tilde{\kappa}^2  D^2 \Big(  \nu \exp( - \tilde{\gamma} \rho) +  2 \nu \zeta \rho +  2 \rho \nu T^{-1} h^{-1} (1 + S^2 + 2 \eta (n + m) \lambda^{-1})\\
        & \qquad + 2 \nu \zeta \eta (n + m) \lambda^{-1} \rho(\rho + 1) \Big) \\
        & \lesssim \frac{1}{\sqrt{T}},
    \end{split}
    \end{equation}
    where we use that $\tilde{\kappa} \lesssim 1$, $\nu \lesssim 1$, $\exp( - \tilde{\gamma} \rho) \lesssim 1/T$, $\zeta \lesssim 1/\sqrt{T}$, $\rho \lesssim 1$, $h \asymp 1$, $\eta \lambda^{-1} \lesssim 1$.
    Next, it holds that,
    \begin{equation}
    \label{eq:c2_bound}
    \begin{split}
        C_2 & = 48 \max\left( \frac{\nu}{\sigma}, 1 \right) d \log(1 + \lambda^{-1} T) + 12 \max\left( \frac{\nu}{\sigma}, 1 \right) \lambda^{-1} R_z \rho \log(2 T / \omega )\\
        & \lesssim 1,
    \end{split}
    \end{equation}
    where we additionally use that $R_z \lesssim 1$.
    Next, it holds that,
    \begin{equation}
    \label{eq:c3_bound}
    \begin{split}
        C_3 & = \frac{\left( 2 \beta + \epsilon \right) D \left( (1 + \tilde{\kappa}) \tilde{\kappa} \left(1 - \tilde{\gamma}/2\right)^{\rho} R_z + D (1 + \kappa) \kappa \left(1 - \gamma\right)^{T} R_1 \right)}{ \Phi^{-1}(1 - \delta)^2}\\
        & \qquad + \frac{2 \Phi^{-1}(1-\delta/2) 4 \beta^2 \omega}{\phi(\Phi^{-1}(1-\delta/2)) \Phi^{-1}(1-\delta)^4} + D^2 (1 + \kappa)^2 \kappa^2 \gamma^{-1} (1 - \gamma)^{2 T}\\
        & \lesssim \frac{1}{T},
    \end{split}
    \end{equation}
    where we use that $(1 - \tilde{\gamma}/2)^\rho \leq \exp(-\tilde{\gamma} \rho /2) \lesssim 1/T$, $\omega \lesssim 1/T$, and $(1 - \gamma)^T \leq \exp(-\gamma T) \leq 1/T$ given that $T \geq T_{\min}$ and we can take $T_{\min} \geq \log(T)$.
    Lastly, since $T \geq T_{\min}$, it holds for a valid choice of $T_{\min}$,
    \begin{equation}
    \label{eq:e1_bound}
    \begin{split}
         \epsilon_1 & = \frac{\epsilon^2}{4 \Phi^{-1}(1 - \delta)^2}  - D^2 (1 + \kappa)^2 \kappa^2 \gamma^{-1} (1 - \gamma)^{2 T} - \frac{2 \Phi^{-1}(1-\delta/2)  \beta^2 \omega}{\phi(\Phi^{-1}(1-\delta/2)) \Phi^{-1}(1-\delta)^4} - C_1\\
        & \quad - \lambda^{-1} \nu (\mu + \kappa^2 \gamma^{-1} \eta D^2)\\
        & > 0,
    \end{split}
    \end{equation}
    where we use that \eqref{eq:lam_thet} implies that $\lambda^{-1} \nu (\mu + \kappa^2 \gamma^{-1} \eta D^2) \leq \frac{\epsilon^2}{8 \Phi^{-1}(1 - \delta)^2}$, and furthermore that $\omega \lesssim 1/T$, $C_1 \lesssim 1/\sqrt{T}$.

    \emph{High-probability Events:} Next, we bound the probability of the high probability events. The high probability events that need to hold for the analysis are $E$, $G_1$, $G_2$ and $G_3$ as defined in \eqref{eq:event_E}, \eqref{eq:g1}, \eqref{eq:g2} and \eqref{eq:g3} respectively.
    First, we know that $\Pb(G_3) \geq 1 - \omega/3$ from Lemma \ref{lem:exploration}, since the required conditions on $\tau_0$ and $c$ are satisfied by the specified choice in Theorem \ref{thm:main_apx}.
    Next, we show that $\Pb(E) \geq 1 - \omega$ via Lemma \ref{lem:bounded_state}.
    Lemma \ref{lem:bounded_state} requires the conditions listed in the lemma as well as the conditions of Lemma \ref{lem:safe_margin}.
    We note that the conditions on $\rho$, $\omega$, $\nu$, $\lambda$ and $\eta$ are satisfied by the chosen values in Theorem \ref{thm:main_apx}.
    The condition that $\Pb ( \| \hat{\Theta}_0 - \Theta_\star \|_F \leq \lambda^{-1/2} ) \geq 1 - \omega/3$ is satisfied given that $\Pb(G_3) \geq 1 - \omega/3$ (as we have already shown) and the choice of $\tau_0$.
    The requirement that $\epsilon_1 > 0$ is shown in \eqref{eq:e1_bound}.
    Thus, we have shown that $\Pb(E) \geq 1 - \omega$.
    Next, we show that $\Pb(G_1) \geq 1 - 2 \omega$ via Lemma \ref{lem:ell_pot}.
    Lemma \ref{lem:ell_pot} requires the conditions of Lemma \ref{lem:main_cov}, which are satisfied by the algorithm parameter choices in Theorem \ref{thm:main_apx}.
    Thus, applying Lemma \ref{lem:ell_pot} and using that $\Pb(E) \geq 1 - \omega$ (as we have already shown) ensures that $\Pb(G_1) \geq 1 - 2 \omega$.
    Then, to bound the probability of $G_2$ we can apply Lemma \ref{lem:term1a} to get that $\Pb(G_2) \geq 1 - 2 \omega$ knowing that $\Pb(E) \geq 1 - \omega$ and the conditions of Lemma \ref{lem:main_cov} hold (as we have already shown).
    Finally, taking the union bound shows that, $$
    \Pb(E \cap G_1 \cap G_2 \cap G_3) \geq 1 - 6 \omega = 1 - \delta/T,
    $$
    where we use the choice of $\omega$ in Theorem \ref{thm:main_apx}.
    The following regret analysis will work under the event $E \cap G_1 \cap G_2 \cap G_3$.

    \emph{Term I:}
    Next, we bound Term I via Lemma \ref{lem:tone_first}.
    Lemma \ref{lem:tone_first} requires the conditions of Lemma \ref{lem:main_cov}, which have already shown to hold.
    Furthermore, we have already bounded the probability of the intersection of $E$, $G_1$ and $G_2$.
    Thus, it holds that,
    \begin{align*}
        \mathrm{Term\ I} & \leq 2 R_Q R_z^2 \sqrt{\rho 2  T \log(\rho/\omega)} + R_Q (1 + \tilde{\kappa})^2 \tilde{\kappa}^2 R_z^2 \exp(-\tilde{\gamma} \rho/2) T + R_Q T C_1\\
        & \quad + 2 (1 + \tilde{\kappa})^2 \tilde{\kappa}^2 R_Q \eta \rho C_2\\
        & \lesssim \sqrt{T},
    \end{align*}
    where we use that $C_1 \lesssim 1/\sqrt{T}$ due to \eqref{eq:c1_bound}, and $\eta \lesssim \sqrt{T}$, with the other terms being $\lesssim 1$.

    \emph{Term II:} Next, we bound Term II via Lemma \ref{lem:ttwo_first}.
    This lemma requires the conditions of Lemma \ref{lem:safe_margin}, which we have already shown to be satisfied.
    Thus, it holds that,
    $$
        \mathrm{Term\ II} \leq \zeta^{-1} R_Q 2 \nu N + 2 \nu R_Q \mu C_2 \epsilon_1^{-1} + 4 (n + m) N \nu^2 R_Q \mu \epsilon_1^{-1} \lambda^{-1} \zeta^{-1} \lesssim \sqrt{T},
    $$
    where we use that $\zeta^{-1} \lesssim \sqrt{T}$, $N \lesssim 1$ (Lemma \ref{lem:phase_upd}), $\mu \lesssim \sqrt{T}$, $C_2 \lesssim 1$ due to \eqref{eq:c2_bound}, and $\mu \lambda^{-1} \lesssim 1$.

    \emph{Term III:} Next, we bound Term III via Lemma \ref{lem:opt_approx_first}.
    This lemma requires that $\ip{\balpha_j, \Sigma_\star^{\safe}} \leq \xi - \epsilon_1$, which holds given Lemma \ref{lem:safe_margin} (which we have already shown that the conditions of are satisfied) and event $E$.
    Thus, it holds that,
    $$
    \tthree \leq R_Q \frac{2 \nu}{\epsilon_1} (D^2 C_1 + C_3) T + R_Q (1 + \kappa)^2 \kappa^2 \gamma^{-2}\bar{w} \lesssim \sqrt{T},
    $$
    where $C_1 \lesssim 1/\sqrt{T}$ due to \eqref{eq:c1_bound} and $C_3 \lesssim 1/T$ due to \eqref{eq:c3_bound}.

    \emph{Term IV:} Finally, we look at Term IV.
    Under event $E$, it holds that,
    $$
        \tfour = \sum_{t=1}^{\tau_1-1}(\ell(x_t, u_t) - \ip{\mathbf{Q}, \tilde{\Sigma}} ) = \sum_{t=1}^{\tau_1-1}\ip{\mathbf{Q}, z_t z_t^\top - \tilde{\Sigma}} \leq (R_z^2 + \nu) R_Q (\tau_0 + \rho),
    $$
    where we use that $\| z_t \| \leq R_z$ under $E$, $\| \tilde{\Sigma} \| \leq \nu$ by definition, $\tr(\mathbf{Q}) \leq R_Q$ by assumption, and $\tau_1 = \tau_0 + \rho$ by definition.
    Since $\tau_0 \lesssim \lambda \lesssim \sqrt{T}$ and $\rho \lesssim 1$, it holds that $\tfour \lesssim \sqrt{T}$.

    \emph{Constraint Satisfaction:}
    Next, we show the constraint satisfaction guarantees.
    We have already shown that the conditions of Lemma \ref{lem:main_cov}, Lemma \ref{lem:safe_margin} and that  $\Pb(E) \geq 1 - \omega$.
    Also, the conditions of Lemma \ref{lem:exploration} are satisfied directly from the choice of algorithm parameters.
    Thus, we can apply Lemma \ref{lem:const_sat_first} to get the constraint satisfaction guarantees.
\end{proof}

\section{Missing Proofs}

\subsection{Proof of Lemma \ref{lem:tone_first} (Term I)}

\label{sec:tone_proof}

In this section, we prove Lemma \ref{lem:tone_first}, which bounds Term I.
This analysis relies on a decomposition of Term I as follows:
$$
    \tone = \underbrace{\sum_{t=\tau_1}^{T} \left( \ell(x_t, u_t) - \ip{\mathbf{Q}, S_{t|t-\rho}} \right)}_{\mathrm{Term\ I.A}} + \underbrace{\sum_{t=\tau_1}^{T} \ip{\mathbf{Q}, S_{t|t-\rho} - \Sigma_t}}_{\mathrm{Term\ I.B}}
$$
In the following, we prove the lemma, where Term I.A is bounded with martingale concentration in Lemma \ref{lem:term1a} and
Term I.B is bounded directly the covariance approximation bound in Lemma \ref{lem:main_cov}.

\begin{proof}[Proof of Lemma \ref{lem:tone_first}]
    Under event $G_1$, it follows immediately from Lemma \ref{lem:main_cov} that (using that $\mathbf{Q} \succeq 0$),
    \begin{align*}
        \mathrm{Term\ I.B} & = \sum_{t=\tau_1}^{T} \ip{\mathbf{Q}, S_{t|t-\rho} - \Sigma_t}\\
        & \leq\tr(\mathbf{Q}) T C_1 + \tr(\mathbf{Q}) (1 + \tilde{\kappa})^2 \tilde{\kappa}^2 2 \rho \eta \sum_{t=\tau_1}^{T} \ip{\bar{V}_{k_t-1}^{-1}, \Sigma_t}\\
        & \leq R_Q T C_1 + 2 R_Q (1 + \tilde{\kappa})^2 \tilde{\kappa}^2 \rho \eta C_2
    \end{align*}
    Then, combining Term I.A (Lemma \ref{lem:term1a}) and Term I.B, the claim folllows.
\end{proof}

\subsection{Proof of Lemma \ref{lem:ttwo_first} (Term II)}

\label{sec:ttwo_proof}

In this section, we prove Lemma \ref{lem:ttwo_first}, which bounds Term II.
The key piece of this analysis is the following lemma, which bounds the safe scaling $\phi_k$ using an approach similar to the stochastic bandit work \cite{hutchinson2024directional}.

\begin{lemma}
    \label{lem:safe_scal}
    Assume that $\ip{\balpha_j, \Sigma_k^{\safe}} + \mu \ip{\bar{V}_{k-1}^{-1}, \Sigma_{k}^{\safe}} \leq \xi - \epsilon_1$.
    Then, it holds that,
    $$
        1 - \phi_k \leq \frac{\mu\ip{\bar{V}_{k-1}^{-1}, \bar{\Sigma}_k}}{\epsilon_1}
    $$
\end{lemma}
\begin{proof}
    Let $\Sigma = \phi \Sigma_k^o + (1 - \phi) \Sigma_k^\safe$ for $\phi = \frac{\epsilon_1}{\mu\ip{\bar{V}_{k-1}^{-1}, \Sigma_k^o} + \epsilon_1}$.
    Then, it holds that,
    \begin{align*}
        \ip{\balpha_j, \Sigma} + \mu \ip{\bar{V}_{k-1}^{-1}, \Sigma} & = \phi(\ip{\balpha_j, \Sigma_k^o} + \mu \ip{\bar{V}_{k-1}^{-1}, \Sigma_k^o}) + (1 - \phi) (\ip{\balpha_j, \Sigma_k^\safe} + \mu \ip{\bar{V}_{k-1}^{-1}, \Sigma_k^\safe})\\
        & \leq \phi(\xi + \mu \ip{\bar{V}_{k-1}^{-1}, \Sigma_k^o}) + (1 - \phi) (\xi - \epsilon_1)\\
        & \leq \xi + \phi \mu \ip{\bar{V}_{k-1}^{-1}, \Sigma_k^o} - (1 - \phi) \epsilon_1\\
        & = \xi.
    \end{align*}
    Therefore, $\phi_k \geq \phi$. Also, since $\bar{\Sigma}_k = \phi_k \Sigma_k^o + (1 - \phi_k) \Sigma_k^\safe$,
    \begin{align*}
        \epsilon_1 & \leq \phi_k \mu\ip{\bar{V}_{k-1}^{-1}, \Sigma_k^o} + \phi_k \epsilon_1 \\
        & = \mu\ip{\bar{V}_{k-1}^{-1}, \bar{\Sigma}_k} - (1 - \phi_k) \mu\ip{\bar{V}_{k-1}^{-1}, \Sigma_k^{\safe}} + \phi_k \epsilon_1\\
        & \leq \mu\ip{\bar{V}_{k-1}^{-1}, \bar{\Sigma}_k} + \phi_k \epsilon_1,
    \end{align*}
    where we use the fact that $\ip{\bar{V}_{k-1}^{-1}, \Sigma_k^{\safe}} \geq 0$ since  $\bar{V}_{k-1}^{-1} \succeq 0$ and $\Sigma_k^o \succeq 0$.
    The lemma follows from rearranging.
\end{proof}


With this, we then give the proof of Lemma \ref{lem:ttwo_first}.

\begin{proof}[Proof of Lemma \ref{lem:ttwo_first}]
We use the following decomposition:
$$
    \sum_{t=\tau_1}^{T} \ip{\mathbf{Q}, \Sigma_t - \tilde{\Sigma}} = \underbrace{\sum_{t=\tau_1}^{T} \ip{\mathbf{Q}, \Sigma_t - \bar{\Sigma}_{k_t}}}_{\mathrm{Term\ II.A}} + \underbrace{\sum_{t=\tau_1}^{T} \ip{\mathbf{Q}, \bar{\Sigma}_{k_t} - \Sigma_k^o}}_{\mathrm{Term\ II.B}} + \underbrace{\sum_{t=\tau_1}^{T} \ip{\mathbf{Q}, \Sigma_k^o - \tilde{\Sigma}}}_{\mathrm{Term\ II.C}}
$$
First, it holds that,
\begin{align*}
    \mathrm{Term\ II.A} & \leq \sum_{t=\tau_1}^{T} \ip{\mathbf{Q}, \Sigma_t - \bar{\Sigma}_{k_t}}\\
    & \leq R_Q \sum_{t=\tau_1}^{T} \| \Sigma_t - \bar{\Sigma}_{k_t} \| \\
    & \leq \zeta^{-1} R_Q 2 \nu N,
\end{align*}
where the last line uses Lemma \ref{lem:slow_var}.
Then, we bound Term II.B using Lemma \ref{lem:safe_scal},
\begin{align*}
    \mathrm{Term\ II.B} & = \sum_{t=\tau_1}^{T} \ip{\mathbf{Q}, \bar{\Sigma}_{k_t} - \Sigma_k^o}\\
    & = \sum_{t=\tau_1}^{T} (1 - \phi_{k_t}) \ip{\mathbf{Q},\Sigma_{k_t}^{\safe} -  \Sigma_k^o}\\
    & \leq 2 \nu R_Q \sum_{t=\tau_1}^{T} (1 - \phi_{k_t})\\
    & \leq \frac{2 \nu R_Q \mu}{\epsilon_1} \sum_{t=\tau_1}^{T} \ip{\bar{V}_{k_t-1}^{-1}, \bar{\Sigma}_{k_t}}\\
    & = \frac{2 \nu R_Q \mu}{\epsilon_1} \sum_{t=\tau_1}^{T} \ip{\bar{V}_{k_t-1}^{-1}, \Sigma_t} + \frac{2 \nu R_Q \mu}{\epsilon_1} \sum_{t=\tau_1}^{T} \ip{\bar{V}_{k_t-1}^{-1}, \bar{\Sigma}_{k_t} - \Sigma_t}\\
    & \leq \frac{2 \nu R_Q \mu}{\epsilon_1} \sum_{t=\tau_1}^{T} \ip{\bar{V}_{k_t-1}^{-1}, \Sigma_t} + \frac{2 \nu R_Q \mu}{\epsilon_1} \sum_{t=\tau_1}^{T} \tr (\bar{V}_{k_t-1}^{-1} ) \| \bar{\Sigma}_{k_t} - \Sigma_t \| \\
    & \leq \frac{2 \nu R_Q \mu}{\epsilon_1} \sum_{t=\tau_1}^{T} \ip{\bar{V}_{k_t-1}^{-1}, \Sigma_t} + \frac{4 (n + m) \nu^2 R_Q \mu N}{\epsilon_1 \lambda \zeta}\\
    & \leq \frac{2 \nu R_Q \mu C_2}{\epsilon_1} + \frac{4 (n + m) \nu^2 R_Q \mu N}{\epsilon_1 \lambda \zeta}
\end{align*}
where the last line uses Lemma \ref{lem:slow_var} and the fact that $\tr (\bar{V}_{k-1}^{-1} ) \leq (n + m) \| \bar{V}_{k-1}^{-1} \| \leq (n + m) \lambda^{-1}$

Lastly, we show that Term II.C is non-positive by showing that, under event $E$, it holds that $\tilde{\Ec}$ is a subset of the feasible set in the optimistic SDP in the algorithm \eqref{eq:opt_sdp}.
Indeed, it holds under $E$, for all $\Sigma \in \tilde{\Ec}$ that $\tr(\Sigma) \leq \nu$, $\ip{\balpha_j, \Sigma} \leq \xi$ and,
\begin{align*}
    \Sigma\xx & = \Theta_\star \Sigma \Theta_\star^\top + W \\
    & \succeq \hat{\Theta}_k \Sigma \hat{\Theta}_k^\top + W - \| \Theta_\star \Sigma \Theta_\star^\top - \hat{\Theta}_k \Sigma \hat{\Theta}_k^\top\|I\\
    & \succeq \hat{\Theta}_k \Sigma \hat{\Theta}_k^\top + W - \eta \ip{\bar{V}_k^{-1}, \Sigma} I,
\end{align*}
where we use that $E$ holds.
Therefore, under event $E$, it holds that $\tilde{\Ec}$ is a subset of the feasible set of \eqref{eq:opt_sdp}.
It follows that $\mathrm{Term\ II.C} \leq 0$.
\end{proof}

\subsection{Proof of Lemma \ref{lem:opt_approx_first} (Term III)}
\label{sec:tthree_proof}

In this section, we give the proof of Lemma \ref{lem:opt_approx_first}, which bounds Term III.
First, we introduce the following notation to refer to the set of steady-state covariance matrices that result from the optimal policy set $\Pi_\star$ (defined in \eqref{eq:pistar}):
$$
    \Ec_\star := \left\{ \Sigma \in \bar{\Ec}_{\infty} : \Sigma = \begin{bmatrix} I \\ K \end{bmatrix} \Sigma\xx \begin{bmatrix} I \\ K \end{bmatrix}^\top, K \in \Pi_\star \right\} 
$$
Also, we define the ``intermediate'' set $\Ec_a$ as follows:
$$
\Ec_a := \left\{ \Sigma \in \bar{\Ec}_\nu : \ip{\balpha_j, \Sigma} \leq \left( \frac{\beta + D (1 + \kappa) \kappa \left(1 - \gamma\right)^{T} R_1}{ \Phi^{-1}(1 - \delta)} \right)^2 + D^2 (1 + \kappa)^2 \kappa^2 \bar{w} \gamma^{-1} (1 - \gamma)^{2 T} \right\}
$$
The key idea behind the analysis in this section, is that we will show that $\Ec_\star \subseteq \Ec_a$ (Lemma \ref{lem:E_a}) and then bound the distance between $\Ec_a$ and $\tilde{\Ec}$ (Lemma \ref{lem:maxdist}) via a local Lipschitz bound on the reciprocal of the Gaussian quantile function (Lemma \ref{lem:quant_sens}).
Thus, the minimum over $\Ec_a$ is less than the minimum over $\Ec_\star$, and therefore we can relate the cost of the best policy in $\Ec_\star$ to a policy in $\tilde{\Ec}$, establishing the bound.

\begin{lemma}
    \label{lem:E_a}
    Suppose that $\nu \geq (1 + \kappa)^2 \gamma^{-1} \kappa^2$.
    Then it holds that $\Ec_\star \subseteq \Ec_a$.
\end{lemma}
\begin{proof}
    First, Lemma \ref{lem:quant_form} tells us that all $K \in \Pi_\star$ satisfy,
    \begin{equation}
        \label{eq:gauss_form}
        \alpha_j^\top m_{t}^K + \sqrt{\ip{\balpha_j, S_{t}^K}} \Phi^{-1}(1 - \delta) \leq \beta \quad \forall j \in [J], t \in [T],
    \end{equation}
    where $S_t^K = \begin{bmatrix} I \\ K \end{bmatrix} S_{t,xx}^K \begin{bmatrix} I \\ K \end{bmatrix}^\top$ with,
    $$
        S_{t,xx}^K = \sum_{s=0}^{t-1} (A_\star + B_\star K)^s W ((A_\star + B_\star K)^s)^\top,
    $$
    and $m_t^K = \begin{bmatrix} I \\ K \end{bmatrix}(A_\star+B_\star K)^{t-1}x_1.
$.
    First, since $1 - \delta \geq 0.5$ and $\balpha_j, S_t^K \succeq 0$, it must be that $\beta - \alpha_j^\top m_{t}^K \geq 0$.
    Therefore, we can equivalently write \eqref{eq:gauss_form} as,
    $$
        \ip{\balpha_j, S_{t}^K} \leq \left( \frac{\beta - \alpha_j^\top m_{t}^K}{ \Phi^{-1}(1 - \delta)} \right)^2 \quad \forall j \in [J], t \in [T],
    $$
    Since $K$ is $(\kappa,\gamma)$-strongly stable, it holds that,
    $$
        | \alpha_j^\top m_{t} | \leq D (1 + \kappa) \kappa \left(1 - \gamma\right)^{t} R_1,
    $$
    Then, we bound the distance between $S_t^K$ and $\Sigma^K$, the (unique) matrix satisfying $\Sigma\xx^K = \Theta_\star \Sigma^K \Theta_\star^\top + W$,
    \begin{align*}
        \| \Sigma^K - S_{t}^K \| & = \left\| \begin{bmatrix} I \\ K \end{bmatrix} \sum_{s=t}^{\infty} (A_\star + B_\star K)^s W ((A_\star + B_\star K)^s)^\top \begin{bmatrix} I \\ K \end{bmatrix}^\top \right\| \\
        & \leq (1 + \kappa)^2 \kappa^2 \bar{w} \gamma^{-1} (1 - \gamma)^{2 t}
    \end{align*}
    Therefore, using the case of $t = T$, it holds for all $\Sigma \in \Ec_\star$ that,
    \begin{align*}
        \ip{\balpha_j, \Sigma} & \leq \ip{\balpha_j, S_T^K} + D^2 \| \Sigma - S_T^K \| \\
        & \leq \left( \frac{\beta - \alpha_j^\top m_{T}^K}{ \Phi^{-1}(1 - \delta)} \right)^2 + D^2 (1 + \kappa)^2 \kappa^2 \bar{w} \gamma^{-1} (1 - \gamma)^{2 T} \\
        & \leq \left( \frac{\beta + D (1 + \kappa) \kappa \left(1 - \gamma\right)^{T} R_1}{ \Phi^{-1}(1 - \delta)} \right)^2 + D^2 (1 + \kappa)^2 \kappa^2 \bar{w}  \gamma^{-1} (1 - \gamma)^{2 T}.
    \end{align*}
    Also, from Lemma 3.3 in \cite{cohen2018online}, we know that $\tr(\Sigma^K\xx) \leq \gamma^{-1} \kappa^2$ and therefore,
    $$
        \tr(\Sigma) = \tr\left( \begin{bmatrix} I \\ K \end{bmatrix} \Sigma^K\xx \begin{bmatrix} I \\ K \end{bmatrix}^\top  \right) \leq \left\| \begin{bmatrix} I \\ K \end{bmatrix} \right\|^2 \tr(\Sigma\xx) \leq (1 + \kappa)^2 \gamma^{-1} \kappa^2 \leq \nu.
    $$
    Applying the previous two inequalities implies that $\Sigma \in \Ec_a$, and thus, $\Ec_\star \subseteq \Ec_a$.
\end{proof}

Next, we give a technical lemma that gives Lipschitz bounds on the reciprocal quantile squared.
This will allow us to do a sensitivity analysis on the probability level.

\begin{lemma}
    \label{lem:quant_sens}
    Consider $x,y$ such that $0.5 < x < y \leq \bar{y} < 1$. Then it holds that,
    $$
        \frac{1}{\Phi^{-1}(x)^2} - \frac{1}{\Phi^{-1}(y)^2} \leq \frac{2 \Phi^{-1}(\bar{y})}{\phi(\Phi^{-1}(\bar{y})) \Phi^{-1}(x)^4 }(y-x),
    $$
    where $\phi$ is the standard normal pdf and $\Phi$ is the standard normal cdf.
\end{lemma}
\begin{proof}
    First, we compute the first and second derivative of $f(x) = \Phi^{-1}(x)^2$,
    $$
    f'(x) = \frac{2 \Phi^{-1}(x)}{\phi(\Phi^{-1}(x))}, \qquad f''(x) = \frac{2(1 + \Phi^{-1}(x)^2)}{\phi(\Phi^{-1}(x))^2}
    $$
    Therefore, for $x \geq 0.5$, $f$ is convex and thus,
    $$
        f(y) - f(x) \leq f'(y)(y-x) \leq f'(\bar{y})(y - x) = \frac{2 \Phi^{-1}(\bar{y})}{\phi(\Phi^{-1}(\bar{y}))} (y - x),
    $$
    where we use the fact that $f'(y)$ is increasing, and  $x < y \leq \bar{y}$.
    Then, since $x,y > 0.5$
    \begin{align*}
        \frac{1}{\Phi^{-1}(x)^2} - \frac{1}{\Phi^{-1}(y)^2} & = \frac{1}{f(x)} - \frac{1}{f(y)}\\
        & = \frac{1}{f(y)}\left( \frac{f(y)}{f(x)} - 1 \right)\\
        & \leq \frac{1}{f(y)}\left( \frac{f(x) + f'(\bar{y}) (y - x)}{f(x)} - 1 \right)\\
        & = \frac{2 \Phi^{-1}(\bar{y})}{\phi(\Phi^{-1}(\bar{y})) \Phi^{-1}(x)^2 \Phi^{-1}(y)^2 }(y-x)\\
        & \leq \frac{2 \Phi^{-1}(\bar{y})}{\phi(\Phi^{-1}(\bar{y})) \Phi^{-1}(x)^4 }(y-x),
    \end{align*}
    where the last line uses that $\Phi^{-1}(y) \geq \Phi^{-1}(x)$.
\end{proof}

Next, we establish a bound on the distance between $\Ec_a$ and $\tilde{\Ec}$.

\begin{lemma}
    \label{lem:maxdist}
    Assume the following:
    \begin{enumerate}
        \item $\rho \geq 2 \tilde{\gamma}^{-1} \log \left( \frac{\max(\beta,\epsilon/4)}{(1+\tilde{\kappa}) \tilde{\kappa} R_z D} \right)$
        \item $T \geq \gamma^{-1} \log \left( \frac{\max(\beta,\epsilon/4)}{(1+\kappa) \kappa R_1 D} \right)$
        \item $\omega \leq \delta/2$
        \item $\xi = \left( \frac{\beta - D (1 + \tilde{\kappa}) \tilde{\kappa} \left(1 - \tilde{\gamma}/2\right)^{\rho} R_z}{ \Phi^{-1}(1 - \delta + \omega)} \right)^2 - D^2 C_1$
        \item $\ip{\balpha_j, \Sigma_\star^{\safe}} \leq \xi - \epsilon_1$
    \end{enumerate}
    Then, it holds that,
    $$
        \max_{\Sigma \in \Ec_a}\dist(\Sigma, \tilde{\Ec}) \leq \frac{2 \nu}{\epsilon_1} (D^2 C_1 + C_3),
    $$
    where $C_3$ is defined in \eqref{eq:c3}.
\end{lemma}
\begin{proof}
    First note that,
    \begin{equation}
    \label{eq:dist_split}
    \begin{split}
        & \left( \beta + D (1 + \kappa) \kappa \left(1 - \gamma\right)^{T} R_1\right)^2\\
        &  = \left( \beta - D (1 + \tilde{\kappa}) \tilde{\kappa} \left(1 - \tilde{\gamma}/2\right)^{\rho} R_z + D \left( (1 + \tilde{\kappa}) \tilde{\kappa} \left(1 - \tilde{\gamma}/2\right)^{\rho} R_z + (1 + \kappa) \kappa \left(1 - \gamma\right)^{T} R_1 \right)\right)^2\\
        & \leq (\beta - D (1 + \tilde{\kappa}) \tilde{\kappa} \left(1 - \tilde{\gamma}/2\right)^{\rho} R_z)^2 + \left( 2 \beta + \epsilon \right) D \left( (1 + \tilde{\kappa}) \tilde{\kappa} \left(1 - \tilde{\gamma}/2\right)^{\rho} R_z + (1 + \kappa) \kappa \left(1 - \gamma\right)^{T} R_1 \right),
    \end{split}
    \end{equation}
    where the inequality uses that it holds that $(x + y)^2 = x^2 + 2 x y + y^2 \leq x^2 + (2 \bar{x} + \bar{y}) y$ for $0 \leq x \leq \bar{x}$, $0 \leq y \leq \bar{y}$, that $\beta \geq D (1 + \tilde{\kappa}) \tilde{\kappa} \left(1 - \tilde{\gamma}/2\right)^{\rho} R_z$ from the assumption on $\rho$.
    Then,
    \begin{align*}
         & \left( \frac{\beta + D (1 + \kappa) \kappa \left(1 - \gamma\right)^{T} R_1}{ \Phi^{-1}(1 - \delta)} \right)^2\\
          & \leq \left( \frac{\beta + D (1 + \kappa) \kappa \left(1 - \gamma\right)^{T} R_1}{ \Phi^{-1}(1 - \delta + \omega)} \right)^2 + \frac{2 \Phi^{-1}(1-\delta/2) \left( \beta + D (1 + \kappa) \kappa \left(1 - \gamma\right)^{T} R_1 \right)^2 \omega}{\phi(\Phi^{-1}(1-\delta/2)) \Phi^{-1}(1-\delta)^4} \tag{a} \label{eq:dist:a}\\
          & \leq \left( \frac{\beta + D (1 + \kappa) \kappa \left(1 - \gamma\right)^{T} R_1}{ \Phi^{-1}(1 - \delta + \omega)} \right)^2 + \frac{2 \Phi^{-1}(1-\delta/2) 4 \beta^2 \omega}{\phi(\Phi^{-1}(1-\delta/2)) \Phi^{-1}(1-\delta)^4}\tag{b} \label{eq:dist:b}\\
          & \leq \left( \frac{\beta - D (1 + \tilde{\kappa}) \tilde{\kappa} \left(1 - \tilde{\gamma}/2\right)^{\rho} R_z}{ \Phi^{-1}(1 - \delta + \omega)} \right)^2 + \frac{\left( 2 \beta + \epsilon \right) \left( D (1 + \tilde{\kappa}) \tilde{\kappa} \left(1 - \tilde{\gamma}/2\right)^{\rho} R_z + D (1 + \kappa) \kappa \left(1 - \gamma\right)^{T} R_1 \right)}{ \Phi^{-1}(1 - \delta + \omega)^2} \tag{c} \label{eq:dist:c}\\
          & \qquad + \frac{2 \Phi^{-1}(1-\delta/2) 4 \beta^2 \omega}{\phi(\Phi^{-1}(1-\delta/2)) \Phi^{-1}(1-\delta)^4} \\
          & \leq \left( \frac{\beta - D (1 + \tilde{\kappa}) \tilde{\kappa} \left(1 - \tilde{\gamma}/2\right)^{\rho} R_z}{ \Phi^{-1}(1 - \delta + \omega)} \right)^2 + \frac{\left( 2 \beta + \epsilon \right) \left(  D (1 + \tilde{\kappa}) \tilde{\kappa} \left(1 - \tilde{\gamma}/2\right)^{\rho} R_z + D (1 + \kappa) \kappa \left(1 - \gamma\right)^{T} R_1 \right)}{ \Phi^{-1}(1 - \delta)^2} \tag{d} \label{eq:dist:d}\\
          & \qquad + \frac{2 \Phi^{-1}(1-\delta/2) 4 \beta^2 \omega}{\phi(\Phi^{-1}(1-\delta/2)) \Phi^{-1}(1-\delta)^4},
    \end{align*}
    where \eqref{eq:dist:a} uses Lemma \ref{lem:quant_sens}, \eqref{eq:dist:b} uses that $\beta \geq D (1 + \tilde{\kappa}) \tilde{\kappa} \left(1 - \tilde{\gamma}/2\right)^{\rho} R_z$ from the assumption on $\rho$, \eqref{eq:dist:c} uses \eqref{eq:dist_split}, and \eqref{eq:dist:d} uses that $(\Phi^{-1}(x))^2$ is increasing for $x \geq 0.5$ and therefore $(\Phi^{-1}(x))^{-2}$ is decreasing.
    Therefore, for any $\Sigma_a \in \Ec_a$, it holds that,
    \begin{align*}
        \ip{\balpha_j, \Sigma_a} & \leq \left( \frac{\beta + D (1 + \kappa) \kappa \left(1 - \gamma\right)^{T} R_1}{ \Phi^{-1}(1 - \delta)} \right)^2 + D^2 (1 + \kappa)^2 \kappa^2 \bar{w} \gamma^{-1} (1 - \gamma)^{2 T} \\
        & \leq \left( \frac{\beta - D (1 + \tilde{\kappa}) \tilde{\kappa} \left(1 - \tilde{\gamma}/2\right)^{\rho} R_z}{ \Phi^{-1}(1 - \delta + \omega)} \right)^2 + C_3\\
        & = \xi + D^2 C_1 + C_3
    \end{align*}
    Then, let $\Sigma_\zeta = \zeta\Sigma_a + (1 - \zeta) \Sigma_\star^\safe$ and $\zeta = \frac{\epsilon_1}{\epsilon_1 + D^2 C_1 + C_3}$.
    Since $\ip{\balpha_j, \Sigma_\star^{\safe}} \leq \xi - \epsilon_1$ by assumption,
    \begin{align*}
        \ip{\balpha_j, \Sigma_\zeta} & = \zeta \ip{\balpha_j, \Sigma_a} + (1 - \zeta) \ip{\balpha_j, \Sigma_{\star}^\safe}\\
        & \leq \zeta \left( \xi + D^2 C_1 + C_3 \right) + (1 - \zeta) (\xi - \epsilon_1)\\
        & = \xi + \zeta (D^2 C_1 + C_3) - (1 - \zeta) \epsilon_1\\
        & = \xi,
    \end{align*}
    and therefore $\Sigma_\zeta \in \tilde{\Ec}$.
    Thus,
    $$
        \dist(\Sigma_a, \tilde{\Ec}) \leq \|\Sigma_a -  \Sigma_\zeta \| = (1 - \zeta) \| \Sigma_a - \Sigma^\safe \| \leq \frac{2 \nu(D^2 C_1 + C_3)}{\epsilon_1 + D^2 C_1 + C_3} \leq \frac{2 \nu}{\epsilon_1} (D^2 C_1 + C_3)
    $$
\end{proof}

Finally, we give the proof of Lemma \ref{lem:opt_approx_first}.

\begin{proof}[Proof of Lemma \ref{lem:opt_approx_first}]
    First, note that,
    \begin{align*}
        \Eb \left[z_t^{\star,\bw} (z_t^{\star,\bw})^\top \right] & = S_t^\star + m_t m_t^\top\\
        & \succeq S_t^\star\\
        & \succeq \Sigma^\star - (1 + \kappa)^2 \kappa^2 \bar{w} \gamma^{-1} (1 - \gamma)^{2 t} I
    \end{align*}
    Furthermore, let $\Sigma_a \in \argmin_{\Sigma \in \Ec_a} \left \ip{ \Sigma, \mathbf{Q} \right }$.
    Then, it holds that,
    \begin{align*}
        J_\star & = \Eb \sum_{t=1}^{T} \| z_t^{K^\star,\bw} \|_{\mathbf{Q}}^2\\
        & = \sum_{t=1}^{T} \left \ip{ \Eb \left[z_t^{K^\star,\bw} (z_t^{K^\star,\bw})^\top \right], \mathbf{Q} \right }\\
        & \geq \left \ip{ \Sigma^\star, \mathbf{Q} \right } T - R_Q (1 + \kappa)^2 \kappa^2 \gamma^{-2}\bar{w},\\
        & \geq \left \ip{ \Sigma_a, \mathbf{Q} \right } T - R_Q (1 + \kappa)^2 \kappa^2 \gamma^{-2}\bar{w},
    \end{align*}
    where the last line uses that $\Ec_\star \subseteq \Ec_a$ from Lemma \ref{lem:E_a}.
    Then, from Lemma \ref{lem:maxdist}, there exists $\tilde{\Sigma} \in \tilde{\Ec}$ such that $\| \Sigma_a - \tilde{\Sigma} \| \leq \frac{2 \nu}{\epsilon_1} (D^2 C_1 + C_3)$,
    \begin{align*}
        \ip{ \Sigma_a, \mathbf{Q} } \geq \ip{ \tilde{\Sigma}, \mathbf{Q} } - R_Q \frac{2 \nu}{\epsilon_1} (D^2 C_1 + C_3) \geq \min_{\Sigma \in \tilde{\Ec}} \ip{\mathbf{Q},\Sigma} - R_Q \frac{2 \nu}{\epsilon_1} (D^2 C_1 + C_3)
    \end{align*}
\end{proof}

\subsection{Proof of Lemma \ref{lem:const_sat_first}}
\label{sec:const_sat_proof}

In this section, we give the proof of Lemma \ref{lem:const_sat_first}, which establishes the constraint satisfaction guarantee.
In order to do so, we give technical lemmas that provides sufficient conditions for constraint satisfaction (Lemma \ref{lem:cov_const}) and show that the target covariance is in the pessimistic set (Lemma \ref{lem:bar_to_t}).
Thus, using the definition of the pessimistic set, Lemma \ref{lem:const_cond} shows that the constraint is satisfied during the main portion of the algorithm (i.e. $t \geq \tau_1$).
Finally, to get the constraint satisfaction guarantees, we combine Lemma \ref{lem:const_cond} with the constraint satisfaction guarantees during the initialization phase given in Lemma \ref{lem:exploration}.

We introduce some additional notation for the pessimistic set in the zeroth phase (even though it is not used in the algorithm):
$$
    \Ec_0^p := \{ \Sigma \succeq 0 : \langle \balpha_j, \Sigma \rangle + \mu \langle \bar{V}_{0}^{-1}, \Sigma \rangle \leq \xi, \forall j \in [J] \} 
$$
Next, we give a technical lemma that provides conditions for ensuring that the chance-constraint is satisfied.

\begin{lemma}
    \label{lem:cov_const}
    Suppose that $z_t|\Fc_\tau \sim \Nc(m_{t|\tau}, S_{t|\tau})$.
    Then, if $\omega' \in (0,0.5)$, then,
    $$
        \Pb\left(\ip{\balpha_j, S_{t|\tau}} \leq \left( \frac{\beta - \alpha_j^\top m_{t|\tau}}{\Phi^{-1}(1 - \omega')} \right)^2 , \alpha_j^\top m_{t|\tau} \leq \beta \right) \geq 1 - \omega, \quad \implies \quad \Pb( \alpha_j^\top z_t \leq \beta) \geq 1 - \omega' - \omega
    $$
\end{lemma}
\begin{proof}
    Let $F$ be the event that $\ip{\balpha_j, S_{t|\tau}} \leq \left( \frac{\beta - \alpha_j^\top m_{t|\tau}}{\Phi^{-1}(1 - \omega')} \right)^2$ and $\alpha_j^\top m_{t|\tau} \leq \beta$.
    First, since we assume that $\omega' \in (0,0.5]$, it holds under $F$ that,
    $$
        \alpha_j^\top m_{t|\tau} + \sqrt{\ip{\balpha_j, S_{t|\tau}}} \Phi^{-1}(1 - \omega') \leq \beta
    $$
    Therefore, due to Lemma \ref{lem:quant_form}, it follows that,
    $$
        \Pb( \alpha_j^\top z_t > \beta | \Fc_\tau) \Ib_{F}  < \omega'
    $$
    Finally,
    \begin{align*}
        \Pb( \alpha_j^\top z_t > \beta ) & = \Eb\left[ \Pb( \alpha_j^\top z_t > \beta | \Fc_\tau) \right]\\
        & = \Eb\left[ \Pb( \alpha_j^\top z_t > \beta | \Fc_\tau)\Ib_{F} +  \Pb( \alpha_j^\top z_t > \beta | \Fc_\tau)\Ib_{F^c}\right]\\
        & < \omega' + \Eb\left[ \Pb( \alpha_j^\top z_t > \beta | \Fc_\tau)\Ib_{F^c}\right] \\
        & \leq \omega' + \Pb(F^c)\\
        & < \omega' + \omega
    \end{align*}
\end{proof}

Next, we give  a lemma establishing that $\Sigma_t$ is within the pessimistic set of the current phase.

\begin{lemma}
    \label{lem:bar_to_t}
    If $\Sigbar_{k'} \in \Ec_{k'}^p$ for all $k' \in [0,k]$, then it follows that $\Sigma_t \in \Ec_k^p$ for all $t \in [\tau_k, \tau_{k+1}]$ and $k \geq 1$.
\end{lemma}
\begin{proof}
    First, we show that $\Ec_{k}^p \subseteq \Ec_{k+1}^p$.
    Indeed, it holds for all $\Sigma \in \Ec_{k}^p$ that,
    $$
        \ip{\balpha_j, \Sigma} + \mu \ip{\bar{V}_{k}^{-1}, \Sigma} = \ip{\balpha_j, \Sigma} + \mu\ \tr(\Sigma^{1/2} \bar{V}_{k}^{-1}  \Sigma^{1/2}) \leq \ip{\balpha_j, \Sigma} + \mu \ip{\bar{V}_{k-1}^{-1}, \Sigma} \leq \xi,
    $$
    where we use the fact that $\Sigma \succeq 0$ and $\bar{V}_k \succeq \bar{V}_{k-1}$.
    Therefore, $\Sigma \in \Ec^p_{k+1}$ and, in turn, $\Ec^p_{k} \subseteq \Ec^p_{k+1}$.

    Therefore, it holds for all $k' \leq k$,
    $$
        \Sigbar_{k'} \in \Ec_{k'}^p \subseteq \Ec_k^p.
    $$
    Then, from Lemma \ref{lem:slow_var},
    $$
        \Sigma_t \in \conv\{\Sigbar_0,...,\Sigbar_k\} \subseteq \Ec_k^p,
    $$
    where we use the fact that $\Ec_k^p$ is convex.
\end{proof}

Next, we show that the constraint is satisfied after the exploration phase.

\begin{lemma}
    \label{lem:const_cond}
    In addition to the assumptions in Lemma \ref{lem:main_cov} and Lemma \ref{lem:safe_margin}, suppose that the following hold:
    \begin{enumerate}
        \item $\xi = \left( \frac{\beta - D (1 + \tilde{\kappa}) \tilde{\kappa} \left(1 - \tilde{\gamma}/2\right)^{\rho} R_z}{ \Phi^{-1}(1 - \delta + \omega)} \right)^2 - D^2 C_1$
        \item $\mu = 2 (1 + \tilde{\kappa})^2 \tilde{\kappa}^2  D^2 \eta$
        \item $\rho \geq 2 \tilde{\gamma}^{-1} \log \left( \frac{\beta}{(1+\tilde{\kappa}) \tilde{\kappa} R_z D} \right)$
        \item $\omega < \delta$
        \item $\Pb(E) \geq 1 - \omega$
    \end{enumerate}
    Then, for all $t \geq \tau_0$, it holds that $\Pb( \alpha_j^\top z_t \leq \beta) \geq 1 - \delta$ for all $j \in [J]$. 
\end{lemma}
\begin{proof}
    In this proof, we show that if $E$ holds, then $\Sigma_t \in \Ec_{k_t}^p$ for all $t \geq \tau_1$, and then show that this implies the constraint satisfaction guarantee via Lemma \ref{lem:cov_const}, given that $\Pb(E) \geq 1 - \omega$ by assumption.
    
    In order to show that $\Sigma_t \in \Ec_{k_t}^p$, we first show that $\bar{\Sigma}_k \in \Ec_k^p$ for all $k \in [0,N]$.
    Indeed, under $E$, $\Sigma^\safe_k \in \Ec_k^p$ for all $k \in \{0,...,N\}$, and therefore the safe scaling step (line \ref{lne:scale}) is well-defined in the sense that there exists $\phi = 0$ such that $\phi \Sigma_k^o + (1 - \phi) \Sigma_k^\safe \in \Ec_k^p$.
    It follows that $\bar{\Sigma}_k \in \Ec_{k}^p$ for all $k \in [0,N]$.
    Therefore, for any given $t \geq \tau_0$, it holds that $\bar{\Sigma}_{k'} \in \Ec_{k_t}^p$ for all $k' \in \{0,..., k_t\}$.
    Thus, Lemma \ref{lem:bar_to_t} says that $\Sigma_t \in \Ec_{k_t}^p$ for all $t \geq \tau_0$, if $E$ holds.
    

    Under $E$, it follows that,
    \begin{equation}
        \label{eq:balph}
    \begin{split}
        \ip{\balpha_j, S_{t|t-\rho}} & \leq \ip{\balpha_j,\Sigma_t} + 2 (1 + \tilde{\kappa})^2 \tilde{\kappa}^2  D^2 \eta \ip{\bar{V}_{k-1}^{-1}, \Sigma_t} + D^2 C_1 \\
        & \leq \xi + D^2 C_1 \\
        & = \left( \frac{\beta - D (1 + \tilde{\kappa}) \tilde{\kappa} \left(1 - \tilde{\gamma}/2\right)^{\rho} R_z}{ \Phi^{-1}(1 - \delta + \omega)} \right)^2\\
        & \leq \left( \frac{\beta - \alpha_j^\top m_{t|t-\rho}}{\Phi^{-1}(1 - \delta + \omega)} \right)^2
    \end{split}
    \end{equation}
    where the first inequality uses Lemma \ref{lem:main_cov}, the second inequality uses the choice of $\mu$, and third inequality uses Lemma \ref{lem:main_cov} and the condition on $\rho$, i.e. 
    $$
        \alpha_j^\top m_{t|t-\rho} \leq D (1 + \tilde{\kappa}) \tilde{\kappa} \left(1 - \tilde{\gamma}/2\right)^{\rho} R_z \leq \beta,
    $$

    Then, we show that the conditions of Lemma \ref{lem:cov_const} hold in order to get the constraint satisfaction guarantee.
    First, since $K_t, U_t$ are $\Fc_{t-\rho}$-measurable (due to the use of delayed information for estimation), it holds that $z_t|\Fc_{t - \rho} \sim \Nc(m_{t|t-\rho}, S_{t|t-\rho})$.
    Also, from \eqref{eq:balph}, and the assumption that $\Pb(E) \geq 1 - \omega$ it holds that
    $$
        \Pb \left( \ip{\balpha_j, S_{t|t-\rho}} \leq \left( \frac{\beta - \alpha_j^\top m_{t|t-\rho}}{\Phi^{-1}(1 - \delta + \omega)} \right)^2, \alpha_j^\top m_{t|t-\rho} \leq \beta \right) \geq 1 - \omega
    $$
    Therefore, we have verified the conditions of Lemma \ref{lem:cov_const} and thus,
    $$
        \Pb( \alpha_j^\top z_t \leq \beta) \geq 1 - \delta, \ \forall j \in [J]
    $$
\end{proof}


Finally, we wrap everything up by combining Lemma \ref{lem:const_cond} with Lemma \ref{lem:exploration}.

\begin{proof}[Proof of Lemma \ref{lem:const_sat_first}]
    Lemma \ref{lem:exploration} gives the guarantee for $t \leq \tau_0 - 1$ and Lemma \ref{lem:const_cond} gives the guarantee for $t \geq \tau_0$.
\end{proof}

\subsection{Proof of Lemma \ref{lem:main_cov}}
\label{sec:main_cov}

In this section, we give the proof of Lemma \ref{lem:main_cov}, which bounds the conditional covariance in terms of the policy.
We first give technical lemmas establishing conditions under which the closed-loop system is sequentially strongly-stable (Lemma \ref{lem:strong_stab}), a bound on the conditional covariance holds (Lemma \ref{lem:cov_approx}) and a bound on the conditional mean holds (Lemma \ref{lem:mean}).
Then, we show how the algorithm ensures that the conditions of these lemmas are satisfied.
In particular, we show that the phase covariance satisfies an approximate steady-state condition (Lemma \ref{lem:stat_cond}), that the target covariance gets close to the phase covariance when the phase is sufficiently long (Lemma \ref{lem:phase_min}), and therefore that the target covariance satisfies an approximate steady-state condition (Lemma \ref{lem:real_stat}).
Next, we show that $\Sigma^{\safe}_k$ (the estimated covariance of the zero policy) is bounded (Lemma \ref{lem:approx_stable}), and therefore that the phase covariance and target covariance is bounded (Lemma \ref{lem:nu}).
This ultimately yields the proof of Lemma \ref{lem:main_cov}.

First, we introduce some notation.
The block notation for $S_{t|\tau}$ is as specified,
$$
    S_{t|\tau} = \begin{bmatrix}
        S_{t|\tau}^{xx} & S_{t|\tau}^{xu}\\
        S_{t|\tau}^{ux} & S_{t|\tau}^{uu}
    \end{bmatrix}
$$
We also use the notation $A_t = (A + B K_t)$.
We also give the definition of sequential strong stability from \cite{cohen2018online}.

\begin{definition}[Sequential Strong Stability, \cite{cohen2018online}]
\label{def:seq_str_stab}
    We say that a sequence of linear controllers $(K_t)_{t \in [T]}$ is $(\kappa,\gamma)$-\emph{sequentially strongly-stable} for system $(A,B)$ if there exists sequences $(H_t)_{t \in [T]}$ and $(L_t)_{t \in [T]}$ such that $A + BK_t = H_t L_t H_t^{-1}$ and:
    \begin{enumerate}
        \item $\| L_t \| \leq 1 - \gamma$ and $\| K_t \| \leq \kappa$,
        \item $\| H_t \| \| H_t^{-1} \| \leq \kappa$,
        \item $\| H_{t+1}^{-1} H_t \| \leq 1 + \gamma/2$
    \end{enumerate}
\end{definition}

Next, we give conditions on the target covariance in order for the system to be sequentially strongly stable.
This will be critical for showing that the actual covariance goes to the approximate steady-state covariance.

\begin{lemma}
    \label{lem:strong_stab}
    Consider an interval of natural numbers $\Tc$.
    Suppose that the following holds for all $t \in \Tc$:
    \begin{enumerate}
        \item $\Sigma_t \succeq 0$
        \item $\tr (\Sigma_t) \leq \nu$
        \item $\Sigma_{t,xx} \succeq \Theta \Sigma_t \Theta^\top + W - \bar{\eta} I$ with $\bar{\eta} < \sigma$ where $\sigma I \preceq W$
        \item $\| \Sigma_{t+1} - \Sigma_t \| \leq \frac{(\sigma - \bar{\eta})^2}{2 \nu}$
    \end{enumerate}
    Then, consider the resulting controllers $K_t = \Sigma_{t,ux} \Sigma_{t,xx}^{-1}$.
    It holds that the sequence of controllers $(K_t)_{t\in\Tc}$ are $(\kappa, \gamma)$-sequentially-strongly stable with $(\kappa, \gamma) = (\frac{\nu}{\sigma - \bar{\eta}}, \frac{\sigma - \bar{\eta}}{2 \nu})$.
\end{lemma}
\begin{proof}
    This proof generalizes Theorem 4.2 and Lemma 4.3 in \cite{cohen2018online} to the case of approximate stationarity condition.
    First, note that under the conditions of the lemma, for all $t \in [T]$,
    $$ 
        \Sigma_{t,xx} \succeq \Theta \Sigma_t \Theta^\top + W - \bar{\eta} I \succeq \bsig I \succ 0,
    $$
    where we use $\bsig = \sigma - \bar{\eta}$.
    Then, it follows that,
    $$
    \Sigma_t \succeq \begin{bmatrix}\Sigma_{t,xx} & \Sigma_{t,xu} \\ \Sigma_{t,ux} & \Sigma_{t,ux} \Sigma_{t,xx}^{-1} \Sigma_{t,xu} \end{bmatrix} = \begin{bmatrix} I\\ K_t \end{bmatrix} \Sigma_{t,xx} \begin{bmatrix} I\\ K_t \end{bmatrix}^\top =: \Sigma_t',
    $$
    since the condition $\Sigma_t \succeq 0$ and $\Sigma_{t,xx} \succ 0$ implies that the Schur complement of $\Sigma_t$ is psd.
    It follows that,
    \begin{equation}
    \label{eq:pd_sig}
    \begin{split}
        \Sigma_{t,xx} & \succeq \Theta \Sigma_t \Theta^\top + W - \bar{\eta} I\\
        & \succeq \Theta \Sigma_t' \Theta^\top + W - \bar{\eta} I\\
        & = (A + B K) \Sigma_{t,xx} (A + B K_t)^\top + W - \bar{\eta} I\\
        & \succeq (A + B K) \Sigma_{t,xx} (A + B K_t)^\top + \bsig I
    \end{split}
    \end{equation}
    Then, we show that the sequence of $K_t$ is $(\kappa, \gamma) = (\frac{\nu}{\bsig}, \frac{\bsig}{2 \nu})$ sequentially strongly-stable.
    We choose the similarity transform $A + B K_t = H_t L_t H_t^{-1}$, where $L_t = \Sigma_{t,xx}^{-1/2} (A + B K_t) \Sigma_{t,xx}^{1/2}$ and $H_t = \Sigma_{t,xx}^{1/2}$.
    Multiplying \eqref{eq:pd_sig} by $\Sigma_{t,xx}^{-1/2}$ on both sides yields,
    \begin{align*}
        I & \succeq \Sigma_{t,xx}^{-1/2} (A + B K_t) \Sigma_{t,xx} (A + B K_t)^\top \Sigma_{t,xx}^{-1/2} + \bsig \Sigma_{t,xx}^{-1}\\
        & \succeq \Sigma_{t,xx}^{-1/2} (A + B K_t) \Sigma_{t,xx} (A + B K_t)^\top \Sigma_{t,xx}^{-1/2} +  I \bsig /\nu,
    \end{align*}
    It follows that,
    $$
        L_t L_t^\top \preceq (1 - \bsig/ \nu) I \ \iff \ \| L_t \|^2 \leq 1 - \bsig/ \nu \ \implies \ \| L_t \| \leq 1 - \bsig/ (2 \nu) = 1 - \gamma.
    $$
    Also, $\| H_t \| \| H_t^{-1} \| \leq \sqrt{\nu/\bsig} \leq \nu/\bsig = \kappa$ (given that $\nu \geq \bsig$).
    Then, to bound $K_t$, note that,
    $$
        \| K_t \|^2_F = \tr(K_t K_t^\top) \leq \bsig^{-1} \tr(K_t \Sigma_{t,xx} K^\top) = \bsig^{-1} \tr(\Sigma_{t,ux} \Sigma_{t,xx}^{-1} \Sigma_{t,xu}) \leq \bsig^{-1}\tr(\Sigma_{t,uu}) \leq \nu /\bsig = \kappa.
    $$
    Therefore, since $\kappa \geq 1$, it holds that $\| K_t \| \leq \| K_t \|_F \leq \sqrt{\kappa} \leq \kappa$.
    Finally, it holds that,
    \begin{align*}
        \| H_{t+1}^{-1} H_t \|^2 & = \| \Sigma_{t+1,xx}^{-1/2} \Sigma_{t,xx}^{1/2} \|^2 \\
        & = \| \Sigma_{t+1,xx}^{-1/2} \Sigma_{t,xx} \Sigma_{t+1,xx}^{-1/2} \| \\
        & = \| \Sigma_{t+1,xx}^{-1/2} \Sigma_{t+1,xx} \Sigma_{t+1,xx}^{-1/2} \| + \| \Sigma_{t+1,xx}^{-1/2} (\Sigma_{t,xx} - \Sigma_{t+1,xx})\Sigma_{t+1,xx}^{-1/2} \|\\
        & \leq 1 + \| \Sigma_{t+1,xx}^{-1} \| \| \Sigma_{t,xx} - \Sigma_{t+1,xx} \|\\
        & \leq 1 + \frac{1}{\bar{\sigma}} \frac{\bar{\sigma}^2}{2 \nu} = 1 + \gamma.
    \end{align*}
    It follows that,
    $$
        \| H_{t+1}^{-1} H_t \| \leq \sqrt{1 + \gamma} \leq 1 + \gamma/2
    $$
\end{proof}

Next, we give a key technical lemma (Lemma \ref{lem:cov_approx}), which gives conditions for the conditional covariance to be bounded in terms of the target covariance.

\begin{lemma}
    \label{lem:cov_approx}
    Suppose that the following hold:
    \begin{enumerate}
        \item $\Sigma_{s}^{xx} \succeq \Theta \Sigma_s \Theta^\top + W - \bar{\eta}_s I$ for all $s \in [\tau,t]$,
        \item $\| \Sigma_t \| \leq \nu$  for all $t$,
        \item $K_\tau,K_{\tau + 1},...,K_t$ are sequentially strongly-stable,
        \item $K_s, U_s$ are $\Fc_\tau$-measurable for all $s \in [\tau,t]$.
    \end{enumerate}
    Then, it holds that,
    \begin{align*}
        S_{t|\tau} \preceq \Sigma_t + (1 + \kappa)^2 \kappa^2 \left(  \nu \exp( - \gamma (t - \tau)) +  \sum_{s=\tau}^{t} \| \Sigma_{s}^{xx} - \Sigma_{s+1}^{xx} \| +  \sum_{s=\tau}^{t} \bar{\eta}_s \right)I,
    \end{align*}
\end{lemma}
\begin{proof}
First, note that the state update is of the form,
$$
    x_{s+1} = A x_s + B (K_s x_s + v_s) + w_s = A_s x_s + B v_s + w_s, \quad v_s \sim \Nc(0,U_t), w_s \sim \Nc(0, W)
$$
Then, since $K_s$ and $U_s$ are $\Fc_\tau$-measurable for all $s \in [\tau,T]$:
\begin{align}
    & S_{s | \tau} = \begin{bmatrix} I\\ K_s \end{bmatrix} S_{s|\tau}^{xx} \begin{bmatrix} I\\ K_s \end{bmatrix}^\top + \begin{bmatrix} 0 & 0 \\ 0 & U_s \end{bmatrix}\\
    & \implies S_{s+1 | \tau}^{xx} = \Theta S_{s | \tau} \Theta^\top + W = A_s S_{s|\tau}^{xx} A_s^\top + B U_s B^\top + W. \label{eq:cond_sig}
\end{align}
Also, note that,
\begin{align*}
    \Sigma_s & = \begin{bmatrix} \Sigma_s^{xx} & \Sigma_s^{xu} \\ \Sigma_s^{ux} & \Sigma_s^{uu} \end{bmatrix}\\
    & =  \begin{bmatrix} \Sigma_s^{xx} & \Sigma_s^{xu} \\ \Sigma_s^{ux} & \Sigma_s^{ux}(\Sigma_s^{xx})^{-1} \Sigma_s^{xu} \end{bmatrix} + \begin{bmatrix}0 & 0\\ 0 & \Sigma_s^{uu} - \Sigma_s^{ux}(\Sigma_s^{xx})^{-1} \Sigma_s^{xu} \end{bmatrix}\\
    & = \begin{bmatrix} I\\ K_s \end{bmatrix} \Sigma_s^{xx} \begin{bmatrix} I\\ K_s \end{bmatrix}^\top + \begin{bmatrix} 0 & 0 \\ 0 & U_s \end{bmatrix}
\end{align*}
Therefore, the condition in the lemma implies that, 
\begin{align}
    \Sigma_s^{xx} & \succeq \Theta \Sigma_s \Theta^\top + W - \bar{\eta}_s I \\
    & = A_s \Sigma_s^{xx} A_s^\top + B U_s B^\top + W - \bar{\eta}_s I \label{eq:sig_fixed}
\end{align}
Then, take $\Delta_t = S_{t|\tau}^{xx} - \Sigma_{t}^{xx}$.
Therefore, subtracting \eqref{eq:sig_fixed} from \eqref{eq:cond_sig} and adding $\Sigma_{t}^{xx} - \Sigma_{t+1}^{xx}$ to both sides:
\begin{align*}
    \Delta_{t+1} &  \preceq A_t \Delta_t A_t^\top + (\Sigma_{t}^{xx} - \Sigma_{t+1}^{xx}) + \bar{\eta}_{t} I\\
    & \preceq A_t \Delta_t A_t^\top + \| \Sigma_{t}^{xx} - \Sigma_{t+1}^{xx} \| I + \bar{\eta}_t I\\
    & \ \vdots\\
    & \preceq (\Pi_{i=t}^{\tau} A_i) \Delta_\tau (\Pi_{i=t}^{\tau} A_i)^\top + \sum_{s=\tau}^{t} \left( \| \Sigma_{s}^{xx} - \Sigma_{s+1}^{xx} \| + \bar{\eta}_s \right) (\Pi_{i=t-1}^{s} A_i)  (\Pi_{i=t-1}^{s} A_i)^\top
\end{align*}
Then, if $K_t$ is sequentially-strongly stable (as defined in \cite{cohen2018online}),
\begin{equation}
    \label{eq:a_prod}
\begin{split}
    \| \Pi_{i=t}^{\tau} A_i \| & = \| H_t L_t H_{t}^{-1} H_{t-1} L_{t-1} H_{t-1}^{-1} \ \hdots \ H_{\tau} L_{\tau} H_{\tau}^{-1} \|\\
    & \leq \| H_t \| \| H_\tau^{-1} \| \left (\Pi_{i = \tau}^{t-1} \|H_{i+1}^{-1} H_{i} \| \right ) \left (\Pi_{i = \tau}^{t} \|L_i \| \right )\\
    & \leq \kappa (1 + \gamma/2)^{t - \tau} (1 - \gamma)^{t - \tau + 1}\\
    & \leq \kappa \left( (1 + \gamma/2) (1 - \gamma) \right)^{t - \tau}\\
    & \leq \kappa \left(1 - \gamma/2 \right)^{t - \tau}
\end{split}
\end{equation}
It follows that,
\begin{align*}
    & \left\| (\Pi_{i=t}^{\tau} A_i) \Delta_\tau (\Pi_{i=t}^{\tau} A_i)^\top + \sum_{s=\tau}^{t} (\Pi_{i=t-1}^{s} A_i) \left( \| \Sigma_{s}^{xx} - \Sigma_{s+1}^{xx} \| + \bar{\eta}_s \right) (\Pi_{i=t-1}^{s} A_i)^\top \right\|\\
    & \leq \left \| \Pi_{i=t}^{\tau} A_i \right \|^2 \| \Delta_\tau \| + \sum_{s=\tau}^{t} \left(\| \Sigma_{s}^{xx} - \Sigma_{s+1}^{xx} \| + \bar{\eta}_s \right)\| \Pi_{i=t-1}^{s} A_i \|^2\\
    & \leq \kappa^2 \left(1 - \gamma/2 \right)^{2(t - \tau)} \| \Sigma_\tau \| + \kappa^2 \sum_{s=\tau}^{t} \left(\| \Sigma_{s}^{xx} - \Sigma_{s+1}^{xx} \| + \bar{\eta}_s \right)\left(1 - \gamma/2 \right)^{2(t-s + 1)}\\
    & \leq \nu \kappa^2 \exp( - \gamma (t - \tau)) + \kappa^2 \sum_{s=\tau}^{t} \left(\| \Sigma_{s}^{xx} - \Sigma_{s+1}^{xx} \| + \bar{\eta}_s \right),
\end{align*}
where we used the fact that $S_{\tau|\tau} = 0$ and $\| \Sigma_\tau \| \leq \nu$.
Finally, putting everything together:
\begin{align*}
    S_{t|\tau} & = \begin{bmatrix} I\\ K_t \end{bmatrix} S_{t|\tau}^{xx} \begin{bmatrix} I\\ K_t \end{bmatrix}^\top + \begin{bmatrix} 0 & 0 \\ 0 & U_t \end{bmatrix}\\
    & \preceq \begin{bmatrix} I\\ K_t \end{bmatrix} \Sigma_{t}^{xx} \begin{bmatrix} I\\ K_t \end{bmatrix}^\top + \begin{bmatrix} 0 & 0 \\ 0 & U_t \end{bmatrix}\\
    & \qquad + (1 + \kappa)^2 \kappa^2 \left(  \nu \exp( - \gamma (t - \tau)) + \sum_{s=\tau}^{t} \left(\| \Sigma_{s}^{xx} - \Sigma_{s+1}^{xx} \| + \bar{\eta}_s \right) \right)I\\
    & = \Sigma_t + (1 + \kappa)^2 \kappa^2 \left(  \nu \exp( - \gamma (t - \tau)) + \sum_{s=\tau}^{t} \left(\| \Sigma_{s}^{xx} - \Sigma_{s+1}^{xx} \| + \bar{\eta}_s \right) \right)I
\end{align*}
\end{proof}

Next, we give a lemma that provides conditions for the conditional expectation to be bounded.

\begin{lemma}
    \label{lem:mean}
    Suppose the following hold:
    \begin{enumerate}
        \item $K_1,K_2,...,K_T$ are sequentially strongly-stable,
        \item $K_s, U_s$ are $\Fc_\tau$-measurable for all $s \in [\tau,t]$,
        \item $\| x_\tau \| \leq R_x$.
    \end{enumerate}
    Then, it holds that
    $$
        \| m_{t|\tau} \| \leq (1 + \kappa) \kappa \left(1 - \gamma/2 \right)^{t - \tau} R_x
    $$
\end{lemma}
\begin{proof}
    It holds that,
    \begin{align*}
        \Eb \left[ x_t \ | \ \Fc_\tau \right] & = \Eb \left[ A x_{t-1} + B (K_{t-1} x_{t-1} + v_{t-1}) + w_{t-1} \ | \ \Fc_\tau \right]\\
        & = A_{t-1} \Eb \left[ x_{t-1} \ | \ \Fc_\tau \right]\\
        & \ \vdots\\
        & = (\Pi_{i=t}^{\tau} A_i) x_\tau
    \end{align*}
    Then,
    $$
        m_{t|\tau} = \Eb \left[ z_t \ | \ \Fc_\tau \right] = \Eb \left[ \begin{bmatrix}I \\ K_t  \end{bmatrix} x_t \ | \ \Fc_\tau \right] = \begin{bmatrix}I \\ K_t  \end{bmatrix} \Eb \left[ x_t \ | \ \Fc_\tau \right] = \begin{bmatrix}I \\ K_t  \end{bmatrix} (\Pi_{i=t}^{\tau} A_i) x_\tau
    $$
    Therefore, using \eqref{eq:a_prod}, it holds that,
    $$
        \| m_{t|\tau} \| \leq \left \| \begin{bmatrix}I \\ K_t  \end{bmatrix} \right \| \| \Pi_{i=t}^{\tau} A_i \| \| x_\tau \| \leq (1 + \kappa) \kappa \left(1 - \gamma/2 \right)^{t - \tau} R_x
    $$
\end{proof}

The remainder of this section will be focused on ensuring that the conditions of Lemma \ref{lem:cov_approx} and Lemma \ref{lem:mean} are satisfied.
To that end, we first give a lemma showing that the phase covariance $\bar{\Sigma}_k$ satisfies an approximate stationarity condition in terms of the true system $\Theta_\star$.


\begin{lemma}
    \label{lem:stat_cond}
    It holds for all $k \in \{0,1,...,N\}$ that if event $F_k$ holds, then it follows that $\bar{\Sigma}_k \in \Ec_k$ where,
    $$
        \Ec_k := \{ \Sigma \succeq 0 : \Sigma_{xx} \succeq \Theta_\star \Sigma \Theta_\star^\top + W - 2 \eta \langle \bar{V}_k^{-1}, \Sigma \rangle I \}.
    $$
    Furthermore, it holds for all $k \in \{0,1,...,N\}$ that $\Ec_k \supseteq \Ec_{k+1}$.
\end{lemma}
\begin{proof}
    First, let $\hat{\Theta}_0 := \hat{\Theta}_0$.
    From the definition of $\Sigma_k^{\safe}$ it holds that,
    \begin{align*}
        \Sigma_{k,xx}^{\safe} & = \hat{\Theta}_k \Sigma_{k}^{\safe} \hat{\Theta}_k^\top + W\\
        & = \Theta_\star \Sigma_{k}^{\safe} \Theta_\star^\top + W + \hat{\Theta}_k \Sigma_{k}^{\safe} \hat{\Theta}_k^\top - \Theta_\star \Sigma_{k}^{\safe} \Theta_\star^\top\\
        & \succeq \Theta_\star \Sigma_{k}^{\safe} \Theta_\star^\top + W - \| \hat{\Theta}_k \Sigma_{k}^{\safe} \hat{\Theta}_k^\top - \Theta_\star \Sigma_{k}^{\safe} \Theta_\star^\top \| I \\
        & \succeq \Theta_\star \Sigma_{k}^{\safe} \Theta_\star^\top + W - \eta \ip{\Sigma_{k}^{\safe}, \bar{V}_k^{-1}} I,
    \end{align*}
    where the last line uses the condition that event $F_k$ holds.
    Therefore, $\Sigma_k^{\safe} \in \Ec_k$.
    Also, note that $\Sigma_k^o \in \Ec_k$ by definition.
    Then, because $\Ec_k$ is convex, it holds for $k \geq 1$ that,
    $$
        \bar{\Sigma}_k \in \conv\{ \Sigma_k^{\safe},  \Sigma_k^o \} \subseteq \Ec_k.
    $$
    Also, for $k = 0$, it holds immediately that $\bar{\Sigma}_0 = \Sigma_0^{\safe} \in \Ec_k$.
    Lastly, we show that $\Ec_k \supseteq \Ec_{k+1}$.
    Indeed, if $\Sigma \in \Ec_{k+1}$, then,
    \begin{align*}
        \Sigma_{xx} & \succeq \Theta_\star \Sigma \Theta_\star^\top + W - 2 \eta \ip{ \bar{V}_{k+1}^{-1}, \Sigma} I\\
        & \succeq \Theta_\star \Sigma \Theta_\star^\top + W - 2 \eta \ip{ \bar{V}_{k}^{-1}, \Sigma} I,
    \end{align*}
    where we use the fact that,
    $$
        \ip{ \bar{V}_{k+1}^{-1}, \Sigma} = \tr(\Sigma^{1/2} \bar{V}_{k+1}^{-1} \Sigma^{1/2}) \leq \tr(\Sigma^{1/2} \bar{V}_{k}^{-1} \Sigma^{1/2}) = \ip{ \bar{V}_{k}^{-1}, \Sigma},
    $$
    given that $\bar{V}_{k+1} \succeq \bar{V}_{k}$.
\end{proof}

Next, we show that if the phases are sufficiently long, then the target covariance gets close to the phase covariance.

\begin{lemma}
    \label{lem:phase_min}
    Let $H := \frac{1}{\zeta} \log(h T)$ for some $h > 0$.
    For $k \geq 2$, it holds for all $t$ such that $t \in [\tau_{k-1} + H, \tau_k - 1]$ that,
    \begin{equation}
        \label{eq:barsigdiff}
        \| \bar{\Sigma}_{k-1} - \Sigma_{t} \| \leq \frac{2 \nu}{h T}.
    \end{equation}
    It follows that, if $\tau_k \geq \tau_{k-1} + H$, it holds for all $t \in [\tau_{k-1} + H,\tau_{k+1} - 1]$,
    \begin{equation}
        \Sigma_t \in \conv\{ \bar{\Sigma}_{k-1}, \bar{\Sigma}_{k} \} \oplus \left(\frac{2 \nu}{h T} \Bb \right)
    \end{equation}
\end{lemma}
\begin{proof}
    First, it holds for all $t \in [\tau_{k-1} + H, \tau_k - 1]$ that,
    \begin{align*}
        \| \bar{\Sigma}_{k-1} - \Sigma_{t} \| & = (1 - \zeta)^{t - \tau_{k-1}} \| \bar{\Sigma}_{k-1} - \Sigma_{\tau_{k-1}-1} \| \\
        & \leq 2 \nu (1 - \zeta)^{t - \tau_{k-1}}\\
        & \leq 2 \nu \exp(- \zeta (t - \tau_{k-1}))\\
        & \leq 2 \nu \exp(- \zeta H)\\
        & = \frac{2 \nu}{h T}
    \end{align*}
    Therefore, it follows for all $t \in [\tau_k,\tau_{k+1}]$ that,
    \begin{align*}
        \Sigma_t & = (1 - \zeta) \Sigma_{t-1} + \zeta \bar{\Sigma}_k\\
        & = (1 - \zeta)^{t - \tau_k} \Sigma_{\tau_k} + \zeta \bar{\Sigma}_k \sum_{s=0}^{t - \tau_k-1} (1 - \zeta)^{s}\\
        & \in (1 - \zeta)^{t - \tau_k} \bar{\Sigma}_{k-1} + \zeta \bar{\Sigma}_k \sum_{s=0}^{t - \tau_k-1} (1 - \zeta)^{s} + \frac{2 \nu (1 - \zeta)^{t - \tau_k}}{h T} \Bb \\
        & \subseteq (1 - \zeta)^{t - \tau_k} \bar{\Sigma}_{k-1} + \zeta \bar{\Sigma}_k \sum_{s=0}^{t - \tau_k-1} (1 - \zeta)^{s} + \frac{2 \nu}{h T} \Bb \\
        & \subseteq \conv\{ \bar{\Sigma}_{k-1},\bar{\Sigma}_k\} \oplus \left(\frac{2 \nu}{h T} \Bb \right).
    \end{align*}
    For $t \in [\tau_{k-1} + H, \tau_k]$,
    \begin{align*}
        \Sigma_t & = (1 - \zeta)^{t - \tau_k} \Sigma_{(\tau_{k-1} + H)} + \zeta \bar{\Sigma}_{k-1} \sum_{s=0}^{t - \tau_k-1} (1 - \zeta)^{s}\\
        & \in (1 - \zeta)^{t - \tau_k} \Sigbar_{k-1} + \zeta \bar{\Sigma}_{k-1} \sum_{s=0}^{t - \tau_k-1} (1 - \zeta)^{s} + \frac{2 \nu}{h T} \Bb \\
        & = \Sigbar_{k-1} + \frac{2 \nu}{h T} \Bb \\
    \end{align*}
\end{proof}

The following lemma then shows that the target covariance satisfies an approximate steady-state condition.

\begin{lemma}
    \label{lem:real_stat}
    Suppose that $\| \Theta \| \leq S$, $\| \Sigma_t \| \leq \nu$ for all $t \leq \tau_{k+1}-1$, and $\lambda \geq \frac{H R_z^2}{\log(2)}$.
    The following holds for all $k \geq 2$.
    If $F_k$ holds, then for all $t \in [\tau_{k-1}+H, \tau_{k+1}-1]$,
    $$
        \Sigma_{t,xx} \succeq \Theta \Sigma_t \Theta^\top + W - 2 \eta \langle \bar{V}_{k-1}^{-1}, \Sigma_t \rangle I - 2 \nu h^{-1} T^{-1} (1 + S^2 + 2 \eta (n + m) \lambda^{-1}) I
    $$
\end{lemma}
\begin{proof}
    First, we apply Lemma \ref{lem:phase_dur} to show that the duration of each phase is greater than $H$.
    Event $F_k$ guarantees that $\| z_t \| \leq R_z$ for $t \leq \tau_k-1$.
    Therefore, from Lemma \ref{lem:phase_dur} and the condition on $\lambda$,
    $$
        \tau_k - \tau_{k-1} \geq \frac{\log(2)}{R_z^2} \lambda \geq H
    $$
    Therefore, we can apply Lemma \ref{lem:phase_min} and Lemma \ref{lem:stat_cond} to get that for all $t \in [\tau_{k-1}+H, \tau_{k+1}-1]$ with $k \geq 2$,
    $$
        \Sigma_t \in \conv \left\{ \bar{\Sigma}_{k-1}, \bar{\Sigma}_{k} \right\} \oplus \left(\frac{2 \nu}{h T} \Bb \right) \subseteq \Ec_{k-1} \oplus \left( \frac{2 \nu}{h T} \Bb \right),
    $$
    where we use the fact that $\bar{\Sigma}_{k} \in \Ec_k \subseteq \Ec_{k-1}$.
    Therefore, there exists $\Sigma \in \Ec_{k-1}$ such that $\| \Sigma - \Sigma_t\| \leq \frac{2 \nu}{h T}$.
    It follows that,
    \begin{align*}
        \Sigma_{t,xx} & \succeq \Sigma\xx - 2 \nu h^{-1} T^{-1} I \\
        & \succeq \Theta \Sigma \Theta^\top + W - 2 \eta \langle \bar{V}_{k-1}^{-1}, \Sigma \rangle I - 2 \nu h^{-1} T^{-1} I\\
        & \succeq \Theta \Sigma_t \Theta^\top + W - 2 \eta \langle \bar{V}_{k-1}^{-1}, \Sigma_t \rangle I - 2 \nu h^{-1} T^{-1} (1 + S^2 + 2 (n + m) \eta \lambda^{-1}) I,
    \end{align*}
    where use that,
    $$
        \| \Theta (\Sigma - \Sigma_t) \Theta^\top \| \leq \| \Theta \|^2 \| \Sigma - \Sigma_t \| \leq S^2 2 \nu h^{-1} T^{-1},
    $$
    and that,
    $$
        2 \eta \ip{ \bar{V}_{k-1}^{-1}, \Sigma  - \Sigma_t} \leq 2 \eta \tr(\bar{V}_{k-1}^{-1}) \| \Sigma  - \Sigma_t \| \leq 2 \eta (n + m) \lambda^{-1} 2 \nu h^{-1} T^{-1}.
    $$
\end{proof}

The following lemma handles the initial phase separately.

\begin{lemma}
    \label{lem:phase_0}
    Suppose that $F_1$ hold, and $\| \Theta \| \leq S$.
    Then, it holds for all $t \in [\tau_{2}-1]$ that $\Sigma_t \in\Ec_0 $.
\end{lemma}
\begin{proof}
    First, note that $\Sigma_t = \bar{\Sigma}_0 = \Sigma_k^{\safe}$ for $t \in [\tau_1-1]$, and therefore $\Sigma_t \in \Ec_0$ due to Lemma \ref{lem:stat_cond}.
    Therefore, for $t \in [\tau_{1}, \tau_{2}-1]$, 
    $$
        \Sigma_t \in \conv\{\Sigma_{\tau_1-1}, \bar{\Sigma}_1 \} \subseteq \conv(\Ec_0 \cup \Ec_1) = \Ec_0,
    $$
    where the last inclusion uses that $\Ec_0 \supseteq \Ec_1$ from Lemma \ref{lem:stat_cond} and that $\Ec_0$ is convex.
\end{proof}

Then, the following lemma establishes the stability of the estimate of the zero policy.

\begin{lemma}
    \label{lem:approx_stable}
    Suppose that $A_\star$ is $(\kappa, \gamma)$-strongly stable.
    Then, if $F_k$ holds, it follows that $\hat{A}_i$ is $(\kappa, \gamma/2)$-strongly stable and that $\tr(\Sigma_{i}^{\safe}) \leq \frac{2 \kappa^2 \tr(W)}{\gamma}$ for all $i \in \{0,1,...,k \}$.
\end{lemma}
\begin{proof}
    Let $\Delta_i = \hat{A}_i - A_\star$ and note that $\| \Delta_i \| \leq \frac{\gamma}{2 \kappa}$ for all $i \in \{0,1,...,k \}$, due to $F_k$ holding.
    Then, it follows that,
    $$
        \hat{A}_i = A_\star + \Delta_i = H L H^{-1} + \Delta_i = H (L + H^{-1} \Delta_i H) H^{-1}.
    $$
    Then,
    $$
        \| L + H^{-1} \Delta_i H \| \leq \| L \| + \|H^{-1} \| \| H \| \| \Delta_i \| \leq 1 - \gamma + \gamma/2 = 1 - \gamma/2 
    $$
    Then, from \eqref{eq:safe_update}, it holds that,
    $$
        \Sigma_{i,xx}^{\safe} = \hat{A}_i \Sigma_{i,xx}^{\safe} \hat{A}_i^\top + W \quad \implies \quad \Sigma_{i,xx}^{\safe} = \sum_{s=0}^{\infty} \hat{A}_i^s W (\hat{A}_i^s)^\top.
    $$
    Therefore, 
    \begin{align*}
        \tr(\Sigma_{i,xx}^{\safe}) & = \sum_{s=0}^{\infty} \tr(\hat{A}_i^s W (\hat{A}_i^s)^\top)\\
        & = \sum_{s=0}^{\infty} \tr(W^{1/2} (\hat{A}_i^s)^\top \hat{A}_i^s W^{1/2})\\
        & \leq \tr(W) \sum_{s=0}^{\infty} \kappa^2 (1 - \gamma/2)^{2 s} \\
        & = \frac{2 \kappa^2 \tr(W)}{\gamma},
    \end{align*}
    where we use the fact that $\|\hat{A}_i^s \| \leq \kappa (1 - \gamma/2)^s$.
    The proof is complete by noting that $\tr(\Sigma_{i}^{\safe}) = \tr(\Sigma_{i,xx}^{\safe})$.
\end{proof}

Next, the following lemma provides conditions for which the phase covariance and target covariance are bounded.

\begin{lemma}
    \label{lem:nu}
    Suppose that $\nu \geq \frac{2 \kappa^2 \tr(W)}{\gamma}$.
    Then, if $F_k$ holds, it follows that $\tr(\bar{\Sigma}_i) \leq \nu$ for all $i \leq k$ and $\tr(\Sigma_{s}) \leq \nu$ for all $s \in [\tau_{k+1}]$.
\end{lemma}
\begin{proof}
    From Lemma \ref{lem:approx_stable}, it holds that $\tr(\Sigma_{i}^{\safe}) \leq \nu$ for all $i \leq k$.
    Then, with $\Ec_\nu := \{ \Sigma : \tr(\Sigma) \leq \nu \}$ (which is convex), it holds that,
    \begin{equation}
        \label{eq:sigbar_bound}
        \Sigbar_i \in \conv\{ \Sigma_i^o, \Sigma_{i}^{\safe} \} \subseteq \Ec_\nu.
    \end{equation}
    Then, we use the fact that for all $i$ and $t \in [\tau_i, \tau_{i+1}-1]$,
    $$
        \Sigma_t \in \conv\{\Sigma_{\tau_i-1}, \bar{\Sigma}_i \},
    $$
    For one, this implies that $\Sigma_{\tau_{i+1}-1} \in \conv\{\Sigma_{\tau_i-1}, \bar{\Sigma}_i \}$. Applying this recursively implies that,
    $$
        \Sigma_s \in \conv\{\bar{\Sigma}_0, \bar{\Sigma}_1, ..., \bar{\Sigma}_i \},
    $$
    for $s \in [\tau_{i+1}-1]$. Since $\Sigbar_i \in \Ec_\nu$ for all $i \leq k$, it holds that $\Sigma_s \in \Ec_\nu$.
\end{proof}

Finally, we give the proof of Lemma \ref{lem:main_cov}.

\begin{proof}[Proof of Lemma \ref{lem:main_cov}]
    We show that the conditions of Lemma \ref{lem:cov_approx} are satisfied for $\tau = t - \rho$ as described in each of the numbered sections.
    Note that, since $E$ is assumed to hold, it follows that $F_k$ holds for all $k \in \{0,...,N\}$.

    \emph{Condition \#1:}
    First note that the conditions of Lemma \ref{lem:real_stat} hold, as Lemma \ref{lem:nu} shows that the $\tr(\Sigma_t)\leq \nu$ for all $t$.
    Then, from condition \ref{it:state:7} $\lambda \geq \frac{R_z^2 (H + \rho)}{\log(2)}$, and therefore Lemma \ref{lem:phase_dur} says that $\tau_{k} - \tau_{k-1} \geq H + \rho$ for all $k \geq 1$.
    It follows for $t \geq \tau_2$, that $[t-\rho,t] \subseteq [\tau_{k_t-1} + H, \tau_{k_t+1}]$, and therefore, Lemma \ref{lem:real_stat} tells us that for all $s \in [t-\rho,t]$,
    $$
        \Sigma_{s,xx} \succeq \Theta \Sigma_s \Theta^\top + W - \underbrace{\left( 2 \eta \ip{\bar{V}_{k-1}^{-1}, \Sigma_s} + 2 \nu T^{-1} h^{-1} (1 + S^2 + 2 (n + m) \eta \lambda^{-1}) \right)}_{\bar{\eta}_s} I,
    $$
    Then, we show that this holds for $t \in [\tau_1, \tau_2-1]$ as well.
    Indeed, from Lemma \ref{lem:phase_0}, it holds that for all $s \in [t-\rho,t]$ (when $t \in [\tau_1, \tau_2-1]$),
    $$
        \Sigma_{s,xx} \succeq \Theta \Sigma_s \Theta^\top + W - \eta \ip{\bar{V}_{k-1}^{-1}, \Sigma_s} I \succeq \Theta \Sigma_s \Theta^\top + W - \bar{\eta}_s I
    $$
    This confirms condition \#1 in Lemma \ref{lem:cov_approx}.

    \emph{Condition \#2:}
    We have already shown that $\tr(\Sigma_t) \leq \nu$, and therefore it holds that  $\| \Sigma_t \| \leq \tr(\Sigma_t) \leq \nu$ given that $\Sigma_t \succeq 0$.

    \emph{Condition \#3:} Next, we apply Lemma \ref{lem:strong_stab} by showing each of the conditions in that lemma are satisfied.
    We have already shown conditions \#1-2 in Lemma \ref{lem:strong_stab}, i.e. that $\Sigma_t \succeq 0$ and $\tr(\Sigma_t) \leq \nu$.
    Then, to establish condition \#3 in Lemma \ref{lem:strong_stab}, we note that,
    \begin{align*}
        \bar{\eta}_s & \leq 2 \eta \lambda^{-1} \nu + 2 \nu h^{-1} T^{-1} (1 + S^2 + 2 (n + m) \eta \lambda^{-1}) \\
        & \leq 2 \eta \lambda^{-1} \nu + 2 \nu h^{-1} (1 + S^2 + 2 (n + m) )\\
        & \leq \sigma/2,
    \end{align*}
    where we use the fact that $T \geq 1, \lambda \geq \eta$ in the second line, and then, the fact that $\lambda \geq \frac{8 \eta \nu}{\sigma}$, $h \geq \frac{8 \nu (1 + S^2 + 2 (n+m))}{\sigma}$ in the third line.
    Then, condition \#4 in Lemma \ref{lem:strong_stab} can be satisfied with the choice $\zeta \leq \frac{\sigma^2}{16 \nu^2}$, which ensures that $\| \Sigma_{t+1} - \Sigma_t \| \leq \frac{(\sigma/2)^2}{2 \nu}$ due to Lemma \ref{lem:slow_var}, as desired.
    Therefore, we have shown that the sequence of controllers $(K_s)_{s \in [\tau,t]}$ are $(\tilde{\kappa}, \tilde{\gamma}) = (\frac{2 \nu}{\sigma}, \frac{\sigma}{4 \nu})$ strongly-stable.

    \emph{Condition \#4:} To show this condition, we need to argue that $K_s, U_s$ are $\Fc_{t-\rho}$-measurable for all $s \in [t,t-\rho]$.
    This holds because the update at each phase only uses information from time steps earlier than $\tau_k - \rho$.
    Therefore, for all time steps $t$ in phase $k$, it holds that $\Sigma_t$ is fully determined by the randomness in time steps earlier than $t - \rho$.
    Thus, $K_s, U_s$ are $\Fc_{t-\rho}$-measurable for all $s \in [t,t-\rho]$.

    Then, to show the bound in the lemma, we note that,
    \begin{align*}
        \sum_{s=t-\rho}^{t} \bar{\eta}_s & = 2 \eta \sum_{s=t-\rho}^{t} \ip{\bar{V}_{k-1}^{-1}, \Sigma_s} + \rho 2 \nu T^{-1} h^{-1} (1 + S^2 + 2 (n + m) \eta \lambda^{-1})\\
        & = 2 \rho \eta \ip{\bar{V}_{k-1}^{-1}, \Sigma_t} + 2 \eta \sum_{s=t-\rho}^{t} \ip{\bar{V}_{k-1}^{-1}, \Sigma_s - \Sigma_t} + \rho 2 \nu T^{-1} h^{-1} (1 + S^2 + 2 (n + m) \eta \lambda^{-1})\\
        & \leq 2 \rho \eta \ip{\bar{V}_{k-1}^{-1}, \Sigma_t} + 2 \eta \tr(\bar{V}_{k-1}^{-1}) \sum_{s=t-\rho}^{t} \| \Sigma_s - \Sigma_t \| + \rho 2 \nu T^{-1} h^{-1} (1 + S^2 + 2 (n + m) \eta \lambda^{-1})\\
        & \leq 2 \rho \eta \ip{\bar{V}_{k-1}^{-1}, \Sigma_t} + 4 \nu \zeta \eta (n + m) \lambda^{-1} \sum_{s=t-\rho}^{t} (t - s) + \rho 2 \nu T^{-1} h^{-1} (1 + S^2 + 2 (n + m) \eta \lambda^{-1})\\
        & = 2 \rho \eta \ip{\bar{V}_{k-1}^{-1}, \Sigma_t} + 2 \nu \zeta \eta (n + m) \lambda^{-1} \rho(\rho + 1) + \rho 2 \nu T^{-1} h^{-1} (1 + S^2 + 2 (n + m) \eta \lambda^{-1})
    \end{align*}

    Finally, we show the bound on $m_{t|t-\rho}$ via Lemma \ref{lem:mean}. Indeed, the conditions \#1 and \#2 in Lemma \ref{lem:mean} are the same as conditions \#3 and \#4 in Lemma \ref{lem:cov_approx} as shown above.
    Then, condition \#3 in Lemma \ref{lem:mean} follows immediately from event $E$.
\end{proof}

\subsection{Proof of Lemma \ref{lem:bounded_state}}
\label{sec:bounded_state_proof}

In this section, we prove Lemma \ref{lem:bounded_state}, which shows that event $E$ holds with high probability and hence the estimation error is bounded and the state is bounded.
This proof uses an induction inspired by the comparable lemmas in \cite{abbasi2011regret} and \cite{cohen2019learning}.
In particular, we show that if the estimation error is bounded in a given phase, the phase update will be well-defined and therefore the resulting controller will be stable.
The controller being stable ensures that the state is bounded and therefore that the estimation error is bounded.
This type of inductive argument thus implies both bounded state and bounded estimation error.
The events $F_k$ are used to isolate each step of this induction, i.e. we show that if $F_k$ holds then $F_{k+1}$ holds.

For the purposes of this section, we take $K_s = 0$ and $\Sigma_s = \bar{\Sigma}_0$ for $s \in [\tau_1-1]$.
We first give a lemma mostly drawn from Lemma 6 in \cite{cohen2019learning} (which itself draws heavily from \cite{abbasi2011regret}) that provides a high probability bound on the estimation error of the dynamics.

\begin{lemma}[Dynamics Estimation Bound]
    Suppose that $\| \hat{\Theta}_0 - \Theta_\star \|_F \leq \Delta$ and $\| W \| \leq \bar{w}$.
    Then, let $\Econf$ be the event that, for all $k \in \{0,...,N\}$,
    \begin{equation}
        \label{eq:est_bound}
        \tr \left( (\Theta_\star - \hat{\Theta}_k) \bar{V}_k (\Theta_\star - \hat{\Theta}_k)^\top \right) \leq 4 \bar{w} d^2 \log\left( \omega_1^{-1}\left(d + \lambda^{-1} \sum_{s=1}^{\tau_k - \rho} \| z_s \|^2 \right) \right) + 2 \lambda \Delta^2.
    \end{equation}
    It holds that $\Pb(\Econf) \geq 1 - \omega_1$.
\end{lemma}
\begin{proof}
    First, we note that for $k = 0$, it holds that,
    $$
        \tr \left( (\Theta_\star - \hat{\Theta}_k) \bar{V}_k (\Theta_\star - \hat{\Theta}_k)^\top \right) = \lambda \| \Theta_\star - \hat{\Theta}_k \|_F^2 \leq \lambda \Delta^2,
    $$
    and therefore \eqref{eq:est_bound} is satisfied almost surely for this case.

    Then, we handle the case of $k \geq 1$.
    To do so, we apply Lemma 6 in \cite{cohen2018online}.
    This lemma requires that each element of $w_t$ is subgaussian.
    This holds as the $i$th element of $w_t$ can be written as $[w_t]_i = \mathbf{e}_i^\top w_t$ and therefore,
    $$
        \mathrm{var}([w_t]_i) = \mathbf{e}_i^\top W \mathbf{e}_i \leq \| W^{1/2} \mathbf{e}_i \|^2 \leq \| W^{1/2} \|^2 \| e_i \|^2 = \| W \|
    $$
    Therefore, Lemma 6 in \cite{cohen2018online} (with $\beta = 1$) tells us that,
    \begin{align*}
        \tr \left( (\Theta_\star - \hat{\Theta}_k) \bar{V}_k (\Theta_\star - \hat{\Theta}_k)^\top \right) & \leq 4 \sigma^2 d \log\left(\frac{d \det(\bar{V}_k)}{\omega \det(\lambda I)} \right) + 2 \lambda \Delta^2.
    \end{align*}
    Then,
    \begin{align*}
        \log\left(\frac{d \det(\bar{V}_k)}{\omega \det(\lambda I)} \right) & = \log\left(d/\omega \right) + \log\left(\det(\bar{V}_k) \lambda^{-d} \right)\\
        & \leq \log\left(d/\omega \right) + \log\left((d^{-1} \tr(\bar{V}_k))^d \lambda^{-d} \right)\\
        & = \log\left(d/\omega \right) + \log\left(\left(d \lambda + \sum_{s=1}^{\tau_k - \rho} \| z_s \|^2 \right)^d d^{-d} \lambda^{-d} \right)\\
        & = \log\left(d/\omega \right) + d \log\left( 1 + \lambda^{-1} d^{-1} \sum_{s=1}^{\tau_k - \rho} \| z_s \|^2  \right)\\
        & \leq d \log\left( \omega^{-1}\left(d + \lambda^{-1} \sum_{s=1}^{\tau_k - \rho} \| z_s \|^2 \right) \right)
    \end{align*}
\end{proof}

Next, we give a bound on the non-linear part of the closed-loop system.
This consists of both the random part of the input, as well as the disturbance.
Since both of these are Gaussian, we use Gaussian tail bounds to establish this bound.

\begin{lemma}[Noise Bound]
    Suppose that $\| W \| \leq \bar{w}$ and $\| \Theta_\star \| \leq S$.
    Let $\Enois$ be the event that, for all $t \in [T]$,
    $$
        \| x_{t+1} - (A + B K_t) x_t \| \leq  2 \left(\sqrt{\bar{w}} + S \sqrt{\tr(\Sigma_{t,uu})} \right) \sqrt{2 n \log(4 n T /\omega_2)} + S c.
    $$
    Then, it holds that $\Pb(\Enois) \geq 1 - \omega_2$
\end{lemma}
\begin{proof}
    Note that,
    $$
    x_{t+1} - (A + B K_t) x_t = \begin{cases} 
        B u_t + w_t & t \leq \tau_0 - 1\\
        w_t & \tau_0 \leq t \leq \tau_1\\
        B v_t + w_t & t \geq \tau_1 \end{cases}
    $$

    We first bound $w_t$ for all $t$, since it appears at every time step.
    Since $[w_t]_i$ is $\bar{w}$-subgaussian, it holds for all $i \in [n]$ that,
    $$
        \Pb\left(| [w_t]_i | \leq \sqrt{2 \bar{w} \log(2/\delta)} \right) \geq 1 - \delta
    $$
    Therefore, applying the union bound for each element $i \in [n]$ and time step $t \in [T]$, and the fact that $\| \cdot \| \leq \sqrt{n} \| \cdot \|_\infty$,
    $$
        \Pb\left( \| w_t \| \leq \sqrt{2 n \bar{w} \log(2 n T /\delta)} \right) \geq 1 - \delta.
    $$
    Then, we look at the term $B v_t$ for $t \geq \tau_1$. First, note that,
    $$
        \| B U_t B^\top \| \leq \tr(B \Sigma_{t,uu} B^\top) - \tr(B \Sigma_{t,ux} \Sigma_{t,xx}^{-1} \Sigma_{t,xu} B^\top) \leq \tr(B \Sigma_{t,uu} B^\top) \leq S^2 \tr(\Sigma_{t,uu})
    $$
    Therefore, $B v_t$ is $S^2 \tr(\Sigma_{t,uu})$-subgaussian conditioned on $\Fc_{t-1}$. Then, applying the (conditional) subgaussian tail bound and taking the union bound over each $i \in [n]$,
    \begin{align*}
        \Pb\left( \| B v_t \|_\infty \leq S \sqrt{2 \tr(\Sigma_{t,uu}) \log(2 n /\delta)} \right) & = \Eb \left[ \Pb\left( \| B v_t \|_\infty \leq S \sqrt{2 \tr(\Sigma_{t,uu}) \log(2 n /\delta)} \ | \ \Fc_{t-1} \right) \right]\\
        & \geq 1 - \delta
    \end{align*}
    Taking the union bound for each $t \geq \tau_1$,
    \begin{align*}
        & \Pb\left( \| B v_t \| \leq S \sqrt{2 n \tr(\Sigma_{t,uu}) \log(2 n T /\delta)}, \forall t \in [T] \right) \geq 1 - \delta.
    \end{align*}

    Next, we note that for $t \leq \tau_0 - 1$, it holds almost surely that $\| u_t \| \leq c$.

    Finally, applying the union bound over the bounds for $v_t$ and $w_t$, it holds with probability at least $1 - \omega_2$ that, for all $t \in [T]$,
    $$
        \| x_{t+1} - (A + B K_t) x_t \| \leq  2 \left(\sqrt{\bar{w}} + S \sqrt{\tr(\Sigma_{t,uu})} \right) \sqrt{2 n \log(4 n T /\omega_2)} + S c,
    $$
    where the $+ S c$ accounts for $u_t$ for $t \leq \tau_0 - 1$.
\end{proof}

The next lemma shows that if $F_k$ holds, then the closed loop system is stable up through the end of the $k$th phase. 

\begin{lemma}[Conditionally Strongly-Stable]
    \label{lem:cond_s}
    In addition to the assumptions of Lemma \ref{lem:safe_margin}, assume that:
    \begin{enumerate}
        \item $W \succeq \sigma I$,
        \item  $A_\star$ is $(\kappa, \gamma)$-strongly stable,
        \item $\| \Theta_\star \| \leq S$
        \item $\nu \geq \frac{2 \kappa^2 \tr(W)}{\gamma}$
        \item $\lambda \geq \max \left(\frac{H R_z^2}{\log(2)},\frac{8 \eta \nu}{\sigma}, \eta \right)$
    \end{enumerate}
    Then, the following holds for all $k \geq 1$. If $F_k$ holds, then the sequence $(K_s)_{s=1}^{\tau_{k+1}-1}$ is $(\bar{\kappa}, \bar{\gamma}) = (\frac{2 \nu}{\sigma}, \frac{\sigma}{4 \nu})$ strongly stable, and $\tr(\Sigma_{s}) \leq \nu$ for all $s \in [\tau_{k+1}-1]$.
\end{lemma}
\begin{proof}
    We show that the conditions of Lemma \ref{lem:strong_stab} hold for all $s\in [\tau_{k+1} - 1]$ when $F_k$ holds, as this implies that $K_s$ are strongly stable.

    First, we show that the update at phase $i \leq k$ of the algorithm is well-defined.
    Indeed, conditioned on $F_k$, Lemma \ref{lem:safe_margin} says that $\Sigma_i^\safe \in \Ec_i^p$, and it follows that the safe scaling step in the algorithm (line \ref{lne:scale}) is well-defined in the sense that there exists $\phi = 0$ such that $\phi \Sigma_i^o + (1 - \phi) \Sigma_i^\safe \in \Ec_i^p$ for all $i \leq k$.

    Then, we show that $\tr(\Sigma_{s}) \leq \nu$ for all $s \in [\tau_{k+1}-1]$.
    Indeed, the conditions for Lemma \ref{lem:nu} are satisfied since $F_k$ is assumed to hold, and due to the assumption on $\nu$ and the assumption that $A_\star$ is strongly stable.
    Therefore Lemma \ref{lem:nu} says that $\tr(\Sigma_{s}) \leq \nu$ for all $s \in [\tau_{k+1}-1]$, and $\tr(\Sigbar_i) \leq \nu$ for all $i \leq k$.
    This verifies condition \#2 in Lemma \ref{lem:strong_stab} for $s \in [\tau_{k+1} - 1]$.

    Next, we show that all $\Sigma_s$ satisfies the approximate stationarity condition (\#3 in Lemma \ref{lem:strong_stab}) for all $s\in [\tau_{k+1} - 1]$.
    We do so by applying Lemma \ref{lem:real_stat} for each phase $i \leq k$.
    Note that since $F_i$ are monotone, our assumption that $F_k$ holds implies that $F_i$ holds for $i \leq k$.
    The other conditions of Lemma \ref{lem:real_stat} hold as we assume that $\| \Theta_\star \| \leq S$ and $\lambda \geq \frac{H R_z^2}{\log(2)}$, and have shown that $\tr(\Sigma_{s}) \leq \nu$ for all $s \in [\tau_{k+1}-1]$ which.
    Therefore, applying Lemma \ref{lem:real_stat} to each phase $i \geq 2$, it holds for all $s \in [\tau_2,\tau_{k+1}-1]$,
    \begin{equation}
        \label{eq:approx_stat}
    \begin{split}
        \Sigma_{s,xx} & \succeq \Theta \Sigma_s \Theta^\top + W - 2 \eta \langle \bar{V}_{k_s-1}^{-1}, \Sigma_s \rangle I - 2 \nu T^{-1} (1 + S^2 + 2 \eta (n + m) \lambda^{-1}) I \\
        & \succeq \Theta \Sigma_s \Theta^\top + W - 2 \eta \lambda^{-1} \nu I - 2 \nu (1 + S^2 + 2 (n + m) ) I \\
        & \succeq \Theta \Sigma_s \Theta^\top + W - (\sigma/2) I,
    \end{split}
    \end{equation}
    where we use the fact that $T \geq 1, \lambda \geq \eta$ in the second line, and then, the fact that $\lambda \geq \frac{8 \eta \nu}{\sigma}$, $h \geq \frac{8 \nu (1 + S^2 + 2 (n+m))}{\sigma}$ in the third line.
    In the case of $s \in [\tau_2-1]$, \eqref{eq:approx_stat} still holds due to Lemma \ref{lem:phase_0}.
    Thus, we have shown that condition \#3 in Lemma \ref{lem:strong_stab} is satisfied with $\bar{\eta} = \sigma/2$ for all $s \in [\tau_{k+1}-1]$.

    Finally, it holds for $s \in [2,\tau_{k+1}-1]$,
    $$
        \| \Sigma_{s} - \Sigma_{s-1} \| = \zeta \| \bar{\Sigma}_{k_s} - \Sigma_s \| \leq 2 \nu \zeta \leq \frac{\sigma/2}{2 \nu^2} = \frac{\sigma - \bar{\eta}}{2 \nu^2},
    $$
    where we use that $\| \Sigma_s \| \leq \nu$ and $\| \Sigbar_k \| \leq \nu$ as we have shown, and additionally, the assumption on $\zeta$.
    Thus, condition \#4 of Lemma \ref{lem:strong_stab} has been satisfied.
\end{proof}

Then, the next lemma shows that if the closed-loop system is stable, then the state is bounded.

\begin{lemma}
    \label{lem:z_bound}
    Suppose that $\Enois$ holds and $\| x_1 \| \leq R_1$.
    Then, if $(K_s)_{s=1}^{\tau_{k+1}-1}$ is $(\bar{\kappa}, \bar{\gamma})$ strongly stable, and $\tr(\Sigma_{s,uu}) \leq \nu$ for all $s \in [\tau_{k+1}-1]$, it follows that
    $$\| z_s \| \leq (1 + \bar{\kappa})\bar{\kappa} R_1 + 4 (1 + \bar{\kappa}) \bar{\kappa} \bar{\gamma}^{-1} \left(\sqrt{\bar{w}} + S \sqrt{\nu} \right) \sqrt{2 n \log(4 n T /\omega_2)} + 2 (1 + \bar{\kappa}) \bar{\kappa} \bar{\gamma}^{-1} S c$$ 
    for all $s \in [\tau_{k+1}-1]$.
\end{lemma}
\begin{proof}
    First, note that,
    \begin{align*}
        x_t & = (A + B K_{t-1}) x_{t-1} + B v_{t-1} + w_{t-1}  \\
        & = \left(\Pi_{i=t-1}^1 (A + B K_{i}) \right) x_1 + \sum_{s=1}^t \left(\Pi_{i=t-1}^s (A + B K_{i})\right) (B v_{s} + w_{s}).
    \end{align*}
    Then, using the definition of sequential strong stability, it holds for all $t \geq \tau$ that,
    \begin{align*}
        \| \Pi_{i=t}^{\tau} (A + B K_{t-1}) \| & \leq \bar{\kappa} \left(1 - \bar{\gamma}/2 \right)^{t - \tau}.
    \end{align*}
    Therefore, under $\Enois$,
    \begin{align*}
        \| x_t\|  & \leq \left\| \Pi_{i=t-1}^1 (A + B K_{i}) \right\| R_1 + \sum_{s=1}^t \left\| \Pi_{i=t-1}^s (A + B K_{i})\right\| \| B v_{s} + w_{s} \|\\
        & \leq \bar{\kappa} \left(1 - \bar{\gamma}/2 \right)^{t} + \bar{\kappa} \sum_{s=1}^t  \left( 2 \sqrt{2 n \log(4 n T /\omega_2)}  \left(\sqrt{\bar{w}} R_1 + S \sqrt{\tr(\Sigma_{t,uu})} \right) + S c \right) \left(1 - \bar{\gamma}/2 \right)^{s - t} \\
        & \leq \bar{\kappa}  R_1 + 4 \bar{\kappa} \bar{\gamma}^{-1} \left(\sqrt{\bar{w}} + S \sqrt{\nu} \right) \sqrt{2 n \log(4 n T /\omega_2)} + 2 \bar{\kappa} \bar{\gamma}^{-1} S c,
    \end{align*}
    where the last line uses the bound condition that $\tr(\Sigma_{t,uu}) \leq \tr(\Sigma_t) \leq \nu$.
    It follows that,
    $$
        \| z_t \| \leq \left\| \begin{bmatrix} I\\ K_t \end{bmatrix} x_t \right\| \leq (1 + \bar{\kappa})\bar{\kappa} R_1 + 4 (1 + \bar{\kappa}) \bar{\kappa} \bar{\gamma}^{-1} \left(\sqrt{\bar{w}} + S \sqrt{\nu} \right) \sqrt{2 n \log(4 n T /\omega_2)} + 2 (1 + \bar{\kappa}) \bar{\kappa} \bar{\gamma}^{-1} S c
    $$
\end{proof}

Next, we show that if the state is bounded, then the estimation error is bounded.

\begin{lemma}
    \label{lem:conf_set}
    Suppose that the following hold:
    \begin{enumerate}
        \item $\Econf$ holds,
        \item $\| \hat{\Theta}_0 - \Theta_\star \|_F \leq \lambda^{-1/2}$
        \item $\lambda \geq 4 \kappa^2 \gamma^{-2} \left( 4 \bar{w} d^2 \log\left( \omega_1^{-1}\left(d + \lambda^{-1} T R_z^2 \right) \right) + 2 \right)$
        \item $\eta \geq (4 \bar{w} d^2 \log\left( \omega_1^{-1}\left(d + \lambda^{-1} T R_z^2 \right) \right) + 2) (1  + S \sqrt{\lambda + T R_z^2 })$
        \item $\| \Theta_\star \| \leq S$
    \end{enumerate}
    Then, if $\| z_s \| \leq R_z$ for all $s \in [\tau_{k+1}-1]$, then it follows that,
    \begin{align*}
        & \| \Theta_\star - \hat{\Theta}_{k+1} \| \leq \frac{\gamma}{2 \kappa},\\
        & \| \Theta_\star \Sigma \Theta_\star - \hat{\Theta}_{k+1} \Sigma \hat{\Theta}_{k+1} \| \leq \eta \ip{\bar{V}^{-1}_{k+1}, \Sigma} \ \forall \Sigma \succeq 0
    \end{align*}
\end{lemma}
\begin{proof}
    Due to $\Econf$ and the boundedness of the states, it holds for all
    \begin{align*}
        \tr \left( (\Theta_\star - \hat{\Theta}_{k+1}) \bar{V}_{k+1} (\Theta_\star - \hat{\Theta}_{k+1})^\top \right) & \leq 4 \bar{w} d^2 \log\left( \omega_1^{-1}\left(d + \lambda^{-1} \sum_{s=1}^{\tau_{k+1} - \rho} \| z_s \|^2 \right) \right) + 2 \lambda \Delta^2\\
        & \leq 4 \bar{w} d^2 \log\left( \omega_1^{-1}\left(d + \lambda^{-1} (\tau_{k+1} - \rho) R_z^2 \right) \right) + 2 \lambda \Delta^2\\
        & \leq 4 \bar{w} d^2 \log\left( \omega_1^{-1}\left(d + \lambda^{-1} T R_z^2 \right) \right) + 2 \lambda \Delta^2 \\
        & \leq \underbrace{4 \bar{w} d^2 \log\left( \omega_1^{-1}\left(d + \lambda^{-1} T R_z^2 \right) \right) + 2}_{r}.
    \end{align*}
    Therefore, it holds that,
    \begin{equation}
    \label{eq:fk_cond1}
    \begin{split}
        \| \Theta_\star - \hat{\Theta}_{k+1} \| & \leq  \| \bar{V}_{k+1}^{-1/2} \bar{V}_{k+1}^{1/2} (\Theta_\star - \hat{\Theta}) \|\\
        & \leq \lambda^{-1/2} \| \bar{V}_{k+1}^{1/2} (\Theta_\star - \hat{\Theta}) \|\\
        & \leq \lambda^{-1/2} \| \bar{V}_{k+1}^{1/2} (\Theta_\star - \hat{\Theta}) \|_F \\
        & \leq \lambda^{-1/2}  r^{1/2} \\
        & \leq \frac{\gamma}{2 \kappa}
    \end{split}
    \end{equation}
    Then, we apply Lemma 14 in \cite{cohen2019learning} with the following assignments $X = \Theta_\star, \Delta = \hat{\Theta} - \Theta_\star, V = \bar{V}_{k+1}/r, \mu = 1 + S r^{-1/2} \sqrt{\lambda + T R_z^2 }$.
    Then, we can verify the assumptions of Lemma 14 in \cite{cohen2019learning}.
    First, $\Delta^\top \Delta \preceq V^{-1}$ because the estimation error bound implies,
    $$
        \tr(\bar{V}_{k+1}^{1/2} \Delta^\top \Delta \bar{V}_{k+1}^{1/2}) = \tr(\Delta \bar{V}_{k+1} \Delta^\top) \leq r \ \implies \ \bar{V}_{k+1}^{1/2} \Delta^\top \Delta \bar{V}_{k+1}^{1/2} \preceq r I \ \implies \ \Delta^\top \Delta \preceq r \bar{V}_{k+1}^{-1} = V^{-1}.
    $$
    Then, $\mu \geq 1 + \| X \| \| V \|^{1/2}$ as,
    \begin{align*}
         1 + \| X \| \| V \|^{1/2} & = 1 + r^{-1/2} \| \Theta_\star \| \| \bar{V}_{k+1} \|^{1/2}\\
         & \leq 1 + S r^{-1/2} \sqrt{\lambda + \sum_{s=1}^{\tau_{k+1} - \rho} \| z_s \|^2 }\\
         & \leq 1 + S r^{-1/2} \sqrt{\lambda + (\tau_{k+1} - \rho) R_z^2 }\\
         & \leq 1 + S r^{-1/2} \sqrt{\lambda + T R_z^2 } = \mu.
    \end{align*}
    Therefore, Lemma 14 in \cite{cohen2019learning} tells us that, for all $\Sigma \succeq 0$,
    \begin{align}
        \label{eq:fk_cond2}
        \| \Theta_\star \Sigma \Theta_\star - \hat{\Theta}_{k+1} \Sigma \hat{\Theta}_{k+1} \| & \leq \mu \ip{V^{-1}, \Sigma}\\
        & = \mu r \ip{\bar{V}_{k+1}^{-1}, \Sigma}\\
        & = (r  + S r^{1/2} \sqrt{\lambda + T R_z^2 }) \ip{\bar{V}_{k+1}^{-1}, \Sigma}\\
        & \leq \underbrace{r (1  + S \sqrt{\lambda + T R_z^2 })}_{\eta} \ip{\bar{V}_{k+1}^{-1}, \Sigma},
    \end{align}
    where we use the fact that $r \geq 1$.
\end{proof}

Finally, putting all of these pieces together, we give the proof of Lemma \ref{lem:bounded_state}.

\begin{proof}[Proof of Lemma \ref{lem:bounded_state}]
    First, define the event $E_{\mathrm{init}} = \{ \| \hat{\Theta}_0 - \Theta_\star \|_F \leq \lambda^{-1/2} \}$.
    Then, note that $E = F_N$. Thus, we prove the claim by showing that $F_k$ holds for all $k \in [N]$ by induction, when $\Enois$, $\Econf$ and $E_{\mathrm{init}}$ hold. Therefore, applying the bounds on the probabilities of $\Enois$, $\Econf$ and $E_{\mathrm{init}}$ completes the proof.

    \textit{Base case:} Since $A_\star$ is assumed to be $(\kappa,\gamma)$ strongly-stable, it holds that $K_s = \bzero$ is $(\kappa,\gamma)$ strongly-stable for $s \in [\tau_1-1]$. Also, $\Sigma_{t,uu} = \Sigma_{0,uu}^\safe = \bzero$ and therefore $\tr(\Sigma_{t,uu}) \leq \nu$.
    Since $(K_s)_{s=1}^{\tau_1-1}$ is strongly-stable and $\tr(\Sigma_{t,uu}) \leq \nu$, as well as $\| x_1 \| \leq R_1$ and $\Enois$ by assumption, we apply Lemma \ref{lem:z_bound} to get that $\| z_s \| \leq R_z$ for all $s \in [\tau_1-1]$.
    Then, due to the choice of $\lambda$ and $\eta$, and that we are under event $E_{\mathrm{init}}$, we can apply Lemma \ref{lem:conf_set} to get that $F_1$ holds.

    \textit{Induction step:}
    Suppose that $F_k$ holds.
    Then, from Lemma \ref{lem:cond_s}, it follows that,  $(K_s)_{s=1}^{\tau_{k+1-1}}$ is $(\frac{2\nu}{\sigma},\frac{\sigma}{4 \nu})$ strongly-stable and $\tr(\Sigma_{s,uu}) \leq \nu$ for all $s \in [\tau_{k+1}-1]$.
    Therefore, Lemma \ref{lem:z_bound} tells us that $\| z_s \| \leq R_z$ for all $s \in [\tau_{k+1}-1]$.
    Thus, Lemma \ref{lem:conf_set} tells us that,
    \begin{align*}
        & \| \Theta_\star - \hat{\Theta}_{k+1} \| \leq \frac{\gamma}{2 \kappa},\\
        & \| \Theta_\star \Sigma \Theta_\star - \hat{\Theta}_{k+1} \Sigma \hat{\Theta}_{k+1} \| \leq \eta \ip{\bar{V}_{k+1}, \Sigma} \ \forall \Sigma \succeq 0.
    \end{align*}
    Combined with $F_k$, this implies that $F_{k+1}$ holds.

    \textit{Completing the proof:} We have shown that $E$ holds if $\Enois$ and $\Econf$ holds.
    Taking $\omega_1 = \omega_2 = \omega/3$, ensures that $\Enois$, $\Econf$ and $E_{\mathrm{init}}$ jointly hold with probability at least $1 - \omega$.
\end{proof}

\subsection{Proof of Lemma \ref{lem:quant_form}}
\label{sec:quant_form}

\begin{proof}
    Notice that the distribution of $\alpha_j^\top z_t$ conditioned on $\Fc_\tau$ is of the form $\Nc(\alpha_j^\top m_{t|\tau}, \alpha_j^\top S_{t | \tau} \alpha_j)$.
    Therefore, with $\Phi(\cdot)$ as the CDF of the standard normal and $\Phi^{-1}(\cdot)$ as its inverse (the quantile function),
    \begin{align*}
        & \Pb( \alpha_j^\top z_t \leq \beta | \Fc_\tau) \geq 1 - \delta\\
        & \iff \quad \Phi\left( \frac{\beta - \alpha_j^\top m_{t|\tau}}{\sqrt{\alpha_j^\top S_{t | \tau} \alpha_j}} \right) \geq 1 - \delta\\
        & \iff \quad \frac{\beta - \alpha_j^\top m_{t|\tau}}{\sqrt{\alpha_j^\top S_{t | \tau} \alpha_j}} \geq \Phi^{-1}(1 - \delta)\\
        & \iff \quad \alpha_j^\top m_{t|\tau} + \sqrt{\alpha_j^\top S_{t | \tau} \alpha_j} \Phi^{-1}(1 - \delta) \leq \beta \\
        & \iff \quad \alpha_j^\top m_{t|\tau} + \sqrt{\ip{\balpha_j, S_{t | \tau}}} \Phi^{-1}(1 - \delta) \leq \beta,
    \end{align*}
    where we use the fact that,
    $$
        \alpha_j^\top S_{t | \tau} \alpha_j = \tr(\alpha_j^\top S_{t | \tau} \alpha_j) = \tr(\alpha_j \alpha_j^\top S_{t | \tau}) = \ip{\balpha_j, S_{t | \tau}}
    $$
\end{proof}

\subsection{Proof of Lemma \ref{lem:ell_pot}}
\label{sec:ell_pot}

In this section, we prove Lemma \ref{lem:ell_pot}, which bounds the sum of the Gram-weighted target covariances $\ip{\bar{V}_{k_t-1}^{-1}, \Sigma_t}$.
To do so, we first show how this term can be related to the conditional expectation of the Gram-weighted norm of the states $\Eb_{t-\rho} [z_t^\top V_t^{-1} z_t]$ in Lemma \ref{lem:ell1}.
Then, we bound this in terms of the realized Gram-weighted norm of the states $z_t^\top V_t^{-1} z_t$ in Lemma \ref{lem:ell2} via Bernstein's inequality.
This ultimately results in a bound over the total over all time steps via the elliptic potential lemma (Lemma \ref{lem:ellipt}).
Overall, this analysis is inspired by \cite{cassel2022efficient}, where they show that the Gram-weighted truncated state can be bound via the elliptic potential lemma using the fact that the disturbance is positive definite.

\begin{lemma}
    \label{lem:ell1}
    Assume the conditions of Lemma \ref{lem:main_cov}, and that $\lambda \geq R_z^2$.
    Then, if under $E$, it holds for all $t \in [\tau_1,T]$,
    \begin{align*}
        \ip{\bar{V}_{k_t-1}^{-1}, \Sigma_t} & \leq 12 \max\left( \frac{\nu}{\sigma}, 1 \right) \Eb_{t-\rho} [z_t^\top V_t^{-1} z_t]
    \end{align*}
\end{lemma}
\begin{proof}
    First, we have that,
    \begin{subequations}
    \label{eq:ell}
    \begin{align}
        \ip{\bar{V}_{k_t - 1}^{-1}, \Sigma_t} & \leq
        12 \ip{V_t^{-1}, \Sigma_t} \label{eq:ell:aa}\\
        & = 12\, \tr(V_t^{-1/2} \Sigma_t V_t^{-1/2}) \\
        & = 12\, \tr\left(V_t^{-1/2} \begin{bmatrix} I\\ K_t \end{bmatrix} \Sigma_t^{xx} \begin{bmatrix} I\\ K_t \end{bmatrix}^\top V_t^{-1/2}\right) + 12\, \tr\left(V_t^{-1/2} \begin{bmatrix} 0 & 0 \\ 0 & U_t \end{bmatrix} V_t^{-1/2}\right) \label{eq:ell:a}\\
        & \leq 12\, \frac{\nu}{\sigma} \tr\left(V_t^{-1/2} \begin{bmatrix} I\\ K_t \end{bmatrix} S_{t|t-\rho}^{xx} \begin{bmatrix} I\\ K_t \end{bmatrix}^\top V_t^{-1/2}\right) + 12\, \tr\left(V_t^{-1/2} \begin{bmatrix} 0 & 0 \\ 0 & U_t \end{bmatrix} V_t^{-1/2}\right) \label{eq:ell:b}\\
        & \leq 12 \max\left( \frac{\nu}{\sigma}, 1 \right)  \tr\left(V_t^{-1/2} S_{t|t-\rho} V_t^{-1/2}\right) \label{eq:ell:c}\\
        & \leq 12 \max\left( \frac{\nu}{\sigma}, 1 \right)  \tr\left(V_t^{-1/2} \Eb_{t-\rho} [z_t^\top V_t^{-1} z_t] V_t^{-1/2}\right) \label{eq:ell:d}\\
        & \leq 12 \max\left( \frac{\nu}{\sigma}, 1 \right) \Eb_{t-\rho} [z_t^\top V_t^{-1} z_t]\label{eq:ell:e}
    \end{align}
    \end{subequations}
    where each step is justified in the following:
    \begin{itemize}
        \item[\eqref{eq:ell:aa}] First, note that under the assumptions of the lemma and event $E$, it holds that $\|z_s \| \leq R_z$ for all $t$ and $\lambda \geq \max(R_z^2, \frac{R_z^2 \rho}{\log(2)})$.
        Therefore, the conditions of Lemma \ref{lem:rho_back} and \ref{lem:phase_det} hold.
        It follows that,
        $$
            V_t \preceq 2 V_{t - \rho} \preceq 2 \frac{\det(V_{t - \rho})}{\det(\bar{V}_{k_t})} \bar{V}_{k_t} \preceq 4 \bar{V}_{k_t} \preceq 12 \bar{V}_{k_t-1},
        $$
        where the first $\preceq$ uses Lemma \ref{lem:rho_back}, the second $\preceq$ uses Lemma \ref{lem:det_lown}, the third $\preceq$ uses the phase transition condition, and the fourth $\preceq$ uses Lemma \ref{lem:phase_det}.
        Then, because $\Sigma_t \succeq 0$, it holds that $\ip{\bar{V}_{k_t - 1}^{-1}, \Sigma_t} \leq 12 \ip{V_t^{-1}, \Sigma_t}$.

        \item[\eqref{eq:ell:a}] Note that the approximate covariance can be written as,
        $$
            \Sigma_t = \begin{bmatrix} I\\ K_t \end{bmatrix} \Sigma_t^{xx} \begin{bmatrix} I\\ K_t \end{bmatrix}^\top + \begin{bmatrix} 0 & 0 \\ 0 & U_t \end{bmatrix}
        $$
        \item[\eqref{eq:ell:b}] From the assumptions that $\| \Sigma_t^{xx} \| \leq \nu$ and $W \succeq \sigma I$,
        $$
            \Sigma_t^{xx} \preceq \nu I = \frac{\nu}{\sigma} \sigma I \preceq \frac{\nu}{\sigma} W \preceq \frac{\nu}{\sigma}( \Theta S_{t-1|t-\rho} \Theta^\top + W) = \frac{\nu}{\sigma} S_{t|t-\rho}^{xx}
        $$
        \item[\eqref{eq:ell:c}] Note that the conditional covariance can be written as,
        $$
            S_{t|t-\rho} = \begin{bmatrix} I\\ K_t \end{bmatrix} S_{t|t-\rho}^{xx} \begin{bmatrix} I\\ K_t \end{bmatrix}^\top + \begin{bmatrix} 0 & 0 \\ 0 & U_t \end{bmatrix}.
        $$
        \item[\eqref{eq:ell:d}] Using the definition of covariance,
        \begin{align*}
            S_{t|t-\rho} = \Eb_{t-\rho}[z_t z_t^\top] - m_{t|t-\rho} m_{t|t-\rho}^\top \preceq \Eb_{t-\rho}[z_t z_t^\top]
        \end{align*}
        \item[\eqref{eq:ell:e}] It holds that,
        $$
            \tr(V_t^{-1/2} \Eb_{t-\rho}[z_t z_t^\top] V_t^{-1/2}) = \Eb_{t-\rho} \tr(V_t^{-1/2} z_t z_t^\top V_t^{-1/2}) = \Eb_{t-\rho} z_t^\top V_t^{-1} z_t
        $$
    \end{itemize}
\end{proof}

Next, we bound the conditional mean in terms of the realized quantity via Bernstein's inequality.

\begin{lemma}
    \label{lem:ell2}
    If $\Pb(E) \geq 1- \omega$, it holds with probability at least $1 - 2\omega$ that,
    $$
        \sum_{t=\tau_1}^T \Eb_{t-\rho} [z_t^\top V_t^{-1} z_t] \leq 2 \sum_{t=\tau_1}^{T} z_t^\top V_t^{-1} z_t +  \lambda^{-1} R_z \rho \log(2 T / \omega )
    $$
\end{lemma}
\begin{proof}
    We use the notation $Z_t = z_t^\top V_t^{-1} z_t$.
    Then,
    \begin{equation}
        \label{eq:z_cond_bound}
        \sum_{t=\tau_1}^{T} \Eb_{t-\rho} [Z_t] = \sum_{t=\tau_1}^{T} \Eb_{t-\rho} \left[Z_t \Ib_{\{Z_t \leq R_z^2/\lambda\}} \right] + \sum_{t=\tau_1}^{T} \Eb_{t-\rho} \left[Z_t \Ib_{\{Z_t > R_z^2/\lambda\}} \right]
    \end{equation}
    We start by bounding the first term in \eqref{eq:z_cond_bound}.
    To do so, we partition the time steps $[T]$ in to $\rho$ sets that start at $i \in [\tau_1, \tau_1+\rho]$ and contain every $\rho$th time step.
    Consider the following definitions: $Y_t = Z_t \Ib_{\{Z_t \leq R_z^2/\lambda\}}$, $\bar{Y}_{i,j} = Y_{i + \rho j}$, $\bar{\Fc}_{i,j} = \Fc_{i + \rho j}$, $J_i = \lfloor (T - i)/\rho \rfloor$.
    Then, it holds that, for each $i$, $(\bar{Y}_{i,j})_{j \in [0,J_i]}$ is adapted to $(\bar{\Fc}_{i,j})_{j \in [0,J_i]}$.
    Furthermore, $\bar{Y}_{i + \rho j} \in [0,\lambda^{-1} R_z^2]$ almost surely.
    Then, we apply Bernstein's inequality (see \cite{rosenberg2020near} Lemma D.4) to get that, with probability at least $1 - \omega/\rho$,
    $$  
        \sum_{j = 0}^{J_i} \Eb[\bar{Y}_{i,j} | \bar{\Fc}_{i,j-1}] \leq 2 \sum_{j=0}^{J_i} \bar{Y}_{i,j} + 4 \lambda^{-1} R_z^2 \log(4 \rho J_i / \omega )
    $$
    Then, taking the union bound over $i$, it holds with probability at least $1 - \omega$,
    \begin{equation}
        \sum_{t=\tau_1}^{T} \Eb_{t-\rho} [Z_t \Ib_{\{Z_t \leq R_z^2/\lambda\}}] = \sum_{i = \tau_1}^{\tau_1+\rho} \sum_{j = 0}^{J_i} \Eb[\bar{Y}_{i,j} | \bar{\Fc}_{i,j-1}] \leq 2 \sum_{t=\tau_1}^{T} z_t^\top V_t^{-1} z_t + 4 \lambda^{-1} R_z^2 \rho \log(2 T / \omega ) \label{eq:bern}
    \end{equation}

    We then bound the second term in \eqref{eq:z_cond_bound}.
    First, note that under $E$, it holds that $z_t^\top V_t^{-1} z_t \leq R_z^2/\lambda$ and therefore $Z_t \Ib_{\{Z_t > R_z^2/\lambda\}} = 0$, and thus,
    \begin{equation}
        \label{eq:two_bound}
        \sum_{t=\tau_1}^{T} \Eb_{t-\rho} \left[Z_t \Ib_{\{Z_t > R_z^2/\lambda\}} \right] = 0 
    \end{equation}
\end{proof}

Next, we give the well-known elliptic potential lemma.

\begin{lemma}[Lemma 11 in \cite{abbasi2011improved}]
    \label{lem:ellipt}
    If $\| z_t \| \leq R_z$ for all $t \in [T]$, and $\lambda \geq \max(1, R_z^2)$ then,
    $$
        \sum_{t=1}^T z_t^\top V_t^{-1} z_t \leq 2 d \log\left( 1 + \frac{T}{\lambda d} \right)
    $$
\end{lemma}

Finally, putting all of these together gives the proof of Lemma \ref{lem:ell_pot}.

\begin{proof}[Proof of Lemma \ref{lem:ell_pot}]
    Combining Lemmas \ref{lem:ell1}, \ref{lem:ell2} and \ref{lem:ellipt} gives the claim.
\end{proof}

\subsection{Proof of Lemma \ref{lem:phase_upd}}
\label{sec:phase_upd}

In this section, we prove Lemma \ref{lem:phase_upd}, which gives important properties of the phase update scheme.
All of these properties are based on the following lemma, which bounds the difference in log-det of Gram matries at different time steps.

\begin{lemma}
    \label{lem:logdet}
    For any $t,t' \in [T]$ such that $t \geq t'$, it holds that,
    $$
        \log \left( \frac{\det(V_t)}{\det(V_{t'})} \right) \leq \frac{1}{\lambda}  \sum_{s=t' + 1}^{t} \| z_s \|^2
    $$
\end{lemma}
\begin{proof}
    First, note that since $t \geq t'$, it holds that $V_t \succeq V_{t'} \succ 0$ and therefore, $V_{t'}^{-1/2} V_t V_{t'}^{-1/2} \succeq 0$.
    Thus,
    \begin{equation}
    \label{eq:v_det}
    \begin{split}
        \frac{\det(V_t)}{\det(V_{t'})} & = \det\left( V_{t'}^{-1/2} V_t V_{t'}^{-1/2} \right)\\
        & \leq \left( \frac{1}{d} \tr \left( V_{t'}^{-1/2} V_t V_{t'}^{-1/2} \right) \right)^d,
    \end{split}
    \end{equation}
    where the last line uses the AM-GM inequality, given that $V_{t'}^{-1/2} V_t V_{t'}^{-1/2} \succeq 0$.
    Then, let,
    $$
        M := V_t - V_{t'} = \sum_{s=t' + 1}^{t } z_s z_s^\top \succeq 0.
    $$
    It follows that,
    \begin{equation}
        \label{eq:v_tr}
    \begin{split}
        \tr\left(V_{t'}^{-1/2} V_t V_{t'}^{-1/2}\right) & = \tr\left(V_{t'}^{-1/2} \left( M  + V_{t'}\right) V_{t'}^{-1/2}\right)\\
        & = \tr\left(V_{t'}^{-1/2} M V_{t'}^{-1/2} + I\right)\\
        & = \tr\left(M^{1/2} V_{t'}^{-1}  M^{1/2}\right) + d \\
        & \leq \lambda^{-1} \tr\left(M\right) + d \\
        & = \lambda^{-1} \sum_{s=t' + 1}^{t} \tr\left(z_s z_s^\top\right) + d \\
        & = \lambda^{-1} \sum_{s=t' + 1}^{t} \| z_s \|^2 + d,
    \end{split}    
    \end{equation}
    where the inequality uses the fact that $V_{t'} \succeq \lambda I$.
    Therefore, combining \eqref{eq:v_det} and \eqref{eq:v_tr} gives,
    $$
        \frac{\det(V_t)}{\det(V_{t'})} \leq \left( \frac{1}{\lambda d} \sum_{s=t' + 1}^{t} \| z_s \|^2 + 1 \right)^d
    $$
    Taking the log of both sides and using the inequality $\log(x+1) \leq x$, yields,
    $$
        \log \left( \frac{\det(V_t)}{\det(V_{t'})} \right) \leq d \log\left( \frac{1}{\lambda d} \sum_{s=t' + 1}^{t} \| z_s \|^2 + 1 \right) \leq \frac{1}{\lambda}  \sum_{s=t' + 1}^{t} \| z_s \|^2
    $$
\end{proof}

We also give a well-known lemma regarding the Loewner order of p.d. matrices.

\begin{lemma}[Lemma 27 in \cite{cohen2019learning}]
    \label{lem:det_lown}
    For matrices $N \succeq M \succ 0$, it holds that $N \preceq \frac{\det(N)}{\det(M)} M$.
\end{lemma}

Next, we show that the duration of each phase is bounded below proportional to $\lambda$.

\begin{lemma}
    \label{lem:phase_dur}
    Suppose that $\| z_t \| \leq R_z$ for all $t \in [\tau_k - \rho, \tau_{k+1} - 1 - \rho]$.
    For all $k$, the duration of phase $k$ satisfies,
    $$
        \tau_{k+1} - \tau_k \geq \frac{\log(2)}{R_z^2} \lambda
    $$
\end{lemma}
\begin{proof}
    Applying the phase update criteria and Lemma \ref{lem:logdet} gives,
    $$
        \log(2) \leq  \log \left( \frac{\det(\bar{V}_{k+1})}{\det(\bar{V}_{k})} \right) = \log \left( \frac{\det(V_{\tau_{k+1} - \rho})}{\det(V_{\tau_{k} - \rho})} \right) \leq \frac{1}{\lambda}  \sum_{s=\tau_k - \rho}^{\tau_{k+1} - 1 - \rho} \| z_s \|^2 \leq \frac{R_z^2}{\lambda} (\tau_{k+1} - \tau_k)
    $$
    Rearranging gives the claim.
\end{proof}

Next, we can bound the Gram matrix at a current time step in terms of the delayed Gram matrix.

\begin{lemma}
    \label{lem:rho_back}
    Given $t \geq \rho + 1$, suppose that $\| z_s \| \leq R_z$ for all $s \in [t-\rho,t]$, and $\lambda \geq \frac{R_z^2 \rho}{\log(2)}$.
    Then, for all $t \geq \rho + 1$,
    $$
        V_t \preceq 2 V_{t - \rho}
    $$
\end{lemma}
\begin{proof}
    Applying Lemma \ref{lem:logdet} and the condition on $\lambda$ gives that,
    $$
        \log \left( \frac{\det(V_t)}{\det(V_{t - \rho})} \right) \leq \frac{R_z^2}{\lambda} \rho \leq \log(2).
    $$
    Then, taking the $exp$ of both sides,
    $$
        \frac{\det(V_t)}{\det(V_{t - \rho})} \leq 2.
    $$
    Next, applying Lemma \ref{lem:det_lown},
    $$
        V_t \preceq \frac{\det(V_t)}{\det(V_{t - \rho})} V_{t - \rho} \preceq 2 V_{t - \rho}.
    $$
\end{proof}

Next, we can relate the Gram matrices between sequential phases.

\begin{lemma}
    \label{lem:phase_det}
    If $\| z_t \| \leq R_z$ for all $t \in [T]$, and $\lambda \geq \max(R_z^2, \frac{R_z^2 \rho}{\log(2)})$, then it holds for $k \in [1,N]$ that,
    $$
        \bar{V}_k \preceq 3 \bar{V}_{k-1}
    $$
\end{lemma}
\begin{proof}
    For the case of $k = 1$, it holds that,
    $$
        \bar{V}_1 = V_{\tau_1 - \rho} \preceq 2 V_{\tau_1 - 2 \rho} = 2 V_{\tau_0 - \rho} = 2 \bar{V}_0 \preceq 3 \bar{V}_0,
    $$
    where we use Lemma \ref{lem:rho_back}, and the fact that $\tau_1 = \tau_0 + \rho$.

    For $k \geq 2$, it holds that,
    \begin{align*}
        \bar{V}_k & = V_{\tau_k - \rho}\\
        & = V_{\tau_k - \rho - 1} + z_{\tau_k - \rho} z_{\tau_k - \rho}^\top \\
        & \preceq \frac{\det(V_{\tau_k - \rho - 1})}{\det(\bar{V}_{k-1})} \bar{V}_{k-1} + z_{\tau_k - \rho} z_{\tau_k - \rho}^\top \tag{a} \label{eq:lown_last:a}\\
        & \preceq 2 \bar{V}_{k-1} + z_{\tau_k - \rho} z_{\tau_k - \rho}^\top \tag{b} \label{eq:lown_last:b} \\
        & \preceq 2 \bar{V}_{k-1} + R_z^2 I \tag{c} \label{eq:lown_last:c} \\
        & \preceq 2 \bar{V}_{k-1} + \lambda I \tag{d} \label{eq:lown_last:d} \\
        & \preceq 3 \bar{V}_{k-1},
    \end{align*}
    where \eqref{eq:lown_last:a} uses Lemma \ref{lem:det_lown}, \eqref{eq:lown_last:b} uses the fact that $\tau_k - 1$ is in phase $k-1$ and $k-1 \geq 1$ which implies that $\det(V_{\tau_k - \rho - 1}) \leq 2 \det(\bar{V}_{k-1})$, \eqref{eq:lown_last:c} uses the assumed bound $\| z_{\tau_k - \rho} \| \leq R_z$, \eqref{eq:lown_last:d} uses the assumption that $\lambda  \geq R_z^2$. 
\end{proof}

Then, we give the well-known bound on the number of phases.

\begin{lemma}
\label{lem:phase_bound}
    If $\| z_t \| \leq R_z$ for all $t \in [T]$, it holds that $N \leq 1 + 2 d \log \left( 1  + \frac{R_z^2 T}{d \lambda} \right)$.
\end{lemma}
\begin{proof}
    From the phase update criteria, it holds that $\det(\bar{V}_{k+1}) \geq 2 \det(\bar{V}_k)$ for all $k \in [1,N-1]$.
    Therefore,
    $$
        \det(V_T) \geq \det(\bar{V}_N) \geq 2^{N-1} \det(\bar{V}_1) \geq 2^{N-1} \det(\lambda I) \quad \implies \quad \log\left( \frac{\det(V_T)}{\det(\lambda I)} \right) \geq (N - 1) \log(2)
    $$
    Also, using the AM-GM inequality,
    \begin{align*}
        \frac{\det(V_T)}{\det(\lambda I)} & = \lambda^{-d} \det(V_T)\\
        & \leq \lambda^{-d} \left( \frac{1}{d} \tr(V_T) \right)^d \\
        & = \left( 1  + \frac{1}{d \lambda} \sum_{s=1}^T \| z_s \|^2 \right)^d\\
        & \leq \left( 1  + \frac{R_z^2 T}{d \lambda} \right)^d
    \end{align*}
    Therefore,
    $$
        N \leq 1 + \frac{d}{\log(2)} \log \left( 1  + \frac{R_z^2 T}{d \lambda} \right) \leq 1 + 2 d \log \left( 1  + \frac{R_z^2 T}{d \lambda} \right),
    $$
    where we use that $1/\log(2) \leq 2$.
\end{proof}

Finally, we can put everything together to get the proof of Lemma \ref{lem:phase_upd}.

\begin{proof}[Proof of Lemma \ref{lem:phase_upd}]
    We prove each of the list items in the lemma.
    Item 1 holds since the phase update criteria ensures that $\det(V_{t-\rho}) \leq 2 \det(\bar{V_{k_t}})$ and $V_{t-\rho} \succeq \bar{V}_{k_t}$, and therefore Lemma \ref{lem:det_lown} gives the bound.
    Item 2 is Lemma \ref{lem:phase_det}.
    Item 3 is Lemma \ref{lem:rho_back}.
    Item 4 is Lemma \ref{lem:phase_bound}.
    Item 5 is Lemma \ref{lem:phase_dur}.
\end{proof}

\subsection{Proof of Lemma \ref{lem:safe_margin}}

\label{sec:safe_margin_proof}

In this section, we prove Lemma \ref{lem:safe_margin}, which ensures that the estimate of the zero policy covariance $\Sigma_k^\safe$ strictly satisfies the linear matrix constraint.
To do so, we first bound the estimation error of the covariance of the zero policy (Lemma \ref{lem:safe_dist}) and then show that the true covariance of the zero policy strictly satisfies the constraint (Lemma \ref{lem:safe_dist}).
Provided that the estimation error is sufficiently small, this ensures that the estimate of zero policy covariance strictly satisfies the constraint.

\begin{lemma}
    \label{lem:safe_dist}
    Assume that:
    \begin{enumerate}
        \item event $F_k$ holds
        \item $A_\star$ is $(\kappa, \gamma)$-strongly stable
        \item $\nu \geq \frac{2 \kappa^2 \tr(W)}{\gamma}$
    \end{enumerate}
    Let $\Sigma_{\star,xx}^{\safe}$ be the true covariance of the state under zero input, i.e. $\Sigma_{\star,xx}^{\safe} = A_\star \Sigma_{\star,xx}^{\safe} A_\star^\top + W$.
    Then, it holds that,
    $$
        \| \Sigma_{\star,xx}^{\safe} - \Sigma_{k,xx}^{\safe} \| \leq \frac{\kappa^2 \eta \nu}{\gamma \lambda}
    $$
\end{lemma}
\begin{proof}
    Note that,
    \begin{align*}
        \Sigma_{\star,xx}^{\safe} - \Sigma_{k,xx}^{\safe} & = A_\star \Sigma_{\star,xx}^{\safe} A_\star^\top - \hat{A}_k \Sigma_{k,xx}^{\safe} \hat{A}_k^\top\\
        & = A_\star (\Sigma_{\star,xx}^{\safe} - \Sigma_{k,xx}^{\safe}) A_\star^\top + A_\star \Sigma_{k,xx}^{\safe} A_\star^\top - \hat{A}_k \Sigma_{k,xx}^{\safe} \hat{A}_k^\top,
    \end{align*}
    Therefore, applying this $n$ times and then taking $n \to \infty$,
    $$
        \Sigma_{\star,xx}^{\safe} - \Sigma_{k,xx}^{\safe} = \sum_{s=0}^{\infty} A_\star^s \left( A_\star \Sigma_{k,xx}^{\safe} A_\star^\top - \hat{A}_k \Sigma_{k,xx}^{\safe} \hat{A}_k^\top \right) (A_\star^s)^\top,
    $$
    where we use the fact that $A_\star$ is stable.
    Therefore, we have that,
    \begin{align*}
        \| \Sigma_{\star,xx}^{\safe} - \Sigma_{k,xx}^{\safe} \| & \leq \| A_\star \Sigma_{k,xx}^{\safe} A_\star^\top - \hat{A}_k \Sigma_{k,xx}^{\safe} \hat{A}_k^\top \| \sum_{s=0}^{\infty} \| A_\star^s \|^2\\
        & \leq \kappa^2 \gamma^{-1} \eta \ip{\bar{V}_k^{-1}, \Sigma_{k}^{\safe}} \\
        & \leq \kappa^2 \gamma^{-1} \eta \lambda^{-1} \nu
    \end{align*}
    where the second line uses event $F_k$ and the fact that $A_\star$ is strongly-stable, and the last line uses the bound on $\Sigma_k^{\safe}$ from Lemma \ref{lem:nu}.
\end{proof}

Next, we provide conditions under which the true covariance of the zero policy strictly satisfies the constraint.

\begin{lemma}
    \label{lem:sig_safe}
    Assume that:
    \begin{enumerate}
        \item $\| \alpha_j \| \leq D$,
        \item $\delta < 0.5$
        \item $\| x_1 \| \leq R_1$
        \item $\tr(W) \leq \bar{w}$
        \item the policy $K = \bzero$ is $(\kappa,\gamma)$-strongly stable and ensures that $\Pb(\alpha_j^\top z_t^{K,\bw} \leq \beta - \epsilon) \geq 1 - \delta$ for all $j \in [J], t \in [T]$,
        \item $\rho \geq 2 \tilde{\gamma}^{-1} \log \left( \frac{\max(\beta,\epsilon/4)}{(1+\tilde{\kappa}) \tilde{\kappa} R_z D} \right)$
        \item $T \geq \gamma^{-1} \log \left( \frac{\max(\beta,\epsilon/4)}{(1+\kappa) \kappa R_1 D} \right)$
        \item $\omega \leq \delta/2$
        \item $\xi = \left( \frac{\beta - D (1 + \tilde{\kappa}) \tilde{\kappa} \left(1 - \tilde{\gamma}/2\right)^{\rho} R_z}{\Phi^{-1}(1 - \delta + \omega)} \right)^2 - D^2 C_1$
        \item $\epsilon_2 > 0$, where,
        \begin{align*}
            \epsilon_2 & := \frac{\epsilon^2}{4 \Phi^{-1}(1 - \delta)^2}  - D^2 (1 + \kappa)^2 \kappa^2 \gamma^{-1} (1 - \gamma)^{2 T} - \frac{2 \Phi^{-1}(1-\delta/2)  \beta^2 \omega}{\phi(\Phi^{-1}(1-\delta/2)) \Phi^{-1}(1-\delta)^4} - D^2 C_1
        \end{align*}
    \end{enumerate}
    Then, it holds that,
    $$
        \ip{\balpha_j, \Sigma_\star^\safe} \leq \xi - \epsilon_2
    $$
\end{lemma}
\begin{proof}
    We use the notation $S_T = \begin{bmatrix} I \\ 0 \end{bmatrix} S_{T,xx} \begin{bmatrix} I \\ 0 \end{bmatrix}^\top$ with $S_{xx} = \sum_{s=0}^{T-1} A_\star^s W (A_\star^s)^\top$, and $m_t = \begin{bmatrix} I \\ 0 \end{bmatrix} A_\star^{T-1} x_1 $.
    Then,
    \begin{align*}
        \ip{\balpha_j, \Sigma^K} & \leq \ip{\balpha_j, S_T^K} + D^2 \| \Sigma_\star^\safe - S_T^K \|\\
        & \leq \left( \frac{\beta + D (1 + \kappa) \kappa \left(1 - \gamma\right)^{T} R_1 - \epsilon}{ \Phi^{-1}(1 - \delta)} \right)^2 + D^2 (1 + \kappa)^2 \kappa^2 \gamma^{-1} (1 - \gamma)^{2 T} \tag{a} \label{eq:sig_eps:a}\\
        & \leq \left( \frac{\beta - D (1 + \tilde{\kappa}) \tilde{\kappa} \left(1 - \tilde{\gamma}/2\right)^{\rho} R_z}{\Phi^{-1}(1 - \delta)} \right)^2 - \frac{\epsilon^2}{4 \Phi^{-1}(1 - \delta)^2} + D^2 (1 + \kappa)^2 \kappa^2 \gamma^{-1} (1 - \gamma)^{2 T} \tag{b} \label{eq:sig_eps:b}\\
        & \leq \left( \frac{\beta - D (1 + \tilde{\kappa}) \tilde{\kappa} \left(1 - \tilde{\gamma}/2\right)^{\rho} R_z}{\Phi^{-1}(1 - \delta + \omega)} \right)^2 \tag{c} \label{eq:sig_eps:c}\\
        & \qquad + \frac{2 \Phi^{-1}(1-\delta/2)}{\phi(\Phi^{-1}(1-\delta/2)) \Phi^{-1}(1-\delta)^4} \left( \beta - D (1 + \tilde{\kappa}) \tilde{\kappa} \left(1 - \tilde{\gamma}/2\right)^{\rho} R_z \right)^2 \omega \\
        & \qquad - \frac{\epsilon^2}{4 \Phi^{-1}(1 - \delta)^2} + D^2 (1 + \kappa)^2 \kappa^2 \gamma^{-1} (1 - \gamma)^{2 T}\\
        & \leq \underbrace{\left( \frac{\beta - D (1 + \tilde{\kappa}) \tilde{\kappa} \left(1 - \tilde{\gamma}/2\right)^{\rho} R_z}{\Phi^{-1}(1 - \delta + \omega)} \right)^2 - D^2 C_1}_{\xi} - \epsilon_2 
    \end{align*}
    where each line is described in the following:
    \begin{itemize}
        \item[\eqref{eq:sig_eps:a}] Lemma \ref{lem:quant_form} tells us that,
        \begin{equation}
            \label{eq:gauss_form2}
            \alpha_j^\top m_{T} + \sqrt{\ip{\balpha_j, S_{T}}} \Phi^{-1}(1 - \delta) \leq \beta - \epsilon \quad \forall j \in [J],
        \end{equation}
        where $S_T = \begin{bmatrix} I \\ 0 \end{bmatrix} S_{T,xx} \begin{bmatrix} I \\ 0 \end{bmatrix}^\top$ with $S_{T,xx} = \sum_{s=0}^{T-1} A_\star^s W (A_\star^s)^\top$, and $m_t = \begin{bmatrix} I \\ 0 \end{bmatrix} A_\star^{T-1} x_1 \begin{bmatrix} I \\ 0 \end{bmatrix}^\top$.
        First, since $1 - \delta \geq 0.5$ and $\balpha_j, S_T \succeq 0$, it must be that $\beta - \epsilon - \alpha_j^\top m_{T} \geq 0$.
        Therefore, we can equivalently write \eqref{eq:gauss_form2} as,
        $$
            \ip{\balpha_j, S_{T}} \leq \left( \frac{\beta - \alpha_j^\top m_{T}}{ \Phi^{-1}(1 - \delta)} \right)^2 \quad \forall j \in [J],
        $$
        Also, from strong-stability,
        $$
            \| \Sigma_\star^\safe - S_{T} \| = \left\| \begin{bmatrix} I \\ 0 \end{bmatrix} \sum_{s=T}^{\infty} A_\star^s W ((A_\star)^s)^\top \begin{bmatrix} I \\ 0 \end{bmatrix}^\top \right\| \leq (1 + \kappa)^2 \kappa^2 \bar{w} \gamma^{-1} (1 - \gamma)^{2 T},
        $$
        and, furthermore,
        $$
            \alpha_j^\top m_{T} \leq D (1 + \kappa) \kappa \left(1 - \gamma\right)^{T} R_1 \leq \beta,
        $$
        where the second inequality uses the assumed lower bound on $T$.
        \item[\eqref{eq:sig_eps:b}]     It holds that,
        \begin{equation}
            \label{eq:beta_var}
            \begin{split}
            & \left( \beta + D (1 + \kappa) \kappa \left(1 - \gamma\right)^{T} R_1 - \epsilon \right)^2\\
            &  = \left( \beta - D (1 + \tilde{\kappa}) \tilde{\kappa} \left(1 - \tilde{\gamma}/2\right)^{\rho} R_z +  D (1 + \tilde{\kappa}) \tilde{\kappa} \left(1 - \tilde{\gamma}/2\right)^{\rho} R_z + D (1 + \kappa) \kappa \left(1 - \gamma\right)^{T} R_1 - \epsilon \right)^2\\
            & \leq (\beta - D (1 + \tilde{\kappa}) \tilde{\kappa} \left(1 - \tilde{\gamma}/2\right)^{\rho} R_z)^2 - \left(\epsilon - D (1 + \tilde{\kappa}) \tilde{\kappa} \left(1 - \tilde{\gamma}/2\right)^{\rho} R_z - D (1 + \kappa) \kappa \left(1 - \gamma\right)^{T} R_1\right)^2\\
            & \leq (\beta - D (1 + \tilde{\kappa}) \tilde{\kappa} \left(1 - \tilde{\gamma}/2\right)^{\rho} R_z)^2 - \frac{1}{4} \epsilon^2,
        \end{split}
        \end{equation}
        where we use the fact that for $x \geq y \geq 0$, it holds that $(x - y)^2 = x^2 - 2xy + y^2 \leq x^2 - y^2$. We also use that $\beta \geq D (1 + \tilde{\kappa}) \tilde{\kappa} \left(1 - \tilde{\gamma}/2\right)^{\rho} R_z$ from our assumption on $\rho$ and our assumption that,
        \begin{align*}
            & \rho \geq 2 \tilde{\gamma}^{-1} \log \left( \frac{\max(\beta,\epsilon/4)}{(1+\tilde{\kappa}) \tilde{\kappa} R_z D} \right), T \geq \gamma^{-1} \log \left( \frac{\max(\beta,\epsilon/4)}{(1+\kappa) \kappa R_1 D} \right)\\
            & \implies \ D (1 + \tilde{\kappa}) \tilde{\kappa} \left(1 - \tilde{\gamma}/2\right)^{\rho} R_z + D (1 + \kappa) \kappa \left(1 - \gamma\right)^{T} R_1 \leq \frac{1}{2} \epsilon\\
            & \implies \ \epsilon - D (1 + \tilde{\kappa}) \tilde{\kappa} \left(1 - \tilde{\gamma}/2\right)^{\rho} R_z - D (1 + \kappa) \kappa \left(1 - \gamma\right)^{T} R_1 \geq \epsilon/2 > 0
        \end{align*}
        \item[\eqref{eq:sig_eps:c}] Uses Lemma \ref{lem:quant_sens} with $x = 1 - \delta$, $y = 1 - \delta + \omega$ and $\bar{y} = 1 - \delta/2$ given that $\omega \leq \delta/2$.
    \end{itemize}
\end{proof}

Then, with this, we give the proof of Lemma \ref{lem:safe_margin}.

\begin{proof}[Proof of Lemma \ref{lem:safe_margin}]
    First, note that the assumptions of Lemma \ref{lem:sig_safe} are satisfied since $\epsilon_2 \geq \epsilon_1 > 0$.
    Then, we apply Lemma \ref{lem:sig_safe}, to get that,
    $$
        \ip{\balpha_j, \Sigma_\star^\safe} \leq \xi - \epsilon_2.
    $$
    Then, under event $E$ and the assumption that $A_\star$ is strongly-stable,
    \begin{align*}
        \ip{\balpha_j, \Sigma_k^{\safe}} + \mu \ip{\bar{V}_{k-1}^{-1}, \Sigma_{k}^{\safe}} & = \ip{\balpha_j, \Sigma_{\star}^{\safe}} + \mu \ip{\bar{V}_{k-1}^{-1}, \Sigma_{k}^{\safe}} + \ip{\balpha_j, \Sigma_{k}^{\safe} - \Sigma_{\star}^{\safe}}\\
        & \leq \ip{\balpha_j, \Sigma_{\star}^{\safe}} + \mu \ip{\bar{V}_{k-1}^{-1}, \Sigma_{k}^{\safe}} + \| \Sigma_{k}^{\safe} - \Sigma_{\star}^{\safe} \| \tr(\balpha_j)\\
        & \leq \xi - \epsilon_2 + \mu \ip{\bar{V}_{k-1}^{-1}, \Sigma_{k}^{\safe}} + \| \Sigma_{k}^{\safe} - \Sigma_{\star}^{\safe} \| \tr(\balpha_j)\\
        & \leq \xi - \epsilon_2 + \lambda^{-1} \nu (\mu + \kappa^2 \gamma^{-1} \eta D^2),
    \end{align*}
    where the last line uses Lemma \ref{lem:safe_dist}, and the last line uses that $\tr(\Sigma_k^\safe) \leq 2 \kappa^2 \bar{w} \gamma^{-1} \leq \nu$ from Lemma \ref{lem:approx_stable}, and that $\tr(\Sigma_\star^\safe) \leq \kappa^2 \gamma^{-1} \tr(W) \leq \nu$ from Lemma 3.3 in \cite{cohen2018online}.
    We then take $\epsilon_1 = \epsilon_2 - \lambda^{-1} \nu (\mu + \kappa^2 \gamma^{-1} \eta D^2)$.
    Then, since $-\epsilon_1 \geq -\epsilon_2$, it holds that,
    $$
        \ip{\balpha_j, \Sigma_\star^\safe} \leq \xi - \epsilon_2 \leq \xi - \epsilon_1
    $$
\end{proof}

\subsection{Proof of Lemma \ref{lem:exploration}}

\label{sec:exploration_proof}

In this section, we give the proof of Lemma \ref{lem:exploration}, which establishes constraint during the initialization phase and bounds the estimation error of the system after the initialization.
First, we provide the constraint satisfiaction guarantees.
\begin{lemma}
\label{lem:init_const}
    Assume that the following hold:
    \begin{enumerate}
        \item $\| \alpha_j \| \leq D$ for all $j$
        \item $c = \gamma \epsilon D^{-1} \kappa^{-2} S^{-1}$
        \item $\| \Theta_\star \| \leq S$
        \item the policy $K = \bzero$ is $(\kappa,\gamma)$-strongly stable and ensures that $\Pb(\alpha_j^\top z_t^{K,\bw} \leq \beta - \epsilon) \geq 1 - \delta$ for all $j \in [J], t \in [T]$,
    \end{enumerate}
    Then, it holds that $\Pb(\alpha_j^\top z_t \leq \beta) \geq 1 - \delta$ for all $t \in [\tau_1-1]$.
\end{lemma}
\begin{proof}
    Consider the sequence $z_t^{\safe} = \begin{bmatrix} x_t^{\safe} \\ \bzero \end{bmatrix}$ where $x_{t+1}^{\safe} = A_\star x_t^{\safe} + w_t$ and $x_1^{\safe} = x_1$.
    Then,
    $$
        x_{t+1} - x_{t+1}^{\safe} = A_\star (x_t - x_t^\safe) + B_\star u_t = A_\star^t (x_1 - x_1^\safe) + \sum_{s=1}^t A_\star^{t-s} B_\star u_s.
    $$
    Therefore, for $t \in [\tau_1 - 1]$, due to strong stability and the fact that $\| u_s \| \leq c$,
    $$
        \| z_t - z_t^{\safe} \| = \| x_t - x_t^{\safe} \| \leq \kappa^2 \gamma^{-1} S c
    $$
    Then, it holds almost surely that,
    $$
        \alpha_j^\top z_t = \alpha_j^\top z_t^{\safe} + \alpha_j^\top (z_t - z_t^{\safe}) \leq \alpha_j^\top z_t^{\safe} + D \kappa^2 \gamma^{-1} S c = \alpha_j^\top z_t^{\safe} + \epsilon.
    $$
    Therefore, it holds that $\Pb(\alpha_j^\top z_t \leq \beta) \geq 1 - \delta$.
\end{proof}

Then, we give the estimation error bounds leveraging \cite{simchowitz2018learning}.

\begin{lemma}
\label{lem:init_exp}
    Assume that,
    \begin{enumerate}
        \item $W \succeq \sigma I$
        \item the policy $K = \bzero$ is $(\kappa,\gamma)$-strongly stable        \item $\| x_1 \| \leq R_1$
    \end{enumerate}
    If,
    $$
        \tau_0 \geq 1 + 134 d \log\left( \frac{34 R_1^2}{\omega \min(\sigma, c^2/m)} \right).
    $$
    Then, it holds that,
    $$
        \Pb\left( \| \hat{\Theta}_0 - \Theta_\star \|_F \leq 300 \bar{w} d \sqrt{\frac{n + d \log(306 \bar{\Gamma} \min(\sigma, c^2/m)^{-1} \omega^{-1})}{(\tau_0 - 1)\min(\sigma, c^2/m)}} \right) \geq 1 - \omega/3,
    $$
    where $\bar{\Gamma} = \left( \kappa^2 R_1 + \kappa^2 \gamma^{-1} (\sqrt{2 n \bar{w} \log(2 n T/\omega)} + S c) + c \right)^2$.
\end{lemma}
\begin{proof}
    We apply Theorem 2.4 of \cite{simchowitz2018learning}.
    We verify each of the requirements for this result in the following.
    Note that $z_t$ is $\Fc_{t}$-measurable.
    
    First, we show that the sequence $(z_t)_{t=1}^{\tau_1-1}$ is $(k,\Gamma_{sb},p)=(2,\min(\sigma,c^2/m)I,0.3)$-block martingale small ball (BMSB).
    To do so, we show that $\Pb(| v^\top z_t| \geq \sqrt{v^\top \Gamma_{sb} v} | \Fc_{t-2}) \geq p$ for all $v \in \Sb$.
    Note that,
    $$
        z_t = \begin{bmatrix}
            A x_{t-1} + B u_{t-1} + w_{t-1}\\
            u_{t}
        \end{bmatrix},
    $$
    We use the random variable $Y_t = v^\top z_t - v^\top \begin{bmatrix}
            A x_{t-1}\\ \bzero
    \end{bmatrix}$ and note that $Y_t | \Fc_{t-2}$ is Gaussian with mean $\Eb[Y_t | \Fc_{t-2}] = 0$ and variance,
    \begin{align*}
        \Eb[Y_t^2 | \Fc_{t-2}] & = v^\top \Eb \left[\left(z_t - \begin{bmatrix}
            A x_{t-1}\\ \bzero
        \end{bmatrix} \right) \left(z_t - \begin{bmatrix}
            A x_{t-1}\\ \bzero
        \end{bmatrix} \right)^\top | \Fc_{t-2} \right] v\\
        & = v^\top \Eb \left[ \begin{bmatrix}
            B u_{t-1} + w_{t-1}\\ u_t
        \end{bmatrix} \begin{bmatrix}
            B u_{t-1} + w_{t-1}\\ u_t
        \end{bmatrix} ^\top | \Fc_{t-2} \right] v\\
        & = v^\top \begin{bmatrix}
            \Eb [B u_{t-1} u_{t-1}^\top B^\top ] + \Eb [w_{t-1} w_{t-1}^\top] & 0 \\ 0 & \Eb[u_t u_t^\top]
        \end{bmatrix} v\\
        & \geq \min(\sigma, c^2/m),
    \end{align*}
    where the third line uses the fact that $w_{t-1}, u_{t-1}, u_t$ are independent of eachother and $\Fc_{t-2}$.
    Then, we note a property of Gaussian random variables (similar to eq (3.12) in \cite{simchowitz2018learning}), for all $t \in \Rb$,
    $$
        \Pb_{X \sim \Nc(0,\sigma')} (| X + t| \geq \sigma') \geq \Pb_{X \sim \Nc(0,\sigma')} (| X| \geq \sigma') = \Pb_{Y \sim \Nc(0,1)} (| Y| \geq 1) = 2 \Pb(Y \geq 1) \geq 0.3,
    $$
    where the last inequality can be verified with a Z-table.
    Therefore, 
    \begin{align*}
        \sum_{i=1}^2 \Pb(| v_t^\top z_t | \geq \sqrt{v^\top \Gamma_{sb} v} | \Fc_{t-i} ) & \geq  \Pb(| v_t^\top z_t | \geq \sqrt{v^\top \Gamma_{sb} v} | \Fc_{t-2} )\\
        & = \Pb(| v_t^\top z_t | \geq \sqrt{\min(\sigma, c^2/m)} | \Fc_{t-2})\\
        & \geq \Pb(| v_t^\top z_t | \geq \Eb[Y_t^2 | \Fc_{t-2}] | \Fc_{t-2})\\
        & \geq \Pb(| Y_t | \geq \Eb[Y_t^2 | \Fc_{t-2}] | \Fc_{t-2})\\
        & \geq 0.3.
    \end{align*}
    Therefore, the $(k,\Gamma_{sb},p)$-BMSB property holds.

    Then, we show that $\Pb(\sum_{t=1}^{\tau_0-1} z_t z_t^\top \preceq \bar{\Gamma} (\tau_0-1) I) \geq 1 - \omega$ for some $\bar{\Gamma}$.
    First, since $[w_t]_i$ is $\bar{w}$-subgaussian for all $i \in [n]$, applying the union bound over $i \in [n]$ and $t \in [T]$ yields,
    $$
        \Pb(\| w_t \| \leq \sqrt{2 n \bar{w} \log(2 n T/\omega)}, \forall t \in [T]) \geq 1 - \omega.
    $$
    Then, for $t \in [\tau_0-1]$, it holds 
    $$
        \| x_t \| = \| A_\star^{t-1} x_1 + \sum_{s=1}^{t-1} A_\star^{t-1-s} (B u_s + w_s) \| \leq \kappa^2 (1 - \gamma)^{t-1} R_1 + \kappa^2 \gamma^{-1} (\max_{t \in [\tau_0]} \| w_t \| + S c)
    $$
    Therefore, with probability at least $1 - \omega$,
    $$
        \left\| \sum_{s=1}^{\tau_0-1} z_t z_t^\top \right\| \leq \sum_{s=1}^{\tau_0-1} \| z_t \|^2 \leq (\tau_0 - 1) \underbrace{\left( \kappa^2 R_1 + \kappa^2 \gamma^{-1} (\sqrt{2 n \bar{w} \log(2 n T/\omega)} + S c) + c \right)^2}_{\bar{\Gamma}},
    $$
    which satisfies the required condition.

    Lastly, we note that $w_t$ is $\bar{w}$-subgaussian. 
    
    Therefore we can apply the guarantee from Lemma 2.4 in \cite{simchowitz2018learning}.
\end{proof}

Lemma \ref{lem:exploration} follows directly from Lemma \ref{lem:init_const} and Lemma \ref{lem:init_exp}.

\subsection{Proof of Lemma \ref{lem:slow_var}}
\label{sec:slow_var_proof}

\begin{proof}[Proof of Lemma \ref{lem:slow_var}]
    For $k \geq 1$, it holds that $\Sigma_{t} - \bar{\Sigma}_k = (1 - \zeta)(\Sigma_{t-1} - \bar{\Sigma}_k)$ for all $t \in [\tau_k, \tau_{k+1}-1]$, and for $k = 0$, it holds that $\Sigma_{t} = \bar{\Sigma}_k$ for all $t \in [\tau_k, \tau_{k+1}-1]$. Thus,
    \begin{alignat}{2}
        \label{eq:sig_var1}
        \Sigma_{t} - \bar{\Sigma}_{k_t} & = (1 - \zeta)^{t - \tau_k + 1} (\Sigma_{\tau_{k_t} - 1} - \bar{\Sigma}_{k_t} ) \quad && t \in [\tau_1, T]\\
        \Sigma_{t} - \bar{\Sigma}_{k_t} & = 0 && t \in [\tau_0, \tau_{1}-1] \notag
    \end{alignat}
    Also, If $E$ holds and $\nu \geq \frac{2 \kappa^2 \tr(W)}{\gamma}$, we can apply Lemma \ref{lem:nu} to get that $\tr(\bar{\Sigma}_k) \leq \nu$ for all $k \in [N]$ and $\tr(\Sigma_{s}) \leq \nu$ for all $s \in [T]$, and therefore, for all $t \geq \tau_0$,
    \begin{equation}
        \label{eq:sig_var2}
        \| \Sigma_{t} - \bar{\Sigma}_{k_t} \| = (1 - \zeta)^{t - \tau_{k_t} + 1} \| \Sigma_{\tau_{k_t} - 1} - \bar{\Sigma}_{k_t} \| \leq 2 \nu (1 - \zeta)^{t - \tau_{k_t} + 1}
    \end{equation}
    With these, we prove each of the list items in the lemma.
    
    \begin{itemize}
        \item[\textit{\#1:}] An immediate consequence of \eqref{eq:sig_var1} is that, for all $t \geq \tau_1$,
        \begin{equation}
            \label{eq:conv1}
            \Sigma_t \in \conv\{\Sigma_{\tau_{k_t} - 1}, \bar{\Sigma}_{k_t} \}.
        \end{equation}
        This implies that, for $k \geq 2$, $\Sigma_{\tau_{k} - 1} \in \conv\{\Sigma_{\tau_{k-1} - 1}, \bar{\Sigma}_{k-1}\}$, and therefore,
        $$
            \conv\{\Sigma_{\tau_{k} - 1}, \bar{\Sigma}_{k}\} \subseteq \conv( \conv\{\Sigma_{\tau_{k-1} - 1}, \bar{\Sigma}_{k-1}\}\cup \{ \bar{\Sigma}_{k}\})\subseteq \conv\{\Sigma_{\tau_{k-1} - 1}, \bar{\Sigma}_{k-1}, \bar{\Sigma}_{k}\}.
        $$
        Applying this recursively, and using that $\Sigma_{\tau_1-1} = \bar{\Sigma}_{0}$, it holds for $k \geq 2$ that,
        \begin{equation}
            \label{eq:conv2}
            \conv\{\Sigma_{\tau_{k} - 1}, \bar{\Sigma}_{k}\} \subseteq \conv\{\Sigma_{\tau_{1} - 1},\bar{\Sigma}_{1},..., \bar{\Sigma}_{k}\} = \conv\{\bar{\Sigma}_{0},\bar{\Sigma}_{1},..., \bar{\Sigma}_{k}\}
        \end{equation}
        Finally, combining \eqref{eq:conv1} and \eqref{eq:conv2} yields,
        $$
            \Sigma_t \in \conv\{\Sigma_{\tau_{k_t} - 1}, \bar{\Sigma}_{k_t} \} \subseteq \conv\{\bar{\Sigma}_{0},\bar{\Sigma}_{1},..., \bar{\Sigma}_{k_t}\}
        $$
        
        \item[\textit{\#2:}] From \eqref{eq:sig_var2},
        $$
            \sum_{t=\tau_0}^T \| \Sigma_{t} - \bar{\Sigma}_{k_t} \| = \sum_{k = 1}^N \sum_{t = \tau_k}^{\tau_{k+1}-1} \| \Sigma_{t} - \bar{\Sigma}_{k_t} \| \leq 2 \nu N \sum_{s=1}^{\infty} (1 - \zeta)^s \leq 2 \nu \zeta^{-1} (N + 1)
        $$
        \item[\textit{\#3:}] Using that $\tr(\bar{\Sigma}_{k_t} ) \leq \nu$ and $\tr(\Sigma_{t-1}) \leq \nu$ under $E$ and $\nu \geq \frac{\kappa^2 \tr(W)}{\gamma}$, it holds for $t \geq \tau_0$ that,
        $$
            \| \Sigma_{t} - \Sigma_{t-1} \| \leq \zeta \| \bar{\Sigma}_{k_t} - \Sigma_{t-1} \| \leq 2 \nu \zeta
        $$
        \item[\textit{\#4:}] A direct application of \eqref{eq:sig_var2}.
    \end{itemize}
\end{proof}

\subsection{Proof of Lemma \ref{lem:term1a}}
\label{sec:term1a}

\begin{proof}
    Let $Z_t = z_t z_t^\top$.
    Then, it holds that,
    \begin{equation}
    \label{eq:az_split}
    \begin{split}
        \sum_{t=\tau_1}^T \ip{\mathbf{Q}, Z_t - \Eb_{t-\rho}[Z_t] } & = \sum_{t=\tau_1}^T \ip{\mathbf{Q}, Z_t\Ib_{\{\| Z_t \| \leq  R_z^2\}} - \Eb_{t-\rho}[Z_t\Ib_{\{\| Z_t \| \leq  R_z^2\}}] }\\
        & + \sum_{t=\tau_1}^T \ip{\mathbf{Q}, Z_t\Ib_{\{\| Z_t \| >  R_z^2\}} - \Eb_{t-\rho}[Z_t\Ib_{\{\| Z_t \| >  R_z^2\}}] }
    \end{split}    
    \end{equation}
    We start by bounding the first term in \eqref{eq:az_split}.
    Let $Y_t = \ip{\mathbf{Q}, Z_t\Ib_{\{\| Z_t \| \leq  R_z^2\}} - \Eb_{t-\rho}[Z_t\Ib_{\{\| Z_t \| \leq  R_z^2\}}] }$.
    Notice that, $Y_t$ is $\Fc_t$-measurable and, furthermore $\Eb_{t-\rho} Y_t = 0$.
    Additionally, $| Y_t | \leq 2 R_Q R_z^2$ almost surely.

    Then, let $\bar{Y}_{i,j} = Y_{i + \rho j}$ for $i \in [\tau_1,\tau_1 + \rho]$ and $j \in [0,\lfloor (T - \tau_1 -  i)/\rho \rfloor - 1]$.
    Also, let $\bar{\Fc}_{j}^i = \Fc_{i + \rho j}$.
    Therefore, $\bar{Y}_{i,j}$ is a martingale difference sequence w.r.t. $j$ given that $\bar{Y}_{i,j}$ is $\bar{\Fc}_{j}^i$-measurable, and $\Eb[\bar{Y}_{i,j} | \bar{\Fc}_{j-1}^i] = 0$.
    It follows from the Azuma inequality, that with probability at least $1 - \omega$,
    $$
    \sum_{j=0}^{\lfloor (T - \tau_1 -  i)/\rho \rfloor - 1} \bar{Y}_{i,j} \leq 2 R_Q R_z^2   \sqrt{2 \lfloor (T - \tau_1 -  i)/\rho \rfloor \log(1/\omega)}
    $$
    Taking the union bound over all $i$, it holds with probability at least $1 - \omega$ that,
    \begin{equation}
    \label{eq:sum_yt}
    \begin{split}
        \sum_{t=\tau_1}^{T} Y_t & = \sum_{i=\tau_1}^{\tau_1 + \rho} \sum_{j=0}^{\lfloor (T - \tau_1 -  i)/\rho \rfloor - 1} \bar{Y}_{i,j}\\
        & \leq  2 R_Q R_z^2 \sum_{i=\tau_1}^{\tau_1 + \rho} \sqrt{2  \lfloor (T - \tau_1 -  i)/\rho \rfloor \log(\rho/\omega)}\\
        & \leq 2 R_Q R_z^2 \sqrt{\rho 2  T \log(\rho/\omega)}.
    \end{split}
    \end{equation}

    Then, we bound the second term in \eqref{eq:az_split}.
    Under $E$, it holds that $\|Z_t\| \leq R_z^2$ for all $t \in [T]$ and therefore,
    $$
        \sum_{t=\tau_1}^T \ip{\mathbf{Q}, Z_t\Ib_{\{\| Z_t \| >  R_z^2\}} - \Eb_{t-\rho}[Z_t\Ib_{\{\| Z_t \| >  R_z^2\}}] } = 0.
    $$

    Also under $E$, it follows from Lemma \ref{lem:main_cov} that,
    \begin{align*}
        \mathrm{Term\ I.A} & = \sum_{t=\tau_1}^{T} \left( \ell(x_t, u_t) - \ip{\mathbf{Q}, S_{t|t-\rho}} \right)\\
        & = \sum_{t=\tau_1}^T \ip{\mathbf{Q}, z_t z_t^\top - \Eb_{t-\rho}[z_t z_t^\top] } + \sum_{t=\tau_1}^T \ip{\mathbf{Q}, m_{t-\rho} m_{t-\rho}^\top } \\
        & \leq \sum_{t=\tau_1}^T \ip{\mathbf{Q}, z_t z_t^\top - \Eb_{t-\rho}[z_t z_t^\top] } + R_Q (1 + \tilde{\kappa})^2 \tilde{\kappa}^2 R_z^2 \exp(-\tilde{\gamma} \rho/2) T.
    \end{align*}
    Taking the union bound over $E$ and the event in \eqref{eq:sum_yt} gives the claim.
\end{proof}

\subsection{Proofs from Section \ref{sec:ana}}
\label{sec:ana_proofs}

In this section, we prove the lemmas in Section \ref{sec:ana}.

\begin{proof}[Proof Lemma \ref{lem:cov_approx1}]
    Under the conditions of Lemma \ref{lem:cov_approx1}, the sequence of policies is sequential-strong stability (in the sense of Definition \ref{def:seq_str_stab}) due to Lemma \ref{lem:strong_stab}.
    Sequential strong stability along with the other conditions of  Lemma \ref{lem:cov_approx1} ensure that the conditions of Lemma \ref{lem:cov_approx} and \ref{lem:mean} are satisfied.
    Thus, Lemma \ref{lem:cov_approx} and \ref{lem:mean} yield the desired bound on $m_{t|t-\rho}$ and $S_{t|t-\rho}$.
\end{proof}

\begin{proof}[Proof of Lemma \ref{lem:approx2}]
    Refer to the proof of Theorem \ref{thm:main} in Appendix \ref{sec:comp_proof}, where it is shown that event $E$ holds with high probability.
    Thus, the bound in Lemma \ref{lem:main_cov} holds with high probability.
    Applying Lemma \ref{lem:main_cov} and using the choice of algorithm parameters in Theorem \ref{thm:main_apx} gives the result.
\end{proof}

\begin{proof}[Proof of Lemma \ref{lem:ellip_pot}]
    Refer to the proof of Theorem \ref{thm:main} in Appendix \ref{sec:comp_proof}, where it is shown that event $G_1$ holds with high probability.
    Using the definition of $G_1$ in \eqref{eq:g1} with the specified choice of algorithm parameters gives the result.
\end{proof}

\begin{proof}[Proof of Lemma \ref{lem:const_bound}]
    The result follows immediately from Lemma \ref{lem:cov_const} by taking $\omega'=\delta - \omega$.
\end{proof}




\end{document}